\documentclass[10pt,a4paper,reqno]{amsart}
\usepackage{amsfonts,amsthm,latexsym,amsmath,amssymb,amscd,amsmath,epsf,morefloats}
\usepackage{placeins}
\usepackage{graphicx}
\usepackage{tikz}

\newtheorem{theo+}           {Theorem}
\newtheorem{prop+}           {Proposition}
\newtheorem{coro+}           {Corollary}
\newtheorem{lemm+}           {Lemma}
\newtheorem{conjecture}      {Conjecture}
\theoremstyle{definition}
\newtheorem{defi+}           {Definition}
\newtheorem{problem}         {Problem}

\theoremstyle{remark}
\newtheorem{rema+}           {Remark}

\newenvironment{theorem}{\begin{theo+}}{\end{theo+}}
\newenvironment{proposition}{\begin{prop+}}{\end{prop+}}
\newenvironment{corollary}{\begin{coro+}}{\end{coro+}}
\newenvironment{lemma}{\begin{lemm+}}{\end{lemm+}}
\newenvironment{remark}{\begin{rema+}}{\end{rema+}}
\newenvironment{definition}{\begin{defi+}}{\end{defi+}}

\newcommand{\al}{\alpha}
\newcommand{\bet}{\beta}

\newcommand {\CC}{\mathcal C}

\newcommand {\bCP} {\mathbb {CP}}

\newcommand {\Ga} {\Gamma}
\newcommand {\ga} {\gamma}
\newcommand {\eps} {\epsilon}

\newcommand \dq {\mathfrak d(z)}
\newcommand \dqn {\mathfrak d_n(z)}

\newcommand{\la}{\lambda}

\newcommand{\bC}{\mathbb C}
\newcommand{\bR}{\mathbb R}

\newcommand{\C}{\mathcal C}

\newcommand{\NN}{\mathbb N}

\newcommand{\be}{\begin{equation}}
\newcommand{\ee}{\end{equation}}
\newcommand{\bl}{\begin{lemma}}
\newcommand{\el}{\end{lemma}}
\newtheorem{lemmaA}{Lemma}
\newcommand{\blemmaA}{\begin{lemmaA}}
\newcommand{\elemmaA}{\end{lemmaA}}
\newcommand{\bprop}{\begin{proposition}}
\newcommand{\eprop}{\end{proposition}}

\newcommand{\bt}{\begin{theorem}}
\newcommand{\et}{\end{theorem}}
\newcommand{\bc}{\begin{corollary}}
\newcommand{\ec}{\end{corollary}}
\newcommand{\bcon}{\begin{conjecture}}
\newcommand{\econ}{\end{conjecture}}

\newcommand{\supp}{\operatorname{supp}}
\newcommand{\locint}{L^1_{loc}}

\def\newop#1{\expandafter\def\csname #1\endcsname{\mathop{\rm
#1}\nolimits}}

\newop{slc}
\newop{sign}
\newop{mdeg}

\begin{document}
          \numberwithin{equation}{section}

          \title[Root-counting measures and quadratic differentials]
          {Root-counting measures of Jacobi polynomials \\ and  topological types and  critical
          geodesics\\
          of  related quadratic differentials}


\author[B.~Shapiro]{Boris Shapiro}
\address{Department of Mathematics, Stockholm University, SE-106 91
Stockholm,
         Sweden}
\email{shapiro@math.su.se}

\author[A.~Solynin]{Alexander Solynin}
\address{Department of Mathematics and Statistics, Texas Tech University, Box 41042, Lubbock, TX 79409, USA}
\email{alex.solynin@math.ttu.edu}

\date{\today}

\begin{abstract}%
 Two main topics of this paper are asymptotic distributions of zeros of
Jacobi polynomials and topology of critical trajectories of
related quadratic differentials. First, we will discuss recent
developments and some new results concerning the limit of the
root-counting measures of these polynomials. In particular, we
will show that the support of the limit measure sits on the
critical trajectories of a quadratic differential of the form
$Q(z)\,dz^2=\frac{az^2+bz+c}{(z^2-1)^2}\,dz^2$. Then we will give
a complete classification, in terms of complex parameters $a$,
$b$, and $c$, of possible topological types of
critical geodesics for the quadratic differential of this type. %
\end{abstract}%

\dedicatory {To Mikael Passare, in memoriam}

\maketitle

\section{Introduction: From Jacobi polynomials to quadratic
differentials} \label{Section-1} %

Two main themes of this work are asymptotic behavior of zeros of
certain polynomials and topological properties of related
quadratic differentials.  The study of asymptotic root
distributions of hypergeometric, Jacobi, and Laguerre polynomials
with variable real  parameters, which grow linearly with  degree,
became a rather hot topic in recent publications, which attracted
attention of many authors \cite{CC}, \cite{DMO}, \cite{DrDu},
\cite{DrJo}, \cite{DuGu}, \cite{KuMF}, \cite{MFMGO}, \cite{MMO},
\cite{QW}.
In this paper, we survey some known results in this area and
present some new results keeping focus on Jacobi polynomials.

 Recall that the   Jacobi polynomial $P_n^{(\al,\bet)}(z)$ of degree $n$
 with complex parameters $\al,\bet$ is defined  by
  $$P_n^{(\al,\bet)}(z)=2^{-n}\sum_{k=0}^n\binom {n+\al}{n-k}\binom{n+\bet}{k}(z-1)^k(z+1)^{n-k},$$
 where  $\binom{\gamma}{k}=\frac{\gamma(\gamma-1)\dots (\gamma-k+1)}{k!}$ with a non-negative integer $k$
 and  an arbitrary complex number $\gamma$.
 Equivalently,  $P_n^{(\al,\bet)}(z)$ can be defined by the well-known Rodrigues
 formula: %
$$ %
P_n^{(\al,\bet)}(z)=\frac{1}{2^nn!}(z-1)^{-\al}(z+1)^{-\bet}
\left(\frac{d}{dz}\right)^n[(z-1)^{n+\al}(z+1)^{n+\bet}].
$$ %
 The following statement, which can be found, for instance, in
\cite[Proposition 2]{MFMGO}, gives an important characterization
of Jacobi polynomials as solutions of second order differential
equation.
\begin{proposition}\label{Proposition-1} %
For arbitrary fixed complex numbers $\al$ and $\bet$, the
differential equation %
$$ 
(1-z^2)y^{\prime\prime}+(\bet -\al-(\al+\bet+2)z)y'+\la y=0 %
$$ 
with a spectral parameter $\la$ has a non-trivial polynomial
solution of degree $n$ if and only if $\la=n(n+\al+\bet+1)$.  This
polynomial solution is unique (up to a constant factor) and
coincides with $P_n^{(\al,\bet)}(z)$. %
\end{proposition} %

Working with root distributions of polynomials, it is convenient
to use root-counting measures and their Cauchy transforms, which
are defined as follows. %

\begin{definition} %
For a polynomial $p(z)$ of  degree $n$ with
(not necessarily distinct) roots    $\xi_1,...,\xi_n$,  its {\it
root-counting measure} $\mu_p$ is defined as
$$\mu_p=\frac{1}{n}\sum_{i=1}^n\delta_{\xi_i},$$ %
  where
$\delta_\xi$ is the Dirac measure supported at $\xi$. %
\end{definition}

\begin{definition} %
 Given a finite complex-valued  Borel  measure $\mu$ compactly supported in $\bC,$
its {\it Cauchy transform} $\C_\mu$ is defined as
\begin{equation} \label{1.1}%
\C_\mu(z) =\int_{\bC} \frac{d\mu(\xi)}{z-\xi}. %
\end{equation} %
and  its {logarithmic potential} $u_\mu$ is defined as
$$u_\mu(z) =\int_{\bC} \log|{z-\xi}|{d\mu(\xi)}.$$
 \end{definition} %
We note that the integral in (\ref{1.1}) converges for all $z$,
for which the Newtonian potential $U_{|\mu|}(z)=\int_{\bC}
\frac{d|\mu|(\xi)}{|\xi-z|}$ of $\mu$ is finite, see e.g.
\cite[Ch. 2]{Ga}.

In case when $\mu=\mu_p$ is the root-counting measure of a polynomial $p(z)$,
we will write $\C_p$ instead of $\C_{\mu_p}$. It follows from
Definitions~1 and 2 that  the Cauchy transform $\C_p(z)$ of the
root-counting measure of a monic polynomial $p(z)$ of degree $n$
coincides
with the normalized logarithmic derivative of $p(z)$; i.e., %
 \begin{equation}\label{Cauchy}
\C_p(z)=\frac{p'(z)}{np(z)}=\int_{\bC} \frac{d\mu_p(\xi)}{z-\xi},
\end{equation}
and its logarithmic potential $u_p(z)$ is given by the formula:
\begin{equation}\label{log}
u_p(z)= \frac{1}{n}\log |p(z)|=\int_{\bC} \log|{z-\xi}|{d\mu_p(\xi)}.\end{equation}

Let  $\{p_n(z)\}$ be a sequence of  Jacobi polynomials
$p_n(z)=P_n^{(\al_n,\bet_n)}(z)$ and let $\{\mu_n\}$ be the
corresponding sequence  of their root-counting
 measures. The main question we are going to address in this paper is the following:

\begin{problem}
{ Assuming that the sequence $\{\mu_n\}$ weakly converges  to a measure $\mu$ compactly supported in
 $\bC$, what can be said about properties of the support of the measure $\mu$
 and about its Cauchy transform $\C_\mu$?}
\end{problem}

\medskip
Regarding the Cauchy transform $\C_\mu$, our main result in this
direction is the following theorem.

 \begin{theorem}\label{Theorem-1} Suppose that a sequence  $\{p_n(z)\}$ of Jacobi polynomials $p_n(z)=P_n^{(\al_n,\bet_n)}(z)$ satisfies conditions:

 \noindent
  {\rm(a)} the limits $A=\lim_{n\to\infty}\frac{\al_n}{n}$ and $B=\lim_{n\to\infty}\frac{\bet_n}{n}$
  exist, and  $1+A+B\neq 0$;

 \noindent
  {\rm(b)}
the sequence $\{\mu_n\}$ of the root-counting
measures  converges weakly to a probability measure
 $\mu$, which is compactly supported in $\bC$.

Then  the Cauchy transform $\C_\mu$ of the limit measure $\mu$
satisfies almost everywhere in $\bC$ the  quadratic
equation:
  \begin{equation}\label{1.2}
  (1-z^2)\C^2_\mu-((A+B)z+A-B)\C_\mu+A+B+1=0.
  \end{equation}
 \end{theorem}

The proof of Theorem~\ref{Theorem-1} given in Section~2 consists
of several steps. Our arguments in Section~2 are similar to the
arguments used in  a number of earlier papers on root asymptotics
of orthogonal polynomials. 

Equation (\ref{1.2}) of Theorem~1 implies that the support of the
limit measure $\mu$ has a remarkable structure described by
Theorem~2 below. And this is exactly the point where quadratic
differentials, which are the second main theme of this paper,
enter into the play.

\begin{theorem}\label{Theorem-2} %
In notation of Theorem~\ref{Theorem-1}, the support of $\mu$
consists of finitely many
trajectories of the quadratic differential 
 $$ 
  Q(z)\,dz^2=-\frac{(A+B+2)^2z^2+2(A^2-B^2)z+(A-B)^2-4(A+B+1)} {(z-1)^2(z+1)^2}\,
  dz^2 
$$ 
and their end points. %
\end{theorem} %

  Thus, to understand geometrical structure of the support of $\mu$ we have to
  study geometry of critical trajectories, or more generally critical geodesics  of the quadratic
  differential $Q(z)\,dz^2$ of Theorem~\ref{Theorem-1}.
  We will consider a slightly more general family of quadratic differentials
$Q(z;a,b,c)\,dz^2$ depending  on three complex parameters
$a,b,c\in \mathbb{C}$, $a\not=0$, 
where %
\begin{equation} \label{1.4} %
Q(z;a,b,c)\,dz^2=\frac{az^2+bz+c}{(z-1)^2(z+1)^2}\,dz^2. %
\end{equation} %

It is well-known that quadratic differentials appear in many areas
of mathematics and mathematical physics such as moduli spaces of
curves, univalent functions, asymptotic theory of linear ordinary
differential equations, spectral theory of Schr\"odinger
equations,  orthogonal polynomials, etc. Postponing necessary
definitions and basic properties of quadratic differentials till
Section~3, we recall here that any meromorphic quadratic
differential $Q(z)\,dz^2$ defines the so-called  {\it $Q$-metric}
and therefore it defines \emph{$Q$-geodesics} in appropriate
classes of curves.
 Motivated by the fact that the family of quadratic differentials (\ref{1.4})
 naturally appears in the study of the root
 asymptotics for sequences of Jacobi polynomials and
 is one of very few examples allowing  detailed and explicit investigation in terms of its
 coefficients,
 we will  consider the following two basic questions: %
\begin{enumerate} %
\item[1)] %
How many simple critical $Q$-geodesics may exist for a quadratic
differential $Q(z)\,dz^2$ of the form (\ref{1.4})? %

\item[2)] %
For given $a,b,c\in\mathbb{\mathbb{C}}$, $a\not=0$, describe
topology of all simple critical $Q$-geodesics.
\end{enumerate} %

A complete description of topological structure of trajectories of
quadratic differentials (\ref{1.4}) which, in particular, answers
questions 1) and 2), is given by  lengthy Theorem~5 stated in
Section~9.

\medskip

The rest of the paper consists of two parts and is structured as
follows. The first part, which is the area of expertise of the
first author, includes Sections~\ref{Section-2},~\ref{sq-roots},
and 5.  
Section~\ref{Section-2} contains the proof of Theorem~1 and
related results. The material presented in Section~\ref{sq-roots}
is mostly borrowed from a recent paper \cite{BoSh} of the first
author. It contains  some general results connecting signed
measures, whose Cauchy transforms satisfy quadratic equations, and
related quadratic differentials in $\bC$. In particular, these
results  imply Theorem~2 as a special case. In Section~5,
we formulate a number of general
conjectures about the type of convergence of root-counting
measures of polynomial solutions of a special class of linear
differential equations with polynomial coefficients, which
includes Riemann's differential equation.

Remaining sections constitute the second part, which is the area
of expertise of the second author. In Section~\ref{Section-4}, we
recall basic information about quadratic differentials, their
critical trajectories and geodesics. This information is needed
for presentation of our results in Sections~6--10. In Section~6,
we describe possible domain configurations for the quadratic
differentials (\ref{1.4}). Then, in Section~7, we describe
possible topological types of the structure of critical
trajectories of quadratic differentials  of the form (\ref{1.4}).
Finally  in Sections~8--10, we identify sets of parameters
corresponding to each topological type. The latter allows us to
answer some related questions.

We note here that our main proofs presented in Sections~6--10 are
geometrical based on general facts of the theory of quadratic
differentials. Thus, our methods can be easily adapted to study
trajectory structure of many quadratic differentials other then
quadratic differential~(\ref{1.4}).

Section~11 is our Figures Zoo, it contains many figures
illustrating our results presented in Sections 6--10.

\medskip
\noindent {\it Acknowledgements.} The authors want to acknowledge
the hospitality of the Mittag-Leffler Institute in Spring 2011
where this project was initiated. The first author is also
sincerely grateful to R.~B\o gvad, A.~Kuijlaars,
A.~Mart\'inez-Finkelshtein, and A.~Vasiliev for many useful
discussions.

\section{Proof of Theorem~1}
\label{Section-2}%

To settle Theorem~1 we will need several auxiliary
statements. Lemma~\ref{lm:basic} below  can be found as Theorem 7.6 of \cite {Ba} and apparently was originally proven by F.~Riesz.

\begin{lemma}\label{lm:basic} If a sequence $\{\mu_n\}$ of Borel probability  measures in $\bC$ weakly converges to a
probability measure $\mu$ with a compact support, then
the sequence $\{\C_{\mu_n}(z)\}$ of its Cauchy transforms
converges to $\C_{\mu}(z)$ in $L^1_{loc}$. Moreover there exists a subsequence of $\{\C_{\mu_n}(z)\}$ which converges to $\C_{\mu}(z)$  pointwise  almost everywhere.
\end{lemma}


The next result is recently obtained by the first author jointly with R.B\o gvad and D.~Khavinsion, see Theorem~1 of \cite{BRSh} and has an independent interest.

\begin{proposition}\label{lm:bound}  Let $\{p_m\}$  be any sequence of polynomials satisfying the following
conditions:

\noindent 1. $n_m:=\deg p_m  \to \infty$ as $m \to \infty$,

\noindent 2. almost all roots of all $p_m$ lie in a bounded convex
open $\Omega\subset \bC$ when $n\to \infty$. (More exactly,  if
$In_m$ denotes the number of roots of $p_m$ counted with
multiplicities which are located in $\Omega,$ then $\lim_{m\to
\infty}\frac{In_m}{n_m}=1$), then for any $\eps>0$,  $$\lim_{m\to
\infty}\frac{In'_m(\eps)}{n_m}=1,$$ where $In'_m(\eps)$ is the
number of roots of $p_m'$ counted with multiplicities which are
located  inside $\Omega(\eps)$, the latter set being the
$\eps$-neighborhood of $\Omega$ in $\bC$.

\end{proposition}



The next statement is a strengthening of Lemma~8 of \cite{BR}
based on  Proposition~\ref{lm:bound}.

\begin{lemma} \label{lm:add2}  Let $\{p_m\}$  be any sequence of polynomials satisfying the following
conditions:

\noindent 1. $n_m:=\deg p_m  \to \infty$ as $m \to \infty$,

\noindent 2. the sequence $\{\mu_m\}$  {\rm(}resp. $\{\mu'_m\}${\rm)} of the
root-counting measures  of $\{p_m\}$  {\rm(}resp. $\{p'_m\}${\rm)} weakly
converges to  compactly supported   measures $\mu$  {\rm(}resp
$\mu'${\rm)}.

Then  $u$ and $u'$   satisfy the inequality $u \ge u'$  with
equality on the unbounded component of    $\bC \setminus
supp(\mu)$. Here $u$ {\rm(}resp. $u'${\rm)} is the logarithmic potential of
the limiting measure $\mu$ {\rm(}resp. $\mu'${\rm)}.
\end{lemma}

\begin{proof}

Without loss of generality, we can assume   that all $p_{m}$ are
monic. Let $K$ be a compact convex set containing almost all the
zeros of the sequences $\{p_{m}\}$ and $\{p_m'\}$, i.e.,
$\lim_{m\to \infty} \frac{In_m(K)}{n_m}=\lim_{m\to \infty}
\frac{In^\prime_m(K)}{n_m}=1$. By \eqref{log} we  have
\begin{equation*}
     u(z) = \lim_{m \to \infty}\frac{1}{n_{m}}\log|p_{m}(z)|
\end{equation*}
and
\begin{equation*}
     u'(z) = \lim_{m\to \infty}
     \frac{1}{n_{m}-1}\log\left|\frac{p_{m}'(z)}{n_{m}}\right| =
     \lim_{m\to
     \infty} \frac{1}{n_{m}}\log\left|\frac{p_{m}'(z)}{n_{m}}\right|
\end{equation*}
with convergence in $\locint$. Hence by \eqref{Cauchy},
\begin{equation}\label{eq2.2}
     u'(z) - u(z) = \lim_{m \to \infty}
     \frac{1}{n_{m}}\log\left|\frac{p_{m}'(z)}{n_{m}p_{m}(z)}\right| =
\lim_{m\to\infty}\frac{1}{n_{m}}\log\left|\int\frac{d\mu_{m}(\zeta)}{z
-\zeta} \right|.
\end{equation}
Now, if $\phi$ is a positive compactly supported test function, then
\begin{equation}\label{eq2.3}
     \begin{split}
    \int\phi(z)(u'(z) - u(z))\,dA(z) &=
    \lim_{m\to\infty}\frac{1}{n_{m}}\int\phi(z)\log\left|\int
    \frac{d\mu_{m}(\zeta)}{z-\zeta}\right|\,dA(z)\\
    &\leq \lim_{m\to \infty}\frac{1}{n_{m}}\int\phi(z) \int
    \frac{d\mu_{m}(\zeta)}{|z-\zeta|}\,dA(z)\\
    &=\lim_{m\to
\infty}\frac{1}{n_{m}}\iint\frac{\phi(z)\,dA(z)}{|z-\zeta|}
    \,d\mu_{m}(\zeta)
\end{split}
\end{equation}
where $dA$ denotes Lebesgue measure in the complex plane. 
Since $1/|z|$ is locally integrable, the function
$\int\phi(z)|z-\zeta|^{-1}\,dA(z)$ is continuous, and hence
bounded by a constant $M$ for all $z$ in $K$.  Since
asymptotically almost all zeros of $\{p_m\}$ belong to $K$, the
last expression in
\eqref{eq2.3} tends to $0$ when $m\to \infty$. 
This proves that $u' \leq u$.

In the complement of $\supp \mu$, $u$ is harmonic and $u'$ is
subharmonic, hence $u' - u$ is a negative subharmonic function.
Moreover, in the complement of $\supp \mu$, $p_{m}'/(n_{m}p_{m})$
converges
  to the Cauchy transform $\C(z)$ of $\mu$ a.e. in $\bC$.
Since $\C(z)$ is a nonconstant holomorphic function in the
unbounded component of $\bC\smallsetminus \supp \mu$, it follows
from \eqref{eq2.2} that $u' - u \equiv 0$ there. 
\end{proof}

Notice that Lemma~\ref{lm:add2} implies the following interesting
fact.

\begin{corollary}\label{cor:path-conn}
In notation of Lemma~\ref{lm:add2}, if $\supp \mu$ has Lebesque
area 0 and the complement  $\bC\smallsetminus \supp \mu$ is
path-connected, then $\mu=\mu'$. In particular, in this case the
whole sequence $\{\mu'_m\}$ weakly converges to $\mu$.
\end{corollary}

\medskip
In general, however $\mu\neq \mu'$ as shown by a trivial example
of the sequence $\{z^n-1\}_{n=1}^\infty$.  Also even if
$\mu=\lim_{m\to \infty}\mu_n$ exists the limit  $\lim_{m\to
\infty} \mu'_n$ does not have to exist for the whole sequence. An
example of this kind is the sequence $\{p_n(z)\}$ where
$p_{2l}(z)=z^{2l}-1$ and $p_{2l+1}(z)=z^{2l+1}-z$, $l=1,2,\dots$.

\medskip
Luckily, the latter phenomenon can never occur for
sequences of Jacobi polynomials, see Proposition~\ref{lm:higher} below.
(Apparently it can not occur for a much more general class of
polynomial sequences introduced in  \S~\ref{Riemann}.)

\begin{lemma}\label{lm:basic2}
If the sequence $\{\mu_n\}$ of the root-counting measures of a
sequence of Jacobi polynomials
$\{p_n(z)\}=\{P_n^{(\al_n,\bet_n)}(z)\}$  weakly converges to a
measure  $\mu$ compactly supported in $\bC,$ and the  sequence $\{\mu'_n\}$ of the root-counting measures of a
sequence
$\{p'_n(z)\}$  weakly converges to a
measure  $\mu'$ compactly supported in $\bC,$
then one of the
following alternatives holds:

\noindent {\rm (i)}  the sequences
$\left\{\frac{\al_n+\bet_n}{n}\right\}$ and
$\left\{\frac{\bet_n-\al_n}{n}\right\}$ {\rm (}and, therefore, the
sequences $\left\{\frac{\al_n}{n}\right\}$ and
$\left\{\frac{\bet_n}{n}\right\}${\rm )} are bounded;

\noindent {\rm (ii)}  the sequence
$\left\{\frac{\al_n+\bet_n}{n}\right\}$ is unbounded and the
sequence $\left\{\frac{\bet_n-\al_n}{n}\right\}$  is bounded, in
which case  $\{\mu_n\}\to \delta_0$ where $\delta_0$ is the unit
point mass at $z=0$  (or, equivalently, $\C_{\delta_0}(z)=1/z$);

\noindent {\rm (iii)}  both  sets
$\left\{\frac{\al_n+\bet_n}{n}\right\}$ and
$\left\{\frac{\bet_n-\al_n}{n}\right\}$   are unbounded, in which
case,  there exists at least one $\kappa \in  \bC$ and a
subsequence $\{n_m\}$  such that $\lim_{m\to
\infty}\frac{\bet_{n_m}-\al_{n_m}}{\al_{n_m}+\beta_{n_m}}=\kappa$
and  $\{\mu_{n_m}\}\to \delta_\kappa, $  where $\delta_\kappa$ is
the unit point mass at  $z=\kappa  $    (or, equivalently,
$\C_{\delta_\kappa}(z)=1/(z-\kappa)$).
\end{lemma}

\medskip
\begin{proof} Indeed, assume that the  alternative (i) does not hold.
Then there is a subsequence $\{n_m\}$ such that at least one of
$\left|\frac{\al_{n_m}+\bet_{n_m}}{n_m}\right|, \; \left|\frac{\bet_{n_m}-\al_{n_m}}{n_m}\right|$
is unbounded along this subsequence.
By our assumptions $\mu_n\to \mu$ and $\mu_n'\to \mu'$ weakly. Hence, by
Lemma~\ref{lm:basic},  there exists a subsequence of indices along which  $\C_{\mu_n}:=\frac{p_n^\prime}{np_n}$
pointwise converges
to $\C_\mu$  and $\C_{\mu'_n}:=\frac{p_n^{\prime\prime}}{(n-1)p'_n}$
pointwise converges
to $\C_{\mu'}$ a.e. in $\bC$. 
Consider the  sequence of differential equations satisfied by
$\{p_n\}$ and divided termwise by $n(n-1)p_n$:
\begin{equation}\label{eq:temp}
\begin{split} %
(1-z^2)\frac{p_n^{\prime\prime}}{(n-1)p^\prime_n} \cdot
\frac{p_n^\prime}{np_n}&+ \left(
\frac{(\bet_n-\al_n)-(\al_n+\bet_n+2)z}{n-1}\right)
\frac{p_n^\prime}{np_n} \\ %
{ }&+\frac{n+\al_n+\bet_n+1}{n-1}=0. %
\end{split} %
\end{equation}

If  for a subsequence of indices, $\left|\frac{\bet_n-\al_n}{n}\right|\to \infty$ while
$\left|\frac{\al_n+\bet_n}{n}\right|$ stays bounded, then the
Cauchy transform $\C_\mu$ of the limiting (along this subsequence) measure $\mu$ must
vanish identically in order for \eqref{eq:temp} to hold in the
limit $n\to \infty$. But $\CC_\mu\equiv 0$  is obviously
impossible.

On the other hand, if  for a subsequence of indices,
$\left|\frac{\al_n+\bet_n}{n}\right|\to \infty$ while
$\left|\frac{\bet_n-\al_n}{n}\right|$ stays bounded, then the
limit of \eqref{eq:temp} when $n\to \infty$ coincides with
$-z\C_\mu+1=0 \Leftrightarrow \C_\mu=\frac{1}{z}$  implying
$\mu=\delta_0$. Thus in Case (ii), the sequence  $\{\mu_n\}$
converges to $\delta_0$.

Now assume, that or a subsequence of indices, both $\left|\frac{\al_n+\bet_n}{n}\right|$ and
$\left|\frac{\bet_n-\al_n}{n}\right|$ tend to $\infty$. Then
dividing \eqref{eq:temp} by $\frac{\al_n+\bet_n}{n}$ and letting
$n\to \infty,$ we conclude that the sequence
$\left\{\frac{\bet_n-\al_n}{\al_n+\bet_n}\right\}$ must be
bounded. Therefore there exists its  subsequence which converges to
some $\kappa\in \bC$. Taking the limit along this subsequence, we
obtain
$$(z-\kappa)\C_\mu=1.$$

This is true for all $z,$ for which the Cauchy transform
converges, i.e. almost everywhere outside the support of $\mu$.
Using the main results of  \cite{BBB, BB}  claiming that  the
support of $\mu$ consists of  piecewise smooth compact curves
and/or isolated points together with the fact that $\C_\mu$ must
have a discontinuity along every curve in its support, we conclude
that the support of $\mu$ is the point $z=\kappa$. Thus in Case
(iii), the sequence $\{\mu_{n_m}\}$ converges to $\delta_\kappa$.
\end{proof}

The next statement provides more information about Case (i) of
Lemma~\ref{lm:basic2}.

\begin{proposition}\label{lm:higher} Assume that  the sequence $\{\mu_n\}$ of the root-counting measures for a sequence  of Jacobi polynomials $\{p_n(z)=P_n^{(\al_n,\bet_n)}(z)\}$   weakly   converges to a compactly supported measure $\mu$ in $\bC.$ Assume additionally that
$\lim_{n\to\infty}\frac{\al_n}{n}=A$ and $\lim_{n\to \infty} \frac {\bet_n}{n}=B$ with $1+A+B\neq 0$. Then, for any positive integer $j,$ the sequence $\{\mu_n^{(j)}\}$ of the root-counting measures for the sequence  $\{p^{(j)}_n(z)\}$ of the $j$-th derivatives converges to the same measure $\mu$.
\end{proposition}

\begin{proof} Observe that  if an arbitrary polynomial sequence $\{p_m\}$ of
increasing degrees has almost all roots in a convex bounded set
$\Omega\subset \bC$, then, by Proposition~\ref{lm:bound},  almost all
roots of $\{p'_m\}$ are in $\Omega_\eps$, for any $\eps>0$.
Therefore, if the sequence $\{\mu_m\}$ of the root-counting
measures of $\{p_m\}$ weakly converges to a compactly supported
measure $\mu,$ then there exists at least one weakly converging
subsequence of $\{\mu'_m\}$. Additionally, by the Gauss-Lucas Theorem,
the support of its limiting measure belongs to  the (closure of
the) convex hull of the support of $\mu$. Thus the weak
convergence of $\{\mu_m\}$ implies the existence of a weakly
converging subsequence $\{\mu^\prime_{n_m}\}$.

\medskip
Proposition~\ref{lm:higher}  is obvious in Cases (ii) and (iii) of Lemma~\ref{lm:basic2}. Let us concentrate on the remaning Case (i). Our assumptions imply that along a subsequence of the sequence  $\left\{\frac{p'_n}{np_n}\right\}$ of Cauchy transforms
of polynomials $p_n$ 
   converges 
pointwise almost
everywhere. We first show that  the above  sequence
$\left\{\frac{p'_n}{np_n}\right\}$  can not converge to $0$ on a
set of positive measure.

\medskip
Indeed,  the differential equation
satisfied by $p_n$  after its  division by $n(n-1)p_n$ is given
by \eqref{eq:temp}.
Since the
sequences $\left\{\frac{\al_n+\bet_n}{n}\right\}$ and
$\left\{\frac{\bet_n-\al_n}{n}\right\}$ converge and $1+A+B\neq 0$,  equation \eqref{eq:temp} shows that   $\frac{p'_n} {np_n}$
cannot converge to $0$
on a set of positive measure. 
Analogously, we  see that
$\frac{p^{\prime\prime}_n}{(n-1)p'_n}$ cannot converge to 0 on a
set of positive measure either. Indeed, differentiating \eqref{eq:temp}, we get that $p_n'$ satisfies the equation
$$(1-z^2)p_n^{\prime\prime\prime}+((\bet_n-\al_n)-(\al_n+\bet_n+4)z)p_n^{\prime\prime}+(n(n+\al_n+\bet_n+1)+(\al_n+\bet_n+2))p_n' =0.$$
Using the same analysis as for $p_n$, we can  conclude
that the limit $\frac{p^{\prime\prime}_n}{n(n-1)p_n}$  along a subsequence exists
pointwise and is non-vanishing  almost everywhere.

Denote the logarithmic potentials of the root-counting measures
associated to $p_n$ and  $p'_n$ by $u_n$ and $u'_n$ respectively. Denote
their limits
 by $u$ and $u'$  (where $u'$ apriori is a limit only along some subsequence).  With a slight abuse of notation,  the following holds
$$|u - u'| =\lim_{n\to\infty} |u_n - u'_n | = \lim_{n\to\infty} \frac{1}{n} \log \left | \frac{p^{\prime\prime}_n} {n(n - 1)p_n} \right |= 0$$ due  to the above claim about $\frac{p^{\prime\prime}_n}{n(n-1)p_n}$ . But since $u \ge  u'$ by Lemma~\ref{lm:add2},  we see that $u = u'$ and, in particular $u'$ exists as a limit over the whole sequence.  Hence the asymptotic root-counting measures of $\{p_n\}$ and $\{p'_n\}$ actually coincide. Similar arguments apply to higher derivatives of the sequence $\{p_n\}$. 
\end{proof}

\begin{proof}[Proof of  Theorem~\ref{Theorem-1}]  The polynomial $p_n(z)=P_n^{ (\al_n,\bet_n) }(z)$ satisfies the equation~\eqref{eq:temp}.
  By Proposition~\ref{lm:higher} we know that, under the assumptions of Theorem~\ref{Theorem-1},  if $ \left\{\frac {p_n'}{np_n}\right\}$ converges  to $\C_\mu$ a.e. in $\bC,$ then the sequence $ \left\{\frac {p_n^{\prime\prime}}{np_n^\prime}\right\}$  also converges  to the same $\C_\mu$ a.e. in $\bC$.  Therefore, the expression $\frac{p_n^{\prime\prime}}{n^2p_n}=\frac{p_n^{\prime\prime}p_n^\prime}{n^2p_np_n^\prime}$ converges to $\C_\mu^2$ a.e. in $\bC$. Thus  $\C_\mu$ (which is well-defined a.e. in $\bC$) should satisfy the equation
$$(1-z^2)\C_{\mu}^2-((A+B)z+A-B)\C_\mu+A+B+1=0,$$
where $A=\lim_{n\to \infty}\frac{\al_n}{n}$ and $B=\lim_{n\to \infty}\frac{\bet_n}{n}$.
\end{proof}

\begin{remark} Apparently the condition that  the sequences $\left\{\frac{\al_n}{n}\right\}$ and
 $\left\{\frac{\bet_n}{n}\right\}$ are bounded should be enough for the conclusion of Theorem~\ref{Theorem-1}. (The existence of the limits $\lim \frac{\al_n}{n}$ and $\lim \frac{\bet_n}{n}$ should follow automatically with some weak additional restriction.) Indeed, since the sequences $\left\{\frac{\al_n}{n}\right\}$ and
 $\left\{\frac{\bet_n}{n}\right\}$ are bounded, we can  find at
 least one subsequence $\{n_m\}$ of indices along which both
 sequences of quotients converge. Assume that we have two possible
 distinct  (pairs of) limits $(A_1,B_1)$ and $(A_2,B_2)$ along
 different subsequences. But then the same complex-analytic function $\C_\mu(z)$ should satisfy a.e. two different algebraic equations of the form \eqref{1.2}  which is impossible at least for generic $(A_1,B_1)$ and $(A_2,B_2)$.
\end{remark}

\section{Preliminaries on quadratic differentials} \label{Section-4}
\setcounter{equation}{0}

In this section, we recall some definitions and results of the
theory of quadratic differentials on the complex sphere
$\overline{\mathbb{C}}=\mathbb{C}\cup\{\infty\}$. Most of these
results remain true for quadratic differentials defined on any
compact Riemann surface. But for the purposes of this paper, we
will focus on results concerning the domain structure and
properties of geodesics of quadratic differentials defined on
$\overline{\mathbb{C}}$. For more information on quadratic
differentials in general, the interested reader may consult
classical monographs of Jenkins \cite{Je} and Strebel \cite{Str}
and papers \cite{S1} and \cite{S2}.

A quadratic differential on a domain $D\subset
\overline{\mathbb{C}}$ is a differential form $Q(z)\,dz^2$ with
meromorphic $Q(z)$ and with conformal transformation rule
\begin{equation} \label{4.1} %
Q_1(\zeta)\,d\zeta^2=Q(\varphi(z))\left(\varphi'(z)\right)^2\,dz^2,
\end{equation}
where $\zeta=\varphi(z)$ is a conformal map from $D$ onto a domain
$G\subset \overline{\mathbb{C}}$. Then zeros and poles of $Q(z)$
are critical points of $Q(z)\,dz^2$, in particular, zeros and
simple poles are finite critical points of $Q(z)\, dz^2$. Below we
will use the following notations. By $H_p$, $C$, and $H$ we
denote, respectively, the set of all poles, set of all finite
critical points, and  set of all infinite critical points of
$Q(z)\,dz^2$. Also, we will use the following notations:
$\mathbb{C}'=\overline{\mathbb{C}}\setminus H$,
$\mathbb{C}''=\overline{\mathbb{C}}\setminus H_p$,
$\mathbb{C}'''=\overline{\mathbb{C}}\setminus (C\cup H)$.

 A trajectory (respectively, orthogonal trajectory) of $Q(z)\,dz^2$
is a closed analytic Jordan curve or maximal open analytic arc
$\gamma\subset D$ such
that %
$$ 
Q(z)\,dz^2>0 \quad {\mbox{along $\gamma$}} \quad \quad
({\mbox{respectively}}, Q(z)\,dz^2<0 \quad {\mbox{along
$\gamma$}}).
$$ 

A trajectory $\gamma$ is called \emph{critical} if at least one of
its end points is a finite critical point of $Q(z)\,dz^2$.  By a
closed critical trajectory we understand a critical trajectory
together with its end points $z_1$ and $z_2$ (not necessarily
distinct), assuming that these end points exist.

Let $\overline{\Phi}$ denote the closure of the set of points of
all critical trajectories of $Q(z)\,dz^2$. Then, by Jenkins' Basic
Structure Theorem \cite[Theorem~3.5]{Je}, the set
$\overline{\mathbb{C}}\setminus \overline{\Phi}$ consists of a
finite number of \emph{circle, ring, strip and end domains}. The
collection of all these domains together with so-called
\emph{density domains} constitute the so-called \emph{domain
configuration} of $Q(z)\,dz^2$. Here, we give definitions of
circle domains and strip domains only; these two types will appear
in our classification of possible domain configurations in
Section~5. Fig. 1--4 show several  domain configurations with
circle and strip domains. For the definitions of other domains, we
refer to \cite[Ch.~3]{Je}.

We recall that a \emph{circle domain} of $Q(z)\,dz^2$ is a
simply connected domain $D$ with the following properties: %
\begin{enumerate} %
\item[1)] %
$D$ contains exactly one critical point $z_0$, which is a second
order pole, %
\item[2)] %
the domain $D\setminus \{z_0\}$  is swept out by trajectories of
$Q(z)\,dz^2$ each of which is a Jordan curve separating $z_0$ from
the boundary $\partial D$,
\item[3)] %
$\partial D$ contains at least one finite critical
point. %
\end{enumerate} %

Similarly, a strip domain of $Q(z)\,dz^2$ is a simply connected
domain $D$ with the following properties: %
\begin{enumerate} %
\item[1)] %
$D$ contains no  critical points of $Q(z)\,dz^2$, %
\item[2)] %
$\partial D$ contains exactly two boundary points $z_1$ and $z_2$
belonging to the set $H$ (these boundary points may be situated at
the same point of $\overline{\mathbb{C}}$),
\item[3)] %
the points $z_1$ and $z_2$ divide $\partial D$ into two boundary
arcs each of which contains at least one finite critical point, %
\item[4)] %
$D$ is swept out by trajectories of $Q(z)\,dz^2$ each of which is
a Jordan arc connecting points $z_1$ and $z_2$.
\end{enumerate} %

As we mentioned in the Introduction, every quadratic differential
$Q(z)dz^2$   defines the so-called (singular) $Q$-metric with the
differential element $|Q(z)|^{1/2}\,|dz|$. If $\gamma$ is a
rectifiable arc in $D$ then its $Q$-length is
defined by %
$$ %
|\gamma|_Q=\int_\gamma |Q(z)|^{1/2}\,|dz|. %
$$ %

According to their $Q$-lengthes, trajectories of $Q(z)\,dz^2$ can
be of two types. A trajectory $\gamma$ is called \emph{finite} if
its $Q$-length is finite, otherwise $\gamma$ is called
\emph{infinite}. In particular, a critical trajectory $\gamma$ is
finite if and only if it has two end points each of which is a
finite critical point.

 An important property of quadratic differentials is that
transformation rule (\ref{4.1}) respects trajectories and
orthogonal trajectories and their $Q$-lengthes, as well as it
respects critical points together with their multiplicities and
trajectory structure nearby.

\begin{definition}  \label{Definition-4.1} %
A locally rectifiable (in the spherical metric) curve
$\gamma\subset \mathbb{C}'$  is called \emph{ a $Q$-geodesic} if
it is locally shortest in the $Q$-metric.
\end{definition} %

Next, given a quadratic differential $Q(z)\,dz^2$, we will discuss
geodesics in homotopic classes. For any two points $z_1,z_2\in
\mathbb{C}'$, let $\mathcal{H}^J=\mathcal{H}^J(z_1,z_2)$ denote
the set of all homotopic classes $H$ of Jordan arcs $\gamma\subset
\mathbb{C}'$ joining $z_1$ and $z_2$. Here the letter $J$ stands
for "Jordan". It is well-known that there is a countable number of
such homotopic classes. Thus, we may write
$\mathcal{H}^J=\{H_k^J\}_{k=1}^\infty$.

Every class $H_k^J$ can be extended to a larger class $H_k$ by
adding non-Jordan continuous curves $\gamma$ joining $z_1$ and
$z_2$, each of which is homotopic on $\mathbb{C}'$ to some curve
$\gamma_0\in H_k^J$ in the following sense.

There is a continuous function $\varphi(t,\tau)$ from the square
$I^2:=[0,1]\times [0,1]$ to $\mathbb{C}'$ such that %
\begin{enumerate} %
\item[1)]%
$\varphi(0,\tau)=z_1$, $\varphi(1,\tau)=z_2$ for all $0\le \tau\le
1$, \item[2)]
$\gamma_0=\{z=\varphi(t,0):\,0\le t\le 1\}$, %
\item[3)] %
$\gamma=\gamma_1=\{z=\varphi(t,1):\,0\le t\le 1\}$, %
\item[4)] %
For every fixed $\tau, 0<\tau<1$, the curve
$\gamma_\tau=\{z=\varphi(t,\tau):\,0\le t\le 1\}$ is in the class
$H_k^J$.
\end{enumerate} %

The following proposition is a special case of a well-known result
about geodesics, see e.g. \cite[Theorem~18.2.1]{Str}.

\begin{proposition} \label{Proposition-4.1} %
For every $k$, there is a unique curve $\gamma'\in H_k$, called
\emph{$Q$-geodesic} in $H_k$, such that $|\gamma'|_Q< |\gamma|_Q$
for all
$\gamma\in H_k$, $\gamma\not=\gamma'$. This geodesic is not necessarily a Jordan arc. %
\end{proposition} %

A $Q$-geodesic from $z_1$ to $z_2$ is called \emph{simple} if
$z_1\not=z_2$ and $\gamma$ is a Jordan arc on $\mathbb{C}'''$
joining $z_1$ and $z_2$. A $Q$-geodesic is called \emph{critical}
if both its end points belong to the set of finite critical points
of $Q(z)\,dz^2$.

\begin{proposition} \label{Proposition-4.2} %
Let $Q(z)\,dz^2$ be a quadratic differential on
$\overline{\mathbb{C}}$. Then for any two points $z_1,z_2\in
\mathbb{C}'$ and every continuous rectifiable curve $\gamma$ on
$\mathbb{C}'''$ joining the points $z_1$ and $z_2$ there is a
unique shortest curve $\gamma_0$ belonging to the homotopic class
of $\gamma$.

Furthermore,  $\gamma_0$ is a
geodesic in this class. %
\end{proposition} %

\begin{definition}  \label{Definition-4.2} %
Let $z_0\in \mathbb{C}'$. A geodesic ray from $z_0$ is a maximal
simple rectifiable arc $\gamma:[0,1) \to \mathbb{C}'''\cup\{z_0\}$
with $\gamma(0)=z_0$ such that for every $t$, $0<t<1$, the arc
$\gamma((0,1))$ is a geodesic from $z_0$ to $z=\gamma(t)$.
\end{definition} %

\begin{lemma} \label{Lemma-4.1} %
Let $D$ be a circle domain of $Q(z)\,dz^2$ centered at $z_0$ and
let $\gamma_a:[0,1)\to \mathbb{C}'''\cup\{a\}$ be a geodesic ray
from $a\in
\partial D$ such that $\gamma_a([0,t_0])\subset \overline{D}$ for
some $t_0>0$.

Then either $\gamma_a$  enters into $D$ through the point $a$ and
then approaches to $z_0$ staying in $D$ or $\gamma_a$ is an
arc of some critical trajectory $\gamma\subset \partial D$. %
\end{lemma} %

\begin{lemma} \label{Lemma-4.2} %
Let $a$ be a second order pole of $Q(z)\,dz^2$ and let $\Gamma$ be
the homotopic class of closed curves on $\mathbb{C}''$ separating
$a$ from $H_p\setminus \{a\}$. Then there is exactly one real
$\theta_0$, $0\le \theta_0<2\pi$, such that the quadratic
differential $e^{i\theta_0}Q(z)\,dz^2$ has a circle domain, say
$D_0$, centered at $a$. Furthermore, the boundary $\partial D_0$
is the only critical $Q$-geodesic (non-Jordan in general) in the
class $\Gamma$.

In particular, $\Gamma$ may contain at most one critical geodesic
loop. %
\end{lemma} %

We will need some simple mapping properties of the canonical
mapping related to the quadratic differential $Q(z)\,dz^2$, which
is
defined by %
$$ %
F(z)=\int_{z_0} \sqrt{Q(z)}\,dz %
$$ %
with some $z_0\in \overline{\mathbb{C}}$ and some fixed branch of
the radical. A simply connected domain $D$ without critical points
of $Q(z)\,dz^2$ is called a $Q$-rectangle if the boundary of $D$
consists of two arcs of trajectories of $Q(z)\,dz^2$ separated by
two arcs of orthogonal trajectories of this quadratic
differential. As well a canonical mapping $F(z)$ maps any
$Q$-rectangle conformally onto a geometrical rectangle in the
plane with two sides parallel to the horizontal axis.

\section {Cauchy transforms satisfying quadratic equations and quadratic differentials} \label{sq-roots}

Below we  relate the question for which triples of polynomials
$(P,Q,R)$ the  equation

\begin{equation}\label{quadr}
P(z)\C^2+Q(z)\C+R(z)=0,
\end{equation}
with $\deg P=n+2,\;\deg Q\le n+1,\;\deg R\le n$
 admits a compactly supported signed measure $\mu$ whose Cauchy transform satisfies
 \eqref{quadr} almost everywhere in $\bC$ to a certain problem about rational quadratic differentials. We call such measure $\mu$ a {\it motherbody measure} for \eqref{quadr}.

For a given quadratic differential $\Psi$ on  a compact surface  $\mathcal R,$  denote by
$K_\Psi\subset \mathcal R$ the union of all its critical
trajectories and critical points. (In general, $K_\Psi$ can be very
complicated. In particular, it can be dense in some subdomains of
$\mathcal R$.)  We denote by $DK_\Psi\subseteq K_\Psi$ (the closure
of) the set of  finite critical trajectories of
\eqref{eq:maindiff}. (One can  show that $DK_\Psi$ is an imbedded
(multi)graph in $\mathcal R$. Here by a {\it multigraph} on a
surface we mean a graph with possibly multiple edges and loops.)
Finally, denote by $DK^0_\Psi\subseteq DK_\Psi$ the subgraph of
$DK_\Psi$ consisting of (the closure of) the set of  finite
critical trajectories whose both ends are zeros of $\Psi$.

A non-critical trajectory $\gamma_{z_0}(t)$ of a meromorphic
$\Psi$  is called \emph{closed} if $\exists \ T>0$ such that
$\gamma_{z_0}(t+T)=\gamma_{z_0}(t)$ for all $t\in\mathbb{R}$. The
least such $T$ is called the \emph{period} of $\gamma_{z_0}$. A
quadratic differential $\Psi$ on a compact Riemann surface
$\mathcal R$ without boundary is called \emph{Strebel}  if the set
of its closed trajectories covers $\mathcal R$ up to a set of
Lebesgue measure
 zero.

\medskip
Going back to Cauchy transforms, we formulate the following necessary condition of the existence
of a motherbody measure for \eqref{quadr}.

\begin{proposition}\label{pr:stand}
Assume that equation \eqref{quadr} admits a signed motherbody measure $\mu$. Denote by    $D(z)=Q^2(z)-4P(z)R(z)$  the discriminant of equation \eqref{quadr}.
Then the following two conditions hold:

\noindent {\rm (i)} any connected smooth curve in the support of $\mu$
coincides with a horizontal trajectory of the quadratic
differential
\begin{equation}\label{eq:maindiff}
\Theta=-\frac{D(z)}{P^2(z)}dz^2=\frac{4P(z)R(z)-Q^2(z)}{P^2(z)}dz^2.
\end{equation}

\noindent
{\rm (ii)}  the support of $\mu$  includes all branching points of \eqref{quadr}.
\end{proposition}

\noindent {\it Remark.} Observe that if $P(z)$ and $Q(z)$ are
coprime, the set of all branching points coincides with the set of
all zeros of $D(z)$.  In particular, in this case  part (ii) of
Proposition~\ref{pr:stand} implies that the set $DK_\Theta^0$ for
the differential $\Theta$ should contain all  zeros  of $D(z)$.

\medskip
\noindent {\it Remark.} Proposition~\ref{pr:stand} applied to
quadratic differential $Q(z)\,dz^2$ of Theorem~\ref{Theorem-1}
implies Theorem~\ref{Theorem-2}.

\begin{proof} The fact that every curve in $\text{supp}(\mu)$ should coincide with some horizontal trajectory of \eqref{eq:maindiff} is well-known and follows from the  Plemelj-Sokhotsky's formula. It is based on the local observation that if a real measure $\mu=\frac{1}{\pi}\frac{\partial \C}{\partial \bar z}$ is supported
 on a smooth curve $\ga$, then the tangent to $\gamma$ at any  point $z_0\in \gamma$ should be perpendicular to $\overline {\C_1(z_0)}- \overline{\C_2(z_0)}$ where $\C_1$ and $\C_2$ are the one-sided limits of $\C$ when  $z\to z_0$, see e.g. \cite{BR}. (Here\;  $\bar {}$\;  stands for the usual complex  conjugation.) Solutions of \eqref{quadr} are given by $$\C_{1,2}=\frac{-Q(z)\pm \sqrt{Q^2(z)-4P(z)R(z)}}{2P(z)},$$  their difference being $$\C_1-\C_2=\frac{\sqrt{Q^2(z)-4P(z)R(z)}}{P(z)}.$$
Since the tangent line to the support of the real motherbody measure
$\mu$ satisfying \eqref{quadr} at its arbitrary smooth point
$z_0$,   is  orthogonal to $\overline{ \C_1(z_0)}-\overline
{\C_2(z_0)},$ it is exactly given by the condition
$\frac{4P(z_0)R(z_0)-Q^2(z_0)}{P^2(z_0)}dz^2>0$. The latter
condition   defines the horizontal trajectory of $\Theta$ at
$z_0$.


  Finally the observation that $\text{supp }\mu$ should contain all branching points of \eqref{quadr} follows immediately from the  fact that $\C_\mu$ is a well-defined univalued function in $\bC\setminus  \text{supp }\mu$.
\end{proof}

In many special cases  statements similar to
Proposition~\ref{pr:stand}  can be found in the literature, see
e.g. recent \cite {AMFMGT} and references therein.

\medskip
Proposition~\ref{pr:stand} allows us, under mild nondegeneracy
assumptions, to formulate  necessary and sufficient conditions for
the existence of a motherbody measure for \eqref{quadr}  which
however are difficult to verify. Namely, let $\Ga\subset
\bCP^1\times \bCP^1$ with affine coordinates $(\C,z)$ be the algebraic
curve given by (the projectivization of) equation \eqref{quadr}.
$\Ga$ has bidegree $(2,n+2)$ and is hyperelliptic. Let
$\pi_z:\Ga\to \bC$ be the projection of $\Ga$ on  the $z$-plane
$\bCP^1$ along the $\C$-coordinate. From \eqref{quadr} we observe
that $\pi_z$ induces a branched double covering of $\bCP^1$ by
$\Ga$.  If $P(z)$ and $Q(z)$ are coprime and if $\deg D(z)=2n+2$,
the set of all branching points of $\pi_z: \Ga\to \bCP^1$
coincides with the set of all zeros of $D(z)$. (If $\deg
D(z)<2n+2,$ then $\infty$ is also a branching pont of $\pi_z$ of
multiplicity $2n+2-\deg D(z)$.)   We need the following lemma.

\begin{lemma}\label{lm:poles} If $P(z)$ and $Q(z)$ are coprime,
then at each pole of \eqref{quadr} i.e. at each zero  of $P(z)$,
only one of two branches of $\Ga$ goes to  $\infty$. Additionally the residue
of this branch at this zero equals that of $-\frac {Q(z)}{P(z)}$.
\end{lemma}

\begin{proof} Indeed if $P(z)$ and $Q(z)$ are coprime, then no zero $z_0$ of $P(z)$
can be a branching point of \eqref{quadr} since $D(z_0)\neq 0$.
Therefore only one of two branches of $\Ga$ goes to $\infty$ at $z_0$.
More exactly, the branch
 $\C_1=\frac{-Q(z)+ \sqrt{Q^2(z)-4P(z)R(z)}}{2P(z)}$ attains a finite value at $z_0$
 while the branch $\C_2=\frac{-Q(z)- \sqrt{Q^2(z)-4P(z)R(z)}}{2P(z)}$ goes to  $\infty$
 where we use the agreement that $\lim_{z\to z_0} \sqrt{Q^2-4P(z)R(z)}=Q(z_0)$.
 Now consider the residue of the branch $\C_2$ at $z_0$. Since residues depend continuously
 on the coefficients $(P(z),Q(z),R(z))$  it suffices to consider only the case when $z_0$
 is a simple zero of $P(z)$.  Further if $z_0$ is a simple zero of $P(z),$ then

$$Res(\C_2,z_0)= \frac {-2Q(z_0)} {2P^\prime(z_0)}= Res \left(-\frac{Q(z)}{P(z)}, z_0\right),$$
which completes the proof. \end{proof}

\medskip
By Proposition~\ref{pr:stand} (besides the obvious condition that \eqref{quadr} has
a real branch near $\infty$ with the asymptotics $\frac{\al}{z}$ for some $\al\in \bR$)
the necessary condition for \eqref{quadr} to admit a motherbody measure is that
the set $DK_\Theta^0$ for the differential \eqref{eq:maindiff} contains all branching points
of \eqref{quadr}, i.e. all zeros of $D(z)$.  Consider $\Ga_{cut}:=\Ga\setminus \pi_z^{-1}(DK_\Theta^0)$.
Since $DK_\Theta^0$ contains all branching points of $\pi_z$, $\Ga_{cut}$ consists
of some number of open sheets, each projecting diffeomorphically on its image
in $\bCP^1\setminus DK^0_\Theta$. (The number of sheets in $\Ga_{cut}$ equals
to twice the number of connected components in $\bC\setminus DK_\Theta^0$.)
Observe that since we have chosen a real branch of \eqref{quadr} at infinity
with the asymptotics $\frac{\al}{z}$,  we have a marked point  $p_{br}\in \Ga$ over $\infty$.
If we additionally assume that $\deg D(z)=2n+2,$ then $\infty$ is not a branching point of $\pi_z$
and   therefore $p_{br}\in \Ga_{cut}$. 

\begin{lemma}\label{lm:cut} If $\deg D(z)=2n+2$, then any choice of a
spanning (multi)subgraph $G\subset DK_\Theta^0$ with no isolated
vertices induces the unique choice of the section $S_G$ of $\Ga$ over $\bCP^1\setminus G$ which:

\medskip
\noindent a) contains $p_{br}$; b) is discontinuous at any point
of $G$; c) is projected    by $\pi_z$ diffeomorphically onto
$\bCP^1\setminus G$.
\end{lemma}

Here by a spanning subgraph  we mean a subgraph containing all the
vertices of the ambient graph. By a section of $\Ga$ over
$\bCP^1\setminus G$ we mean a choice of one of two possible values
of $\Ga$ at each point in $\bCP^1\setminus G$. After these
clarifications the proof is evident.

\medskip

Observe that the section $S_G$ might attain the value $\infty$ at
some points, i.e. contain some poles of \eqref{quadr}. Denote the
set of poles of $S_G$ by $Poles_G$. Now we can formulate our
necessary and sufficient conditions.

\begin{theorem}\label{th:necsuf} Assume that the following conditions are valid:

\noindent {\rm (i)} equation \eqref{quadr} has a real branch near
$\infty$ with the asymptotic behavior  $\frac{\al}{z}$ for some
$\al\in \bR$;

\noindent {\rm (ii)} $P(z)$ and $Q(z)$ are coprime, and the
discriminant   $D(z)=Q^2(z)-4P(z)R(z)$ of equation \eqref{quadr}
has degree $2n+2$;

\noindent {\rm (iii)}  the set $DK_\Theta^0$ for  the quadratic
differential $\Theta$ given by \eqref{eq:maindiff} contains all
zeros of $D(z)$;

\noindent {\rm (iv)} $\Theta$ has no closed horizontal
trajectories.

Then \eqref{quadr} admits a real motherbody measure if and only if
there exists a spanning  (multi)subgraph $G\subseteq DK^0_\Theta$
with no isolated vertices, such that all poles in $Poles_g$ are
simple and all  their residues  are real, see notation above.
\end{theorem}

\begin{proof} Indeed assume that $\eqref{quadr}$   satisfying {\rm(ii)}
admits a real motherbody measure $\mu$. Assumption {\rm(i)}
is obviously neccesary for the existence of a real motherbody measure
and the necessity of  assumption {\rm(iii)} follows
from Proposition~\ref{pr:stand} if {\rm(ii)} is satisfied.
The support of $\mu$ consists of a finite number of  curves
and possibly a finite number of isolated points. Since each curve
in the support of $\mu$ is a trajectory of $\Theta$ and $\Theta$
has no closed trajectories, then the whole support of $\mu$ consists
of finite critical  trajectories of $\Theta$ connecting its zeros,
i.e. belongs to $DK_\Theta^0$. Moreover the support of $\mu$
should contain sufficently many finite critical trajectories
of $\Theta$ such that they include all the branching points
of \eqref{quadr}. By {\rm(ii)} these are exactly all zeros of $D(z)$.
Therefore the union of finite critical trajectories of $\Theta$
belonging to the support of $\mu$ is a spanning (multi)graph of $DK^0_\Theta$
without isolated vertices. The isolated points in the support of $\mu$
are necessarily the poles of \eqref{quadr}. Observe that the Cauchy transform
of any (complex-valued) measure can only have simple poles (as opposed
to the Cauchy transform of a more general distribution).
Since $\mu$ is real the residue of its Cauchy transform at each pole must be real as well.
Therefore the existence of a real motherbody under the assumptions {\rm (i)}--{\rm(iv)}
implies the existence of a spanning (multi)graph $G$ with the above properties.
The converse is also immediate.
\end{proof}

\noindent {\it Remark.} Observe that if {\rm (i)} is valid, then
assumptions  {\rm (ii)} and {\rm (iv)} are generically satisfied.
Notice however that {\rm (iv)} is violated in the special case
when $Q(z)$ is absent.  Additionally, if {\rm (iv)} is satisfied,
then the number of possible motherbody measures is finite.  On the
other hand, it is the assumption {\rm (iii)} which  imposes severe
additional restrictions on admissible triples $(P(z),Q(z),R(z))$.
At the moment the authors have  no information about possible
cardinalities of the sets $Poles_G$ introduced above. Thus it is
difficult to estimate the number of conditions required for
\eqref{quadr} to admit a motherbody measure.
 Theorem~\ref{th:necsuf} however leads to the following sufficient condition for the existence of a real motherbody measure for \eqref{quadr}.

\begin{corollary}\label{cor:suf} If, additionally to assumptions {\rm (i)}--{\rm (iii)} of Theorem~\ref{th:necsuf}, one assumes  that all roots of $P(z)$ are simple and all residues of $\frac{Q(z)}{P(z)}$ are real,  then \eqref{quadr} admits a real motherbody measure.
\end{corollary}

\begin{proof} Indeed if all roots of $P(z)$ are simple and all residues of $\frac{Q(z)}{P(z)}$ are real, then all  poles of \eqref{quadr} are simple with real residues.  In this case for any choice of a  spanning (multi)subgraph $G$ of $DK_\Theta^0$, there exists a real motherbody measure whose support coincides with $G$ plus possibly some poles of \eqref{quadr}.  Observe that if all roots of $P(z)$ are simple and all residues of $\frac{Q(z)}{P(z)}$ are real one can omit assumption {\rm (iv)}. In case when $\Theta$ has no closed trajectories, then all possible real motherbody measures are in a bijective correspondence with all spanning (multi)subgraphs of $DK_\Theta^0$ without isolated vertices. In the opposite case such measures are in a bijective correspondence with the unions of  a spanning (multi)subgraph of $DK_\Theta^0$ and an arbitrary (possibly empty)  finite collection of closed trajectories.
\end{proof}


\section{Does weak convergence of Jacobi polynomials imply stronger forms of convergence?}\label{Riemann}

Observe that, if one considers an arbitrary sequence
$\{s_n(z)\},\;n=0,1,\dots$ of monic univariate polynomials of
increasing  degrees, then even if the sequence $\{\theta_n\}$ of
their root-counting measures weakly converges to some limiting
probability measure $\Theta$ with compact support in $\bC$, in
general, it is not true that the roots of $s_n$ stay on some
finite distance from $\supp \Theta$ for all $n$ simultaneously.
Similarly nothing can be said in general about the weak
convergence of the sequence $\{\theta^\prime_n\}$ of the
root-counting measures of $\{s^\prime_n(z)\}$. However we have
already seen that the situation with sequences of Jacobi
polynomials seems to be  different, comp.
Proposition~\ref{lm:higher}.

\medskip
In the present appendix we formulate a general conjecture (and
give some evidence of its validity) about sequences of Jacobi
polynomials as well as sequences of  more general polynomial
solutions of a special class of linear differentials equations
which includes Riemann's differential equation.

\medskip
Consider    a  linear ordinary differential operator
\begin{equation}\label{eq:oper}
\dq=\sum_{i=1}^kQ_j(z)\frac{d^{j}}{dz^j}
\end{equation}
 with polynomial coefficients. We say that \eqref{eq:oper} is {\it exactly solvable} if a) $\deg Q_j\le j,$ for all $j=1,\dots, k$;   b)  there exists at least one value $j_0$ such that $\deg Q_{j_0}(z)=j_0$. We say that an exactly solvable operator \eqref{eq:oper}
is {\it non-degenerate} if $\deg Q_k=k$.

Observe that  any exactly solvable operator $\dq$ has a unique (up
to a constant factor) eigenpolynomial of any sufficiently large
degree, see e.g. \cite {BR}. Fixing an arbitrary monic  polynomial
$Q_k(z)$ of degree $k$, consider the family $\mathcal F_{Q_k}$ of
all exactly solvable operators of the form \eqref{eq:oper} whose
leading term is $Q_k(z)\frac{d^k}{dz^k}$. ($\mathcal F_{Q_k}$ is a
complex affine space of dimension $\binom {k+1}{2}-1$.)   Given a
sequence $\{\dqn\}$ of exactly solvable operators from $\mathcal
F_{Q_k}$ of the form
$$\dqn= Q_k(z)\frac{d^k}{dz^k}+\sum_{i=1}^{k-1}Q_{j,n}(z)\frac{d^{j}}{dz^j},$$
we say that this sequence has a  {\it moderate growth} if, for
each $j=1,\dots, k-1,$ the sequence of polynomials
$\left\{\frac{Q_{j,n}(z)}{n^{k-j}}\right\}$ has all bounded
coefficients.  (Recall that $\forall n$, $\deg Q_{j,n}\le j$.)

\begin{conjecture}\label{conj:bounded}
For any    sequence  $\{\dqn\}$ of exactly solvable operators of
moderate growth, the union of all roots of all the
eigenpolynomials of all $\dqn$ is bounded in $\bC$.
\end{conjecture}

Now take a sequence $\{s_n(z)\},\; \deg s_n=n$ of polynomial
eigenfunctions of the sequence of operators $\dqn\in \mathcal
F_{Q_k}$. (Observe that, in general,  we have  a   different
exactly solvable  operator for each eigenpolynomial but with the
same leading term.)

\medskip
\begin{conjecture}\label{conj:Main} In the above notation, assume that $\{\dqn\}$
is a sequence  of exactly solvable operators
of  moderate growth and that $\{s_n(z)\}$ is the sequence of their
eigenpolynomials (i.e $s_n(z)$ is the eigenpolynomial of $\dqn$ of
degree $n$) such that:

\noindent
 {\rm a)} the limits $\widetilde Q_{j}(z):=\lim_{n\to \infty} \frac{1}{n^{k-j}}Q_{j,n}(z),\;
 j=1,\dots, k-1$ exist;

\noindent {\rm b)} the sequence $\{\theta_n\}$ of the
root-counting measures of $\{s_n(z)\}$ weakly converges to a
compactly supported probability measure $\Theta$ in $\bC$,

\smallskip
\noindent then

\smallskip
\noindent {\rm (i)} the Cauchy transform $\C_\Theta$ of $\Theta$
satisfies a.e. in $\bC$ the algebraic equation
\begin{equation}\label{genCauchy}
Q_k(z)\left(\frac{\C_{\Theta}}{\ga}\right)^k+\sum_{j=1}^{k-1}\widetilde
Q_j(z)\left(\frac{\C_\Theta}{\ga}\right)^j=1,
\end{equation}
where $\ga=\lim_{n\to \infty}\frac{\root k \of {\la_n}}{n},$ \quad $\la_n$
being the eigenvalue of $s_n(z)$.  

\smallskip
\noindent {\rm (ii)} for any positive $\eps>0,$ there exist
$n_\eps$ such that, for $n\ge n_{\eps}$, all roots of all
eigenpolynomials $s_n(z)$  are located within $\eps$-neighborhood
of $\supp \Theta$, i.e., the weak convergence of $\theta_n\to
\Theta$ implies a stronger form of this convergence.
\end{conjecture}

Certain cases of Part (i) of  the above Conjecture are settled in
\cite {BR} and \cite {BBS} and a version of Part (ii) is discussed
in an unpublished preprint \cite {BoPi}.

Now we present some partial confirmation of the above conjectures.
Consider the  family of linear differential operators of second
order depending on parameter $\la$ and given by
\begin{equation}\label{eq:pencil2}
T_\la = Q_2(z)\frac{d^2}{dz^2} + (Q_1(z)\la + P_1(z))\frac{d}{dz}+
(\la^2 + p\la + q)Q_0,
\end{equation}
where  $Q_2(z)$ is a quadratic polynomial  in $z$, $Q_1(z)$ and
$P_1(z)$ are polynomials in $z$ of degree at most $1,$  and $Q_0$
is a non-vanishing constant.  (Observe that our use of parameter
$\la$ here is the same as of the parameter $\ga$ in the latter
Conjecture.)

Denote $Q_i(z) =\sum_{j=0}^iq_{ji}z^j,\;  i = 0,1, 2$  and put
$P_1 = p_{11}z + p_{01}$. The quadratic polynomial
\begin{equation}\label{eq:charac}
q_{22}+ q_{11}t+q_{00}t^2
\end{equation}
is called the {\it characteristic polynomial} of $T_\la$. Here
$q_{22}\neq 0$ and $q_{00}=Q_0\neq 0$.

\begin{definition}
We say that the family $T_\la$ has a {\it generic type} if the
roots of \eqref{eq:charac} have distinct arguments (and in
particular $0$ is not a root of \eqref{eq:charac} which is
guaranteed by $q_{22}\neq 0$ together with $q_{00}\neq 0$), comp.
\cite{BBS}.
\end{definition}

 Below we will denote the roots of characteristic polynomial \eqref{eq:charac}
 by $\al_1$ and $\al_2$. Thus $T_\la$ has a generic type if and only if $\arg \al_1\neq \arg \al_2$.

\begin{lemma}\label{gentype}
Equation~\eqref{eq:charac} has two roots with the same arguments
if and only if $q_{22}q_{00}=\rho q_{11}^2,$ where $0\le \rho\le
\frac{1}{4}$.
\end{lemma}

\begin{proof} Straightforward calculation, see Example 1 of \cite{BBSh1}. \end{proof}

\begin{lemma}\label{pr:basic}  In the  above notation, for a family $T_\la$ of generic type,
there exists a positive integer $N$ such that, for any integer
$n\ge N,$ there exist two eigenvalues $\la_{1,n}$ and $\la_{2,n}$
such that the differential equation
\begin{equation}\label{eq:triv}
T_\la(y) =0
\end{equation}
has a polynomial solution of degree $n$. Moreover,  $\lim_{n\to
\infty}\frac{\la_{i,n}}{n}=\al_i$
 where $\al_1, \al_2$ are the roots of the characteristic polynomial of $T_\la$.
 \end{lemma}

\begin{proof}
Observe that for any  $\la\in \bC$, the operator  $T_\la$  acts on
each  linear space $Pol_n$ of all polynomials of degree at most
$n$, $n=0,1,2,\dots,$ and its matrix presentation
$(c_{ij})^n_{i,j=0}$ in the standard monomial basis $(1, z,
z^2,...,  z^n)$ of $Pol_n$  is an upper-triangular matrix with
diagonal entries $$c_{jj} = j(j-1)q_{22} +jq_{11} +q +(jq_{11}
+p)\la+q_{00}\la^2.$$  Therefore, for any given non-negative
integer $n$, we have a (unique) polynomial solution of
\eqref{eq:triv} of degree $n$ if and only if  $c_{nn} = 0$ but
$c_{jj}\neq 0$ for $0 \le j < n$. The asymptotic formula for
$\la_{i,n}$ follows from the form of the equation $c_{nn} = 0$.
The genericity assumption that the equations
$$n(n-1)q_{22}+nq_{11}+q+(nq_{11}+p)\la+q_{00}\la^2 =0 $$  and
$$j(j - 1)q_{22} + jq_{11} + q + (jq_{11} + p)\la + q_{00}\la^2=0$$
should not have a common root, for $0 \le j < n$ and $n$
sufficiently large, is clearly satisfied if we
assume that the characteristic equation does not have two roots with the same argument. 
\end{proof}


We can now prove  the following stronger result.

\begin{proposition}\label{pr:local}  For a general type family of differential operators
$T_\la$  of the form  \eqref{eq:pencil2},  all roots of all  polynomial solutions
of $T_\la(p) = 0,\; \la\in \bC$   are  located in some compact set
$K \subset  \bC$.
\end{proposition}

\begin{proof} Since  $T_\la$ is assumed to be of general type,  one gets  $Q_0\neq 0$.
Therefore, without loss of generality
we can assume that  $Q_0 = 1$ in \eqref{eq:triv}. Let $\{p_n\},
\deg(p_n) = n$ be a sequence of eigenpolynomials for
\eqref{eq:triv}, and assume that $\lim_{n\to \infty}
\frac{\la_n}{n}=\al$. (By Lemma~\ref{pr:basic},  $\al$ equals
either $\al_1$ or $\al_2$.) Define $w_n = \frac{p'_n}{\la_np_n}$
and notice that  $p_n = e^{\la_n\int w_n dz}$.
 We then have
$$p'_n = \la_nw_np_n;\; p^{\prime\prime}_n = (\la_n^2w_n^2 + \la_nw'_n)p_n.$$

Substituting these expressions in \eqref{eq:triv},  we obtain:
$$p_n(Q_2(z)(\la_n^2w_n^2(z) +\la_nw'_n(z)) + \la_n^2Q_1(z)w_n(z) + P_1(z)\la_nw_n(z) +\la^2_n + p\la_n + q
= 0.$$ For each fixed $n$,  near  $z = \infty$  we can  conclude
that
$$Q_2(z)(\la_n^2w_n^2(z)+\la_nw'_n (z))+\la_n^2Q_1(z)w_n(z)+P_1(z)\la_nw_n(z)+\la_n^2+p\la_n+q = 0.$$
This relation defines a rational function $w_n$ near infinity.  We
will show that the sequence   $\{w_n\}$  converges uniformly to an
analytic function $w$ in a sufficiently small disc around
$\infty$. Moreover $w$ does not vanish identically.
Proposition~\ref{pr:local}   will immediately follow from this
claim. Introducing $t = \frac{1}{z},$ one obtains
$$\widetilde Q_2\left(\left(\frac{w_n}{t}\right)^2-\frac{1}{\la_n}w'_n\right)+
\widetilde Q_1\left(\frac{w_n}{t}\right)
+\frac{1}{\la_n}\widetilde P_1\left(\frac{w_n}{t}\right)+ 1
+\frac{ p}{\la_n}+\frac{ q}{\la^2_n}= 0,$$ where $\widetilde
Q_2(t) := t^2Q_2(1/t),\; \widetilde Q_1(t) := tQ_1(1/t)$ and
$\widetilde P_1(t) := tP_1(1/t)$. Expand $w_n = c_1t + c_2t^2 +
...$  in a power series  around $\infty$, i.e. around $t = 0$. (By
a  slight abuse  of notation, we temporarily disregard  the fact
that  the coefficients $c_k$ depend on $n$  until we make their
proper estimate.) Set $(w_n/t)^2 = b_0 + b_1t +\dots$. Then  $$b_k
= c_1c_{k+1} + c_2c_k +...+ c_kc_2 + c_{k+1}c_1.$$ Finally,
introduce  $\eps_n = 1/\la_n$. Using these notations we obtain the
following system of recurrence relations for the coefficients
$c_k$:
$$q_{22}c^2_1 + (q_{11} - \eps_nq_{22} + \eps_np_{11})c_1 + 1 + \eps_n p + \eps^2_n q = 0,$$
$$q_{22}(b_1 - 2\eps_nc_2) + q_{12}(b_0 - \eps_nc_1) + (q_{11} + \eps_np_{11})c_2 +
(q_{01} + \eps_np_{01})c_1 = 0,$$
$$q_{22}(b_2-3\eps_nc_3)+q_{12}(b_1-2\eps_n c_2)+q_{02}(b_0-\eps_n c_1)+
(q_{11}+\eps_np_{11})c_3+(q_{01}+\eps_np_{01})c_2 = 0,$$
and, more generally,
$$q_{22}(b_k - (k + 1)\eps_n c_{k+1}) + q_{12}(b_{k-1} - k\eps_n c_k) +
q_{02}(b_{k-2} - (k-1)\eps_n c_{k-1})+
(q_{11} + \eps_np_{11})c_{k+1}$$ $$ + (q_{01} + \eps_np_{01})c_k =
0\quad \text{for} \quad k\ge 2.$$

Therefore, for any given $n$, we get $2$ possible values for
$c_1(n)$, which  tend to the roots of $q_{22}t^2 + q_{11}t + 1 =
0$ as $n\to \infty$. Notice that  $c_1(n) \to \frac{ 1}{\al}$ as
$n \to  \infty$. Choosing one  of two possible values for  $c_1,$
we uniquely determine the remaining  coefficients  (as
 rational functions of the previously calculated coefficients). Introducing
 $\tilde b_k =b_k - 2c_1c_{k+1},$  we can observe that $\tilde b_k$
 is independent of $c_{k+1}$ and we obtain the
following explicit formulas:
$$c_2 = -\frac{ q_{12}(c^2_1 - \eps_n c_1) + (q_{01} + \eps_np_{01})c_1}
{(2c_1 - 2\eps_n)q_{22} + q_{11} + \eps_np_{11}},$$
$$c_3 = -
\frac{q_{22}\tilde b_2 + q_{12}(b_1 - 2\eps_n c_2) + q_{02}(b_0 -
\eps_n c_1) + (q_{01} + \eps_np_{01})c_2} {(2c_1 - 3\eps_n)q_{22}
+ q_{11} + \eps_np_{11}},$$ and more generally,
\[ %
\begin{split} %
c_k =& - \frac{q_{22}\tilde b_{k-1} + q_{12}(b_{k-2} - (k -
1)\eps_n c_{k-1})}{(2c_1 - k\eps_n)q_{22} + q_{11} + \eps_n
p_{11}} \\  %
&+ \frac{q_{02}(b_{k-2} - (k - 3)\eps_n c_{k-3}) + (q_{01} +
\eps_n p_{01})c_{k-1}}{(2c_1 - k\eps_n)q_{22} + q_{11} +
\eps_n p_{11}}. %
\end{split} %
\] %
We will now  include  the dependence of $c_k$ on $n$ and show that
the coefficients $c_k(n)$ are majorated by the coefficients of a
convergent power series independent of  $n$.  First we show that
the denominators in these recurrence relations are bounded from
below. Notice that under our assumption,  the rational  functions
$w_n$ exist and have a power series expansion near $z = \infty$
with coefficients given by the above recurrence relations.
Therefore the denominators in these recurrences do not vanish.
Notice also that  $\eps_n \simeq \frac{c_1(n)}{n}$ asymptotically.
For fixed $k, $ it is therefore clear that the limits
$$\lim_{n\to\infty} (2c_1(n) - k\eps_n)q_{22} + q_{11} + \eps_np_{11} =
\lim_{n\to \infty} 2c_1(n)q_{22} + q_{11}$$
vanish if and only if  the characteristic polynomial
\eqref{eq:charac} has a double root. We must however find a
uniform bound for $c_k(n)$ valid  for all $k$ simultaneously.
Indeed, there might exist a subsequence $I \subset \NN$  of $k_n$
such that
\begin{equation}\label{eq:limit}
\lim_{n\in I;n\to \infty}(2c_1(n) - k_n\eps_n)q_{22} + q_{11} +
\eps_np_{11} = 0.
\end{equation}
(1) But this implies, using the asymptotics of $c_1(n)$ and
$\eps_n$, the existence of a real number $r$ such that $\frac{1 -
r}{\al} = -\frac{q_{22}}{2q_{11}}$  which is clearly impossible if
the characteristic equation does not have two roots with the same
argument.  Thus we have established a positive lower bound for the
absolute value of the denominators in the recurrence relations for
the coefficients $c_k$. The latter circumstance gives us a
possibility of majorizing the coefficients $c_k(n)$ independently
of $k$ and $n$. Namely, if there is a unbounded sequence
$k_n\eps_n,$  then we can factor  it out from the rational
functions in the recurrence. The existence of the sequence
mentioned above follow from an elementary lemma stated below,
which we leave without a proof. Thus, Proposition~\ref{pr:local}
is now settled.
\end{proof}

\begin{lemma}\label{lm:add}  Consider a recurrence relation $c_{m+1} = P_m(c_1,..., c_m)$ where each
$P_m$ is a polynomial and assume that $d_{m+1} = Q_m(d_1,...,
d_m)$  is a similar recurrence relation whose  polynomials have
all positive coefficients. If the polynomials under consideration
satisfy the inequalities
$$|P_m(z_1,..., z_m)| \le  Q_m(|z_1|,..., |z_m|),$$
then the power series $\sum c_iz^i$ is dominated by the series
$\sum d_iz^i$ whenever $d_1\ge |c_1|$.
\end{lemma}

\section{Domain configurations of normalized quadratic differentials} \label{Section-6}
\setcounter{equation}{0}

Let $Q(z;a,b,c)\,dz^2$ be a quadratic differential of the form
(\ref{1.4}). Multiplying $Q(z;a,b,c)\,dz^2$ by a non-zero constant
$A\in \mathbb{C}$, we rescale the corresponding $Q$-metric
$|Q|^{1/2}\,|dz|$ by a positive constant $|A|^{1/2}$. Hence
$A\,Q(z;a,b,c)\,dz^2$ has the same geodesics as the quadratic
differential $Q(z;a,b,c)\,dz^2$ has. Obviously, multiplication
does not affect the homotopic classes.
 Thus, while studying geodesics of the quadratic differential $Q(z;a,b,c)\,dz^2$,
 we may assume without loss of generality that it
has the form %
\begin{equation} \label{6.1} %
Q(z)\,dz^2=-\frac{(z-p_1)(z-p_2)}{(z-1)^2(z+1)^2}\,dz^2. %
\end{equation} %
In Sections 6--9, we will work  with the generic case; i.e   we assume that %
\begin{equation} \label{6.1.1} %
p_1\not= \pm 1, \quad p_2\not= \pm 1, \quad p_1\not = p_2, %
\end{equation} %
unless otherwise is mentioned. Some typical configurations in the
limit (or non-generic) cases are shown in Fig.~5a--5g. Expanding
$Q(z)$ into Laurent series at $z=\infty$, we
obtain %
\begin{equation} \label{6.2} %
Q(z)=-\frac{1}{z^2}+{\mbox{higher degrees of $z$}} \quad  \quad
{\mbox{as $z\to \infty$.}}
\end{equation} %

Since the leading coefficient in the series expansion (\ref{6.2})
is real and negative it follows that  $Q(z)\,dz^2$ has a circle
domain $D_\infty$ centered at $z=\infty$. The boundary
$L_\infty=\partial D_\infty$ of $D_\infty$ consists of a finite
number of critical trajectories of the quadratic differential
$Q(z)\,dz^2$ and therefore $L_\infty$ contains at least one of the
zeros $p_1$ and $p_2$ of $Q(z)\,dz^2$.

Next, we will discuss possible trajectory structures of
$Q(z)\,dz^2$  on the complement $D_0=\mathbb{C}\setminus
\overline{D}_\infty$.  As we have mentioned in Section~3,
according to the Basic Structure Theorem, \cite[Theorem~3.5]{Je},
the domain configuration of a quadratic differential $Q(z)\,dz^2$
on $\overline{\mathbb{C}}$, which will be denoted by
$\mathcal{D}_Q$, may include circle domains, ring domains, strip
domains, end domains, and density domains. For  the quadratic
differential (\ref{6.1}),  by the Three Pole Theorem \cite[Theorem
3.6]{Je}, there are no density domains in its domain configuration
$\mathcal{D}_Q$. In addition, since $Q(z)\,dz^2$ has only three
poles of order two each, the domain configuration $\mathcal{D}_Q$
does not contain end domains and may contain at most three circle
domains centered at $z=\infty$, $z=-1$, and $z=1$.

We note here that $\mathcal{D}_Q$ may have strip domains (also
called  \emph{bilaterals}) with vertices at the double poles
$z=-1$ and $z=1$ but $\mathcal{D}_Q$ does not have ring domains.
Indeed, if there were a ring domain $\widehat{D}\subset D_0$ with
boundary components $l_1$ and $l_2$ then, by the Basic Structure
Theorem, each component must contain a zero of $Q(z)\,dz^2$. In
particular, $p_1\not=p_2$ in this case. Suppose that $l_1$
contains a zero $p_1$ and that $p_1\in L_\infty$. Then $L_\infty$
contains a critical trajectory $\gamma'$, which has both its end
points at $p_1$. There is one more critical trajectory $\gamma''$,
which has one of its end points at $p_1$. This trajectory
$\gamma''$ is either lies on the boundary of the circle domain
$D_\infty$ or it lies on the boundary of the ring domain
$\widehat{D}$. Therefore the second end point of $\gamma''$ must
be at a zero of $Q(z)\,dz^2$. Since the only remaining zero is
$p_2$, which lies on the boundary component $l_2$ not intersecting
$l_1$, we obtain a contradiction with our assumption. The latter
shows that $\mathcal{D}_Q$ does not have ring domains.

\smallskip

 Next, we will classify topological types of domain configurations according
to the number of circle domains in $\mathcal{D}_Q$. The first
digit in our further classifications stands for the section where
this classification is introduced. The second and further digits
will denote the case under consideration.

\textbf{6.1.} Assume first that $\mathcal{D}_Q$ contains three
circle domains $D_\infty\ni \infty$, $D_{-1}\ni -1$, and $D_1\ni
1$. Then, of course, there are  no strip domains in
$\mathcal{D}_Q$. In this case, the domains $D_\infty,D_{-1},D_1$
constitute an extremal configuration of the Jenkins extremal
problem for the weighted sum of reduced moduli with appropriate
choice of positive weights $\alpha_\infty$, $\alpha_{-1}$, and
$\alpha_1$; see, for example, \cite{Str}, \cite{S1}, \cite{S2}.
More precisely, the
problem is to find  all possible  configurations realizing the following maximum: %
\begin{equation} \label{6.3} %
\max \
\left(\alpha_\infty^2m(B_\infty,\infty)+\alpha_{-1}^2m(B_{-1},-1)+\alpha_1^2m(B_1,1)\right) %
\end{equation} %
 over all triples of non-overlapping
simply connected domains $B_\infty\ni \infty$, $B_{-1}\ni -1$, and
$B_1\ni 1$. Here, $m(B,z_0)$ stands for  the reduced module of a
simply connected domain $B$ with respect to the point $z_0\in B$;
see \cite[p.24]{Je}.

Since the extremal configuration of problem~(\ref{6.3}) is unique
it follows that the domains $D_\infty$, $D_{-1}$, and $D_1$ are
symmetric with respect to the real axis. In particular, the zeros
$p_1$ and $p_2$ are either both real or they are complex
conjugates of each other. Of course, this symmetry property of
zeros can be derived directly from the fact that the leading
coefficient of the Laurent expansion of $Q(z)$ at each its pole is
negative in the case under consideration. We have three
essentially different possible positions for the zeros: %
\begin{enumerate} %
\item[\textbf{(a)}] %
$-1<p_2<p_1<1$,

\item[\textbf{(b)}] %
$1<p_2<p_1$ or $p_1<p_2<-1$, %
\item[\textbf{(c)}] %
$p_1=\overline{p}_2=p$, where $\Im p>0$.
\end{enumerate} %

We note here that in the case when $-1<p_2<1$ and, in addition,
$p_1>1$ or $p_1<-1$ the domain configuration $\mathcal{D}_Q$ must
contain a strip domain.

\medskip

Case \textbf{(a)}. %
The trajectory structure of $Q(z)\,dz^2$ corresponding to this
case is shown in Fig.~1a. There are three critical trajectories:
$\gamma_{-1}$, which is on the boundary of $D_{-1}$ and  has both
its end points at $z=p_2$; $\gamma_1$, which is on the boundary of
$D_1$ and has both its end points at $z=p_1$, and $\gamma_0$,
which is the segment $[p_2,p_1]$.

\medskip

Case \textbf{(b)}. An example of a domain configuration for the
case $1<p_2<p_1$ is shown in Fig.~1b. The boundary of $D_1$
consists of a single critical trajectory $\gamma_1$ having  both
end points at $p_2$.  The boundary of $D_{-1}$ consists of
critical trajectories $\gamma_\infty$, $\gamma_1$, and $\gamma_0$,
which is the segment $[p_2,p_1]$. In the case $p_1<p_2<-1$, the
domain configuration is similar.

Case \textbf{(c)}. %
Since the domain configuration is symmetric, $p_1$ and $p_2$ both
belong to the boundary of $D_\infty$. Furthermore, there are three
critical trajectories: $\gamma_{-1}$, which joins $p_1$
 and $p_2$ and intersects the real axis at some point $d_{-1}<-1$, $\gamma_1$, which joins $p_1$
 and $p_2$ and intersects the real axis at some point $d_1>1$, and $\gamma^0$, which joins $p_1$
 and $p_2$ and intersects the real axis at some point $d_0$,
 $-1<d_0<1$. In this case, $\gamma_1\cup \gamma_0\subset \partial D_1$,
 $\gamma_{-1}\cup \gamma_0\subset \partial D_{-1}$. An example of
 a domain configuration of this type is shown in Fig.~1c. %

\medskip

\textbf{6.2.} Next we consider the case when $\mathcal{D}_Q$ has
exactly  two circle domains. Suppose that these domains are
$D_\infty\ni \infty$ and $D_{-1}\ni -1$. In this case it is not
difficult to see that $L_\infty$ contains exactly one zero.
Indeed, if $p_1,p_2\in L_\infty$, then $L_\infty$ must contain one
or two critical trajectories joining $p_1$ and $p_2$. Suppose that
$L_\infty$ contains one such trajectory, call it $\gamma_0$. Since
$p_1,p_2\in L_\infty$ the boundary of $D_\infty$ must contain a
trajectory $\gamma_1$, which has both its end points at $p_1$ and
a trajectory $\gamma_{-1}$, which has both its end points at
$p_2$. Thus, $\gamma_1\cup\{p_1\}$ and $\gamma_{-1}\cup\{p_2\}$
each surrounds a simply connected domain, which must contain a
critical point of $Q(z)\,dz^2$. This implies that $z=-1$ and $z=1$
are centers of circle domains of $Q(z)\,dz^2$, which is the case
considered in part \textbf{ 6.1(a)}.

If $L_\infty$ contains two critical trajectories joining $p_1$ and
$p_2$, then there are critical trajectories $\gamma'$ having one
of its end points at $p_1$ and $\gamma''$ having one of its end
points at $p_2$. If $\gamma'=\gamma''$, then $D_0\setminus
\gamma'$ consists of two simply connected domains, which in this
case must be circle domains of $Q(z)\,dz^2$ as it is shown in
Fig.~1c.

If $\gamma'\not=\gamma''$, then each of these trajectories must
have its second end point at one of the poles $z=-1$ or $z=1$.
Moreover, if $\gamma'$ has an end point at $z=-1$ then $\gamma''$
must have its end point at $z=1$. Thus, there is no second circle
domain of $Q(z)\,dz^2$ in this case. Instead, there is one circle
domain $D_\infty$ and a strip domain, call it  $G_2$, as it shown
in Fig. 3a-3e.

 Now, let  $p_1$ be the only zero of $Q(z)\,dz^2$ lying on $L_\infty$.
Then  $L_\infty$ consists of a single critical trajectory of
$Q(z)\,dz^2$, call it $\gamma_\infty$, together with its end
points, each of  which is at $p_1$. There is one more critical
trajectory, call it $\gamma_1^+$, that has one of its end points
at $p_1$. Then the second end point of $\gamma_1^+$ is either at
the point $p_2$ or at the second order pole at $z=1$.

If $\gamma_1^+$ terminates at $p_2$, then there is one more
critical trajectory, call it $\gamma_2$, having one of its end
points at $p_2$.  Since $D_{-1}$ is a circle domain and $\partial
D_{-1}$ contains at least one zero of $Q(z)\,dz^2$ it follows that
$\gamma_2$ belongs to the boundary  of $D_{-1}$.  Since $\gamma_2$
lies on the boundary of $D_{-1}$ it have to terminate at a finite
critical point of $Q(z)\,dz^2$ and the only possibility for this
is that $\gamma_2$ terminates at $p_2$. In this case,
$\gamma_\infty$, $\gamma_1^+$, and $\gamma_2$ divide
$\overline{\mathbb{C}}$ into three circle domains, the case which
was already discussed in part \textbf{6.1(b)}.

\medskip

Suppose that $\gamma_1^+$ joins the points $z=p_1$ and $z=1$. Then
$\mathcal{D}_Q$ contains a strip domain $G_1$. Since $z=1$ is the
only second order pole of $Q(z)\,dz^2$, which has a non-negative
non-zero leading coefficient, the strip domain $G_1$ has both its
vertices at the point $z=1$. Furthermore, one side of $G_1$
consists of two critical trajectories $\gamma_\infty$ and
$\gamma_1^+$. Therefore there is a critical trajectory, call it
$\gamma_1^-$ of $Q(z)\,dz^2$ lying on $\partial G_1$, which joins
$z=1$ and $z=p_2$. Now, the remaining possibility is that the
boundary of $D_{-1}$ consists of a single critical trajectory
$\gamma_{-1}$, which has both its end points at $p_2$. Then $G_1$
is the only strip domain in $\mathcal{D}_Q$ and the second side of
$G_1$ consists of the critical trajectories $\gamma_1^-$ and
$\gamma_{-1}$. Two examples of a domain configuration of this
type, symmetric and non-symmetric,  are shown in Fig.~2a and
Fig.~2b.

\medskip

\textbf{6.3.} Finally,  we consider the case when $D_\infty$ is
the only circle domain of $Q(z)\,dz^2$. We consider two
possibilities.

Case \textbf{(a)}. Suppose that both zeros $p_1$ and $p_2$ belong
to the boundary of $D_\infty$. As we have found in part\textbf{
6.2} above, the domain configuration in this case consists of the
circle domain $D_\infty$ and the strip domain $G_2$. The boundary
of $D_\infty$ consists of two critical trajectories
$\gamma_\infty^+$ and $\gamma_\infty^-$ and their end points,
while the boundary of $G_2$ consists of the trajectories
$\gamma_\infty^+$, $\gamma_\infty^-$, $\gamma_1$, and
$\gamma_{-1}$ and their end points, as it is shown in Fig.~3a-3c.

Case \textbf{(b)}. Suppose that  the boundary $L_\infty$ of
$D_\infty$ contains only one zero $p_1$. Then there is a critical
trajectory $\gamma_\infty$ having both its end points at $p_1$
such that $L_\infty=\gamma_\infty\cup \{p_1\}$. Since $p_1$ is a
simple zero of $Q(z)\,dz^2$ there is one more critical trajectory
having one of its end points at $p_1$. The second end point of
this trajectory is either at the pole $z=1$, or at the pole
$z=-1$, or at the zero $z=p_2$. Depending on which of these
possibilities is realized, this trajectory will be denoted by
$\gamma_1$, or $\gamma_{-1}$, or $\gamma_0$, respectively. Thus,
we have two essentially different subcases. %

Case \textbf{(b1)}. Suppose that there is a critical trajectory
$\gamma_0$ joining the zeros $p_1$ and $p_2$. Then there are two
critical trajectories, call them $\gamma_1$ and $\gamma_{-1}$,
each of which has one of its end point at $p_2$. We note that
$\gamma_1\not=\gamma_{-1}$. Indeed, if $\gamma_1=\gamma_{-1}$,
then the closed curve $\gamma_1\cup \{p_2\}$ must enclose a
bounded circle domain of $Q(z)\,dz^2$, which does not exist.
Furthermore, $\gamma_1$ and $\gamma_{-1}$ both cannot have their
second end points at the same pole at $z=1$ or $z=-1$. If this
occurs then again $\gamma_1$ and $\gamma_{-1}$ will enclose a
simply connected domain having a single pole of order $2$ on its
boundary, which is not possible. The remaining possibility is that
one of these critical trajectories, let assume that $\gamma_1$,
joins the zero $z=p_2$ and the pole at $z=1$ while $\gamma_{-1}$
joins $z=p_2$ and $z=-1$. %

In this case the domain configuration $\mathcal{D}_Q$ consists of
the circle domain $D_\infty$ and the strip domain $G_2$; see
Fig.~3d- and Fig.~3e. The boundary of $G_2$ consists of two sides,
call them $l_1$ and $l_2$. The side $l_1$ is the set of boundary
points of $G_2$ traversed by the point $z$ moving along $\gamma_1$
from $z=1$ to $z=p_2$ and then along $\gamma_{-1}$ from the point
$z=p_2$ to $z=-1$. The side $l_2$ is the set of boundary points of
$G_2$ traversed by the point $z$ moving along $\gamma_1$ from
$z=1$ to $z=p_2$, then along $\gamma_0$ from $z=p_2$ to $z=p_1$,
then along $\gamma_\infty$ from $z=p_1$ to the same point $z=p_1$,
then along $\gamma_0$ from $z=p_1$ to $z=p_2$, and finally along
$\gamma_{-1}$ from $z=p_2$ to $z=-1$.

Case \textbf{(b2)}. Suppose that there is a critical trajectory
$\gamma_1$ joining the zero $p_1$ and the pole $z=1$. Then there
is a strip domain, call it $G_1$, which has both its vertices at
the pole $z=1$ and has the critical trajectories $\gamma_1$ and
$\gamma_\infty$ on one of its sides, call it $l_1^1$. More
precisely, the side $l_1^1$ is the set of boundary points of $G_1$
traversed by the point $z$ moving along $\gamma_1$ from $z=1$ to
$z=p_1$,  then along $\gamma_\infty$ from $z=p_1$ to the same
point $z=p_1$, and then along $\gamma_1$ from $z=p_1$ to $z=1$.

Let $l_1^2$ denote the second side of $G_1$. Since a side of a
strip domain always has a finite critical point it follows that
$l_1^2$ contains two critical trajectories, call them $\gamma_0^+$
and $\gamma_0^-$, which join the pole $z=1$ with zero $z=p_2$.
There is one critical trajectory of $Q(z)\,dz^2$, call it
$\gamma_{-1}$, which has one of its end points at $z=p_2$. Since
$z=-1$ is a second order pole, which is not the center of a circle
domain, there should be at least one critical trajectory of
$Q(z)\,dz^2$ approaching $z=-1$ at least in one direction. Since
the end points of all critical trajectories, except $\gamma_{-1}$,
are already identified and they are not at $z=-1$, the remaining
possibility is that $\gamma_{-1}$ has its second end point at
$z=-1$. In this case there is one more strip domain, call it
$G_2$, which has vertices at the poles $z=1$ and $z=-1$ and sides
$l_2^1$ and $l_2^2$. Two examples of configurations with one
circle domain and two strip domains, symmetric and non-symmetric,
are shown in Fig.~4a and Fig.~4b. Now we can identify all sides of
$G_1$ and $G_2$. The side $l_1^2$ is the set of boundary points of
$G_1$ traversed by the point $z$ moving along $\gamma_0^+$ from
$z=1$ to $z=p_2$ and then along
$\gamma_0^-$ from $z=p_2$ to  $z=1$. %
The side $l_2^1$ is the set of boundary points of $G_2$ traversed
by the point $z$ moving along $\gamma_0^+$ from $z=1$ to $z=p_2$
and then along $\gamma_{-1}$ from $z=p_2$  to $z=-1$. Finally, the
side $l_2^2$ is the set of boundary points of $G_2$ traversed by
the point $z$ moving along $\gamma_0^-$ from $z=1$ to $z=p_2$ and
then along $\gamma_{-1}$ from $z=p_2$  to $z=-1$; see Fig.~4a and
Fig.~4b.

Case \textbf{(b3)}. In the case when there is a critical
trajectory  joining the zero $p_1$ and the pole $z=-1$, the domain
configuration is similar to one described above, we just have to
switch the roles of the poles at $z=1$ and $z=-1$.


\begin{remark}  \label{Remark-2} %
We have described above all possible configurations in the generic
case; i.e.  under conditions (\ref{6.1.1}). The remaining special
cases can be obtained from the generic case as limit cases when
$p_2\to -1$, when $p_2\to p_1$; etc. In the case $p_1=p_2$,
possible configurations are shown in Fig.~5a-5c.

In the case when $p_2=-1$, $p_1\not=\pm 1$, possible
configurations are shown in Fig.~5d-5g.

In the case when $p_1=p_2=1$, the limit position of critical
trajectories is just a circle centers at $z=-1$ with radius
$2$configuration and in the case  when $p_1=1$, $p_2=-1$ there is
one critical trajectory which is an open interval from $z=-1$ to
$z=1$.
\end{remark} %

\section{How parameters determine the type of  domain configuration} %
\setcounter{equation}{0}

Our  goal in this section is to identify the ranges of the
parameters $p_1$ and $p_2$ corresponding to topological types
discussed in Section~6. For a fixed $p_1$ with $\Im p_1\not=0$, we
will define four regions of the parameter $p_2$. These regions and
their boundary arcs will correspond to domain configurations with
specific properties; see Fig.~6.

It will be useful to introduce the following notation. For $a\in
\mathbb{C}$ with $\Im a\not=0$, by $L(a)$ and $H(a)$ we denote,
respectively, an ellipse and hyperbola with foci at $z=1$ and
$z=-1$, which pass through the point $z=a$. If $\Im a\not= 0$,
then the set $\mathbb{C}\setminus (L(a)\cup H(a))$ consists of
four connected components, which will be denoted by $E_1^+(a)$,
$E_1^-(a)$, $E_{-1}^+(a)$, and $E_{-1}^-(a)$. We assume here that
$1\in E_1^+(a)$, $-1\in E_{-1}^+(a)$, $E_1^-(a)\cap
\mathbb{R}_+\not= \emptyset$, and $E_{-1}^-(a)\cap
\mathbb{R}_-\not=\emptyset$. Furthermore, assuming that $\Im
a\not=\emptyset$, we define the following open arcs:
$L^+(a)=(L(a)\cap
\partial E_1^+(a))\setminus\{a,\bar a\}$, $L^-(a)=(L(a)\cap \partial E_{-1}^+(a))
\setminus\{a,\bar a\}$, $H^+(a)=(H(a)\cap
\partial E_1^+(a))\setminus\{a,\bar a\}$, $H^-(a)=(H(a)\cap \partial E_1^-(a))
\setminus\{a,\bar a\}$. Let $l_1(a)$ and $l_{-1}(a)$ be straight
lines passing through the points $1$ and $\bar a$ and $-1$ and
$\bar a$, respectively. Let $l_1^+(a)$ and $l_{-1}^+(a)$ be open
rays issuing from the points $z=1$ and $z=-1$, respectively, which
pass through the point $z=\overline{a}$ and let $l_1^-(a)$ and
$l_{-1}^-(a)$ be their complementary rays. The line $l_1(a)$
divides $\mathbb{C}$ into two half-planes, we call them  $P_1$ and
$P_2$ and  enumerate such that $P_1\ni 2$. Similarly, the line
$l_{-1}(a)$ divides $\mathbb{C}$ into two half-planes  $P_3$ and
$P_4$, where $P_3\ni -2$.

Before we state the main result of this section, we recall the
reader that  the local structure of trajectories near a pole $z_0$
is completely determined by the leading coefficient of the Laurent
expansion of $Q(z)$ at $z_0$, see \cite[Ch.~3]{Je}.
 In particular, for the quadratic differential $Q(z)\,dz^2$ defined by (\ref{6.1}) we have  %
\begin{equation} \label{7.1.1}  %
Q(z)=-\frac{1}{4}\frac{C_1}{ (z-1)^2}+{\mbox{higher degrees of
$(z- 1)$ \quad as $z\to 1$}}
\end{equation} %
and %
$$ 
Q(z)=-\frac{1}{4}\frac{C_{-1}}{ (z+1)^2}+{\mbox{higher degrees of
$(z+ 1)$ \quad as $z\to -1$.}}
$$ 
Then, assuming that $p_1\not= \pm 1$, $p_2\not= \pm 1$, we find %
\begin{equation}  \label{5.1} %
C_1=(p_1-1)(p_2-1)\not= 0 \quad {\mbox{and}} \quad C_{-1}=(p_1+1)(p_2+1)\not =0.%
\end{equation} %

A complete description of sets of pairs $p_1$, $p_2$ with $\Im
p_1>0$ corresponding to all possible types of domain
configurations discussed in Section~6 is given by the following
theorem. %

\begin{theorem} \label{Theorem 5.1} %
Let $p_1$ with $\Im p_1>0$ be fixed. Then the following holds. %

\textbf{7.A}. The types of domain configurations $\mathcal{D}_Q$
correspond to the following sets of the parameter~$p_2$.
\begin{enumerate} %
\item[\textbf{(1)}] %
If $p_2=\bar p_1$, then the domain configuration $\mathcal{D}_Q$
is of the type \textbf{6.1(c)}.

\item[\textbf{(2)}] %
If $p_2\in l_1^+(p_1)\setminus\{\bar p_1\}$, then $\mathcal{D}_Q$
has the type~\textbf{6.2} with circle domains $D_\infty\ni\infty$
and $D_1\ni 1$. Furthermore, if $p_2\in l_1^+(p_1)\cap
E_1^+(p_1)$, then $p_1\in \partial D_\infty$ and if $p_2\in
l_1^+(p_1)\cap E_{-1}^-(p_1)$, then $p_2\in \partial D_\infty$.

If $p_2\in l_{-1}^+(p_1)\setminus\{\bar p_1\}$, then
$\mathcal{D}_Q$ has the type~\textbf{6.2} with circle domains
$D_\infty\ni\infty$ and $D_{-1}\ni -1$. Furthermore, if $p_2\in
l_{-1}^+(p_1)\cap E_{-1}^+(p_1)$, then $p_1\in \partial D_\infty$
and if $p_2\in l_{-1}^+(p_1)\cap E_{-1}^-(p_1)$, then $p_2\in
\partial D_\infty$.
\item[\textbf{(3a)}] %
If $p_2\in L(a)\setminus\{p_1,\bar p_1\}$, then the domain
configuration $\mathcal{D}_Q$ has type \textbf{6.3(a)}.
Furthermore, if $p_2\in L^+(p_1)$, then there is a critical
trajectory having one end point at $p_2$, which in other direction
approaches the pole $z=1$. Similarly, if $p_2\in L^-(p_1)$, then
there is a critical trajectory having one end point at $p_2$,
which in other direction
approaches the pole $z=-1$. %
\item[\textbf{(3b1)}] %
If $p_2\in H(p_1)\setminus\{p_1,\bar p_1\}$, then $\mathcal{D}_Q$
has type \textbf{6.3(b1)}. Furthermore, if $p_2\in H^+(p_1)$, then
there is a critical trajectory having both  end points at $p_1$.
If $p_2\in H^-(p_1)$, then there is a critical trajectory having
both end points at $p_2$.

\item[\textbf{(3b2)}] %
In all remaining cases, i.e. if $p_2\not\in L(p_1)\cup H(p_1)\cup
l_1^+(p_1)\cup l_{-1}^+(p_1)\cup\{-1,1\}$, the domain
configuration $\mathcal{D}_Q$
belongs to type \textbf{6.3(b2)}. %
Furthermore, if $p_2\in (E_1^+(p_1)\cup E_{-1}^+(p_1))\setminus
(l_1^+(p_1)\cup l_{-1}^+(p_1)\cup\{-1,1\})$, then $p_1\in \partial
D_\infty$ and if $p_2\in (E_1^-(p_1)\cup E_{-1}^-(p_1))\setminus
(l_1^+(p_1)\cup l_{-1}^+(p_1))$, then $p_2\in \partial D_\infty$.

In addition, if $p_2\in E_1^+(p_1)\setminus
(l_1^+(p_1)\cup\{1\})$, then the pole $z=1$ attracts only one
critical trajectory of the quadratic differential (\ref{6.1}),
which has its second end point at $z=p_2$ and if $p_2\in
E_{-1}^-(p_1)\setminus (l_1^+(p_1))$, then the pole $z=1$ attracts
only one critical trajectory of the quadratic differential
(\ref{6.1}), which has its second end point at $z=p_1$. If $p_2\in
E_{-1}^+(p_1)\setminus (l_{-1}^+(p_1)\cup\{-1\})$, then the pole
$z=-1$ attracts only one critical trajectory of the quadratic
differential (\ref{6.1}), which has its second end point at
$z=p_2$ and if $p_2\in E_1^-(p_1)\setminus (l_{-1}^+(p_1))$, then
the pole $z=-1$ attracts only one critical trajectory of the
quadratic differential (\ref{6.1}), which has its second end point
at $z=p_1$. %
\end{enumerate} %

\textbf{7.B}. The local behavior of the trajectories near the
poles $z=1$ and $z=-1$ is controlled  by the position of the zero
$p_2$ with respect to the lines $l_1(p_1)$ and $l_{-1}(p_1)$.
Precisely,
we have the following possibilities. %
\begin{enumerate} %
\item[\textbf{(1)}] %
If $p_2\in l_1^-(p_1)$ or, respectively,  $p_2\in l_{-1}^-(p_1)$,
then $Q(z)\,dz^2$ has radial structure of trajectories near the
pole $z=1$ or, respectively,  near the pole $z=-1$.

\item[\textbf{(2)}] %

If $p_2\in P_1$ or, respectively, $p_2\in P_2$, then the
trajectories of $Q(z)\,dz^2$ approaching the pole $z=1$ spiral
counterclockwise or, respectively, clockwise.

If $p_2\in P_3$ or, respectively, $p_2\in P_4$, then the
trajectories of $Q(z)\,dz^2$ approaching the pole $z=-1$ spiral
counterclockwise or, respectively, clockwise. %
\end{enumerate} %
\end{theorem} %

\noindent %
\emph{Proof}. %
 \textbf{7.A(1).} We have shown in Section~6 that a domain
 configuration $\mathcal{D}_Q$ of the type~\textbf{6.1(c)} occurs if
 and only if $p_2=\bar p_1$. Thus, we have to consider cases \textbf{7.A(2)}
 and \textbf{7.A(3)}.
 We first prove statements about positions of zeros $p_1$ and
 $p_2$ for each of these cases. Then we will turn to statements
 about critical trajectories.

 \textbf{7.A(2).} A domain configuration $\mathcal{D}_Q$ contains
 exactly two circle domains centered at $z=\infty$ and $z=-1$ if
 and only if
 $C_{-1}>0$ and $C_1$ is not a positive real number. This is equivalent to the following conditions:%
\begin{equation} \label{5.2} %
\arg (p_1+1)=-\arg(p_2+1) \quad \mod(2\pi), %
\end{equation} %
\begin{equation} \label{5.3} %
\arg (p_1-1)\not=-\arg(p_2-1) \quad \mod(2\pi). %
\end{equation} %

Geometrically, equations (\ref{5.2}) and (\ref{5.3}) mean that the
points $p_1$ and $p_2$ lie on the rays issuing from the pole
$z=-1$, which are symmetric to each other with respect to the real
axis. Furthermore, each ray contains one of these points and
$p_1\not= \bar p_2$.

Assuming (\ref{5.2}), (\ref{5.3}), we claim that $p_1\in
\partial D_\infty$ if and only if $|p_2+1|<|p_1+1|$.
First we prove that the claim is true for all $p_2$ sufficiently
close to $z=-1$ if $p_1$ is fixed. Arguing by contradiction,
suppose that there is a sequence $s_k\to -1$ such that
$\arg(s_k+1)=-\arg(p_1+1)$ and $p_1\in
\partial D_{-1}^k$, $s_k\in \partial D_\infty^k$  for all $k=1,2,\ldots$ Here
$D_{-1}^k\ni -1$ and $D_\infty^k\ni \infty$ denote the
corresponding circle domains of the quadratic
differential %
 \begin{equation} \label{5.4} %
 Q_k(z)\,dz^2=-\frac{(z-p_1)(z-s_k)}{(z-1)^2(z+1)^2}\,dz^2. %
 \end{equation} %

Changing variables in (\ref{5.4}) via $z=(s_k+1)\zeta-1$ and then
dividing the resulting quadratic differential by
$\delta_k=|s_k+1|$, we obtain the following quadratic
differential: %
\begin{equation} \label{5.5} %
\widehat{Q}_k(\zeta)\,d\zeta^2=\frac{\zeta-1}{\zeta^2}\frac{|1+p_1|-\delta_k^{-1}(s_k+1)^2\zeta}{(2-(s_k+1)\zeta)^2}
\,d\zeta^2. %
\end{equation} %
We note that the trajectories of $Q_k(z)\,dz^2$ correspond under
the mapping $z=(s_k+1)\zeta-1$ to the trajectories of the
quadratic differential $\widehat{Q}_k(\zeta)\,d\zeta^2$. Thus,
$\widehat{Q}_k(\zeta)\,d\zeta^2$ has two circle domains
$\widehat{D}_{k,\infty}\ni \infty$ and $\widehat{D}_{k,0}\ni 0$.
The zeros of $\widehat{Q}_k(\zeta)\,d\zeta^2$ are at the points  %
\begin{equation} \label{5.6} %
\zeta'_k=1\in \partial \widehat{D}_{k,\infty}, \quad
\zeta''_k=\delta_k|1+p_1|(s_k+1)^{-2}\in \partial
\widehat{D}_{k,0}. %
\end{equation} %

From (\ref{5.5}), we find that %
\begin{equation} \label{5.7} %
\widehat{Q}_k(\zeta)\,d\zeta^2\to
\widehat{Q}(\zeta)\,d\zeta^2:=\frac{|1+p_1|}{4}\frac{\zeta-1}{\zeta^2}\,d\zeta^2, %
\end{equation} %
where convergence is uniform on compact subsets of
$\mathbb{C}\setminus\{0\}$. Since %
$$ %
\widehat{Q}(\zeta)=-(|1+p_1|/4)\zeta^{-2}+\cdots \quad {\mbox{ as
$\zeta\to 0$}}%
$$ %
the quadratic differential  $\widehat{Q}(\zeta)\,d\zeta^2$ has a
circle domain $\widehat{D}$ centered at $\zeta=0$. Let
$\widehat{\gamma}$ be a trajectory of
$\widehat{Q}(\zeta)\,d\zeta^2$ lying in $\widehat{D}$ and let
$\hat\gamma_k$ be an arbitrary trajectory of
$\widehat{Q}_k(\zeta)\,d\zeta^2$ lying in the circle domain
$\widehat{D}_{k,0}$.  Since
$\hat\gamma_k$ is a $\widehat{Q}_k$-geodesic in its class and by (\ref{5.7}) we have %
\begin{equation} \label{5.8} %
|\hat\gamma_k|_{\widehat{Q}_k}\le
|\widehat{\gamma}|_{\widehat{Q}_k}\to
|\widehat{\gamma}|_{\widehat{Q}}=|1+p_1|^{1/2} \quad {\mbox {as
$k\to \infty$.}}
\end{equation} %
On the other hand, conditions (\ref{5.6}) imply that for every
$R>1$ there is $k_0$ such that for every $k\ge k_0$ there is an
arc $\tau_k$ joining the circles $\{\zeta:\,|\zeta|=1\}$ and
$\{\zeta:\,|\zeta|=R\}$, which lies on regular trajectory of the
quadratic differential $\widehat{Q}_k(\zeta)\,d\zeta^2$ lying in
the circle domain $\widehat{D}_{k,0}$. Then, using (\ref{5.5}), we
conclude that there is a constant $C>0$ independent on $R$ and
$k$ such that %
$$ %
|\hat\gamma_k|_{\widehat{Q}_k}\ge
|\tau_k|_{\widehat{Q}_k}=\int_{\tau_k}\left|\widehat{Q}_k(\zeta)\right|^{1/2}\,|d\zeta|\ge
C\,\int_1^R \frac{\sqrt{|\zeta|-1}}{|\zeta|}\,d|\zeta| %
$$ %
for all  $k\ge k_0$. Since  $\int_1^R x^{-1}\sqrt{x-1}\,dx\to
\infty$ as $R\to \infty$, the latter equation contradicts equation
(\ref{5.8}). Thus, we have proved that if $p_1$ is fixed and $p_2$
is sufficiently close to $z=-1$ then $p_1\in \partial D_\infty$
and $p_2\in \partial D_{-1}$.

Now, we fix $p_1$ with $\Im p_1\not =0$ and  consider the set $A$
consisting of all points $p'_2$ on the ray $r=\{z: \, \arg(z+1)
=-\arg(p_1+1)\}$ such that $p_1\in
\partial D_\infty(p_1,p_2)$ and $p_2\in \partial
D_{-1}(p_1,p_2)$ for all $p_2\in r$ such that %
$|p_2+1|<|p'_2+1|$. Here $D_\infty(p_1,p_2)$ and $D_{-1}(p_1,p_2)$
are corresponding circle domains of the quadratic differential
(\ref{6.1}). Our argument above shows that $A\not= \emptyset$. Let
$p_2^m\in r$ be such that %
$$ %
|p_2^m+1|=sup_{p_2\in A} \,|p_2+1|. %
$$ %

Consider the quadratic differential $Q(z;p_1,p_2^m)\,dz^2$ of the
form (\ref{6.1}) with $p_2$ replaced by $p_2^m$. Let
$D_\infty(p_1,p_2^m)\ni \infty$ and $D_{-1}(p_1,p_2^m)\ni -1$ be
the corresponding circle domains of $Q(z;p_1,p_2^m)\,dz^2$. Since
the quadratic differential (\ref{6.1}) depends continuously on the
parameters $p_1$ and $p_2$, it is not difficult to show, using our
definition of $p_2^m$, that both zeros of  $Q(z;p_1,p_2^m)\,dz^2$
belong to the boundary of each of the domains $D_{-1}(p_1,p_2^m)$
and $D_\infty(p_1,p_2^m)$. But, as we have shown in part
\textbf{6.2} of Section~6, in this case the domain configuration
of $Q(z;p_1,p_2^m)\,dz^2$ must consist of three circle domains.
Therefore, as we have shown in part~\textbf{6.1} of Section~6, we
must have $p_1^m=\bar p_1$.

Thus, we have shown that $p_2\in \partial D_{-1}$ if $p_1$ and
$p_2$ satisfy (\ref{5.2}) and $|p_2+1|<|p_1+1|$. The M\"{o}bius
map $w=\frac{3-z}{1+z}$ interchanges the poles $z=\infty$ and
$z=-1$ of the quadratic differential (\ref{6.1}) and does not
change the type of its domain configuration. Therefore, our
argument shows also that $p_1\in \partial D_\infty$ if
$|p_2+1|<|p_1+1|$. This complete the proof of our claim  that
$p_1\in
\partial D_\infty$ if and only if $|p_2+1|<|p_1+1|$.

Similarly, if $Q(z)\,dz^2$ has exactly two circle domains
$D_\infty\ni \infty$ and $D_1\ni 1$, then $p_2\in \partial D_1$
and $p_1\in \partial D_\infty$ if and only if %
$$ %
\arg(p_1-1)=-\arg(p_2-1) \ \  \mod 2\pi \quad {\mbox{and}} \quad |p_2-1|<|p_1-1|.  %
$$ %

\smallskip

 \textbf{7.A(3).} In this part, we will discuss cases \textbf{6.3(a)}, \textbf{6.3(b1)}, and \textbf{6.3(b2)} discussed in Section~6.
 A domain configuration $\mathcal{D}_Q$ contains
 exactly one circle domains centered at $z=\infty$  if
 and only if neither $C_1$ or $C_{-1}$  is a positive real number.
 As we have found in Section~6, in this case there exist one or
 two strip domains $G_1$ and $G_2$ having their vertices at the
 poles $z=1$ and $z=-1$. In what follows, we will use the notion
 of the \emph{normalized height} $h$ of a strip domain $G$, which is
 defined as %
 $$ %
 h=\frac{1}{2\pi} \Im\,\int_\gamma \sqrt{Q(z)}\,dz>0, %
 $$ %
 where the integral is taken over any rectifiable arc
 $\gamma\subset G$ connecting the sides of~$G$.

 The sum of \emph{normalized heights }in the $Q$-metric of
 the strip domains, which have a vertex at the pole $z=1$ or at
 the pole $z=-1$ can be found using integration over
 circles $\{z:\,|z-1|=r\}$ and $\{z:\,|z+1|=r\}$ of radius $r$, $0<r<1$,
 as follows: %
 \begin{equation} \label{5.10} %
h_+=\frac{1}{2\pi}\Im \int_{|z-1|=r} \sqrt{Q(z)}\,dz=\frac{1}{2}
\Im\sqrt{C_1}=\frac{1}{2} \Im
\sqrt{(p_1-1)(p_2-1)}  %
 \end{equation} %
if $z=1$ and
 \begin{equation} \label{5.11} %
h_-=\frac{1}{2\pi}\Im \int_{|z+1|=r} \sqrt{Q(z)}\,dz=\frac{1}{2}
\Im\sqrt{C_{-1}}=\frac{1}{2} \Im
\sqrt{(p_1+1)(p_2+1)}  %
 \end{equation} %
if $z=-1$. The branches of the radicals in (\ref{5.10}) and
(\ref{5.11}) are chosen such that $h_+\ge 0$, $h_-\ge 0$.
Also, we assume here that if a strip domain has both vertices at
the same pole then its height is counted twice.

Comparing $h_+$ and $h_-$, we find three possibilities: %
\begin{itemize} %
\item[1)] %
If $h_+=h_-$, then the domain configuration $\mathcal{D}_Q$ has
only one strip domain $G_2$. This is the case discussed in parts
\textbf{6.3(a) }and \textbf{6.3(b1)} in Section~6. %
\item[2)] %
The case $h_+>h_-$ corresponds to the configuration with two strip
domains $G_1$ and $G_2$ discussed in part \textbf{6.3(b2)} in
Section~6.  In this case, the normalized heights $h_1$ and $h_2$
of the strip domains $G_1$ and $G_2$ can be calculated as follows:%
\begin{equation} \label{7.9.1}
h_1=\frac{1}{2}\left(h_+-h_-\right), \quad h_2=h_-.
\end{equation}
\item[3)] %
The case $h_+<h_-1$ corresponds to the configuration with two
strip
domains mentioned in part \textbf{6.3(b3)} in Section~6. %
\end{itemize}

Next, we will identify pairs $p_1$, $p_2$, which correspond to
each of the cases \textbf{6.3(a)}, \textbf{6.3(b1)}, and
\textbf{6.3(b2)}. The domain configuration $\mathcal{D}_Q$ has
exactly one strip domain if and only if $h_+=h_-$. Now,
(\ref{5.10}) and (\ref{5.11}) imply that the latter equation is
equivalent to the
following equation: %
$$%
\begin{array}{l} %
\left(\sqrt{(p_1-1)(p_2-1)}-\sqrt{(\bar p_1-1)(\bar
p_2-1)}\right)^2= \\%
 \left(\sqrt{(p_1+1)(p_2+1)}-\sqrt{(\bar
p_1+1)(\bar p_2+1)}\right)^2. %
\end{array} %
$$ %
Simplifying this equation, we conclude that $h_+=h_-$ if and
only if $p_1$ and $p_2$ satisfy the following equation: %
\begin{equation} \label{5.12} %
p_1+\bar p_1+p_2+\bar p_2+|p_1-1||p_2-1|-|p_1+1||p_2+1|=0 %
\end{equation} %

We claim that for a fixed $p_1$ with $\Im p_1\not=0$, the pair
$p_1$, $p_2$ satisfies equation (\ref{5.12}) if and only if
$p_2\in L(p_1)$ or $p_2\in H(p_1)$. Indeed, $p_2\in L(p_1)$ if and
only if %
\begin{equation} \label{5.13} %
|p_1-1|+|p_1+1|=|p_2-1|+|p_2+1|. %
\end{equation} %
Similarly, $p_2\in H(p_1)$ if and
only if %
\begin{equation} \label{5.14} %
|p_1-1|-|p_1+1|=|p_2-1|-|p_2+1|. %
\end{equation} %
Multiplying equations (\ref{5.13}) and (\ref{5.14}), after
simplification we again obtain equation (\ref{5.12}). Therefore,
$p_2\in L(p_1)$ or $p_2\in H(p_1)$ if and only if the pair $p_1$,
$p_2$ satisfy equation (\ref{5.12}). Thus, $\mathcal{D}_Q$ has
only one strip domain if and only if $p_2\in
L(p_1)\setminus\{p_1,\bar p_1\}$ or $p_2\in
H(p_2)\setminus\{p_1,\bar p_1\}$. This proves the first parts of
statements \textbf{6.3(a)} and \textbf{6.3(b1)}.

Now, we will prove that $p_1\in \partial D_\infty$ for all $p_2\in
E_{-1}^+(p_1)$. First, we claim that $p_1\in
\partial D_\infty$ for all $p_2$ sufficiently close to $-1$.
Arguing by contradiction, suppose that there is a sequence $s_k\to
-1$ such that  $s_k\in \partial D_\infty^k$  for all
$k=1,2,\ldots$ Here $D_\infty^k\ni \infty$ denotes the
corresponding circle domain of the quadratic differential
$Q_k(z)\,dz^2$ having the form (\ref{5.4}).
From (\ref{5.4}) we find that %
$$ 
Q_k(z)\,dz^2\to
\widehat{Q}(z)\,dz^2:=-\frac{z-p_1}{(z+1)(z-1)^2}\,dz^2, %
$$ 
where convergence is uniform on compact subsets of
$\mathbb{C}\setminus\{-1,1\}$. Since the residue of
$\widehat{Q}(z)$ at $z=\infty$ equals $1$, the quadratic
differential  $\widehat{Q}(z)\,dz^2$ has a circle domain
$\widehat{D}_\infty\ni \infty$ and if $\gamma\subset
\widehat{D}_\infty$ is a closed trajectory of
$\widehat{Q}(z)\,dz^2$, then $|\gamma|_{\widehat{Q}}=2\pi$.

Let us  show that the boundary of $\widehat{D}_\infty$ consists of
a single critical trajectory $\widehat{\gamma}_\infty$ of
$\widehat{Q}(z)\,dz^2$, which has both its end points at $z=p_1$.
Indeed, $\partial \widehat{D}_\infty$ consists of a finite number
of critical trajectories of $\widehat{Q}(z)\,dz^2$, which have
their end points at finite critical points. Therefore, if $-1\in
\partial \widehat{D}_\infty$, then $\partial \widehat{D}_\infty$
contains a critical trajectory, call it $\widehat{\gamma}_1$,
which joins $z=-1$ and $z=p_1$. Some notations used in this part
of the proof are shown in Fig.~7a.  This figure shows the limit
configuration, which is, in fact, impossible as we explain below.
In this case, $\partial \widehat{D}_\infty$ must contain a second
critical trajectory, call it $\widehat{\gamma}_2$, which has both
its end points at $z=p_1$. This implies that $z=1$ is the only
pole of $\widehat{Q}(z)\,dz^2$ lying in a simply connected domain,
call it $\widehat{D}_1$, which is bounded by critical
trajectories. Hence, $\widehat{D}_1$ must be a circle domain of
$\widehat{Q}(z)\,dz^2$. Furthermore, the domain configuration
$\mathcal{D}_{\widehat{Q}}$ consists of two circle domains
$\widehat{D}_1$, $\widehat{D}_\infty$, which in this case must be
the extremal domains of Jenkins module problem on the following
maximum of the
sum of reduced moduli: %
$$ %
m(B_\infty,\infty)+t^2 m(B_1,1) \quad {\mbox{with some fixed $t>0$,}} %
$$ %
where the maximum is taken over all pairs of simply connected
non-overlapping domains $B_\infty\ni \infty$ and $B_1\ni 1$. It is
well known that such a pair of extremal domains is unique; see for
example, \cite{S1}. Therefore, $\widehat{D}_1$ and
$\widehat{D}_\infty$ must be symmetric with respect to the real
line (as is shown, for instance, in Fig.~5d), which is not the
case since $\widehat{Q}(z)\,dz^2$ has only one zero $p_1$ with
$\Im p_1>0$.

Thus, $\partial \widehat{D}_\infty=\widehat{\gamma}_\infty\cup
\{p_1\}$ and $z=-1$ lies in the domain complementary to the
closure of  $\widehat{D}_\infty$. Fig.~7b illustrates notations
used further on in this part of the proof.

Let $\tilde \gamma_{-1}$ denote the $\widehat{Q}$-geodesic in the
class of all curves having their end points at $z=-1$, which
separate the points $z=1$ and $z=p_1$ from $z=\infty$. Since
$-1\not \in \partial \widehat{D}_\infty$ it follows that %
\begin{equation} \label{5.16} %
|\tilde
\gamma_{-1}|_{\widehat{Q}}>|\widehat{\gamma}_\infty|_{\widehat{Q}}=2\pi. %
\end{equation} %
Let $\varepsilon>0$ be such that %
\begin{equation} \label{5.17} %
\varepsilon<\frac{1}{4}\left(|\tilde{\gamma}_{-1}|_{\widehat{Q}}-2\pi\right). %
\end{equation} %
Let $r>0$ be sufficiently small such that %
\begin{equation} \label{5.18} %
|[-1,-1+re^{i\theta}]|_{\widehat{Q}}<\varepsilon/8 \quad
{\mbox{for all $0\le \theta<2\pi$.}} %
\end{equation} %
Now let $\tilde \gamma_r$ be the shortest in the
$\widehat{Q}$-metric among all arcs having their end points on the
circle $C_r(-1)=\{z:\,|z+1|=r\}$ and separating the points $z=1$
and $z=p_1$ from the point $z=\infty$ in the exterior of the
circle $C_r(-1)$. It is not difficult to show that there is at
least one such curve $\tilde\gamma_r$. It follows from
(\ref{5.18}) that %
\begin{equation} \label{5.19} %
|\tilde\gamma_r|_{\widehat{Q}}>|\tilde\gamma_{-1}|_{\widehat{Q}}-\varepsilon/4.%
\end{equation} %

Since $s_k\to -1$, $s_k\in \partial D_\infty^k$, and $p_1\not \in
D_\infty^k$, it follows that for every sufficiently large $k$
there is a regular trajectory $\gamma(k)$ of $Q_k(z)\,dz^2$
intersecting the circle $C_r(-1)$ and such that the arc
$\gamma'(k)=\gamma(k)\setminus \{z:\,|z+1|\le r\}$ separates the
points $z=1$ and $z=p_1$ from $z=\infty$ in the exterior of
$C_r(-1)$. Since $|\gamma(k)|_{Q_k}=2\pi$ for all $k$ and since
every quadratic differential $Q_k(z)\,dz^2$ has second order poles
at $z=1$ and $z=\infty$ it follows from (\ref{5.4}) that there is
$r_0>0$ small enough such that $\gamma'(k)$ lies on the compact
set $K_0=\{z:\,|z|\le
1/r_0\}\setminus\left(\{z:\,|z-1|<r_0\}\cup\{z:\,|z+1|<r\}\right)$
for all $k$ sufficiently large. We note also that $Q_k(z)\to
\widehat{Q}(z)$ uniformly on $K_0$. This implies, in particular,
that for all $k$ the Euclidean lengthes of
$\gamma'(k)$ are bounded by the same constant and that %
\begin{equation} \label{5.20} %
|\gamma'(k)|_{Q_k}\ge |\gamma'(k)|_{\widehat{Q}}-\varepsilon/4
\end{equation} %
for all $k$ sufficiently large.

Combining (\ref{5.16})--(\ref{5.20}), we obtain the following
relations: %
$$ 
\begin{array}{rl}
2\pi=&|\gamma(k)|_{Q_k}\ge |\gamma'(k)|_{Q_k}\ge
|\gamma'(k)|_{\widehat{Q}}-\varepsilon/4\ge
|\tilde\gamma_r|_{\widehat{Q}}-\varepsilon/4 \\ %
>&|\tilde\gamma_{-1}|_{\widehat{Q}}-\varepsilon/2>|\tilde\gamma_{-1}|_{\widehat{Q}}
-\frac{1}{2}\left(|\tilde\gamma_{-1}|_{\widehat{Q}}-2\pi\right)     \\ %
 =&
\frac{1}{2}\left(|\tilde\gamma_r|_{\widehat{Q}}+2\pi\right)>2\pi,
\end{array}
$$ 
which, of course, is absurd. Thus, $p_2\in \partial D_\infty$ for
all $p_2$ sufficiently close to $-1$.

\smallskip

Let  $\Delta\not= \emptyset$ be the set of all $p_2\in
E_{-1}^+(p_1)$ such that $p_1\in \partial D_\infty$.  To prove
that $\Delta=E_{-1}^+(p_1)\setminus\{-1\}$, it is  sufficient to
show that $\Delta$ is closed and open in $E_{-1}^+(p_1)$. Arguing
by contradiction, we suppose that there is a sequence of poles
$s_k:=p_2^k\in E_{-1}^+(p_1)$, $k=1,2,\ldots,$ such that $s_k\to
s_0:=p_2^0\in E_{-1}^+(p_1)$ and $p_1\in
\partial D_\infty^k$ for all $k=1,2,\ldots$ but $p_1\not\in \partial D_\infty^0$.
In this part of the proof, the index $k=0,1,2,\ldots$, used in the
notations  $D_\infty^k$, $\tilde\gamma_k$, etc., will denote
domains, trajectories, and other objects corresponding to the
quadratic differential $Q_k(z)\,dz^2$ defined by (\ref{5.4}).
Since $\partial D_\infty^0$ contains a critical point and
$p_1\not\in \partial D_\infty^0$, we must have  $p_2^0\in
\partial D_\infty^0$.

Fig.~7c illustrates some notations used in this part of the proof.
In this case, the boundary $\partial D_\infty^0$ consists of a
single critical trajectory $\gamma_\infty^0$ and its end points,
each of which is at $z=p_2^0$. In addition, there is a critical
trajectory of infinite $Q^0$-length, called it $\hat\gamma$, which
has one end point at $p_2^0$ and which approaches to the pole
$z=-1$ or the pole $z=1$ in the other direction. Let $P_0$ be a
point on $\hat \gamma$ such that the $Q^0$-length of the arc
$\hat\gamma_0$ of $\hat\gamma$ joining $p_2^0$ and $P_0$ equals
$L$, where $L>0$ is sufficiently large. For $\delta>0$
sufficiently small, let $\gamma_1^\perp$ and $\gamma_2^\perp$
denote disjoint open arcs on the orthogonal trajectory of
$Q^0(z)\,dz^2$ passing through $P_0$ such that each of
$\gamma_1^\perp$ and $\gamma_2^\perp$ has one end point at $P_0$
and each of them has $Q^0$-length equal to $\delta$. If $\delta$
is small enough, then there is an arc of a  trajectory of
$Q^0(z)\,dz^2$, call it $\tilde\gamma$, which  connects the second
end point of $\gamma_1^\perp$ with the second end point of
$\gamma_2^\perp$. Now, let $D(\delta)$ be the domain, the boundary
of which consists of the arcs $\gamma_\infty^0$, $\hat\gamma_0$,
$\gamma_1^\perp$, $\gamma_2^\perp$, and their end points. In the
terminology explained in Section~3, the domain $D(\delta)$ is a
$Q^0$-rectangle of $Q^o$-height $\delta$.

 If
$\delta>0$ is sufficiently small, then $p_1$ belong to the bounded
component of $\mathbb{C}\setminus \overline{D(\delta)}$. Let
$\tilde\gamma_1$ be the arc of a trajectory of $Q^0(z)\,dz^2$,
which divide $D(\delta)$ into two $Q^0$-rectangles, each of which
has the $Q^0$-height equal to $\delta/2$.  Since $p_1\in \partial
D_k$ for all $k$ and $p_1$ belongs to the bounded component of
$\mathbb{C}\setminus \overline{D(\delta)}$, it follows that, for
each $k=1,2,\ldots,$ there is a closed trajectory $\hat\gamma_k$
of $Q_k(z)\,dz^2$ lying in $D_\infty^k$, which intersects
$\tilde\gamma_1$ at some point $\tilde z_k\in D(\delta)$.

Since $Q_k(z)\to Q^0(z)$ it follows that, for all sufficiently
large $k$, the trajectory $\hat\gamma_k$ has an arc
$\tilde\gamma_k$ such that $\tilde\gamma_k\subset D(\delta)$ and
$\tilde\gamma_k$ has one end point on each of the arcs
$\gamma_1^\perp$ and $\gamma_2^\perp$.

Now, since $Q_k(z)\to Q^0(z)$ uniformly on $\overline{D(\delta)}$
 it follows that %
 $$ 
|\hat\gamma_k|_{Q_k}\ge |\tilde\gamma_k|_{Q_k}\to
|\tilde\gamma_1|_{Q^0}=|\gamma_\infty^0|_{Q^0}+2|\hat\gamma_0|_{Q^0}=2\pi+2L,
 $$ 
 contradicting to the fact that $|\hat\gamma_k|_{Q_k}=2\pi$. The latter fact
 follows from the assumption that $\hat \gamma_k$ is a closed
 trajectory of $Q_k(z)\,dz^2$, which lies in a circle domain $D_\infty^k$.

Thus, we have proved that $\Delta$ is closed in $E_{-1}^+(p_1)$. A
similar  argument can be used to show that $\Delta$ is open in
$E_{-1}^+(p_1)$. The difference is that to construct a domain
$D(\delta)$, we now use an arc $\tilde\gamma_1$ of a critical
trajectory $\hat \gamma_1$, which has one of its end points at the
pole $p_1$ and not at the pole $p_1^0$ as we had in the previous
case.

Therefore, we have proved that if $p_2\in E_{-1}^+(p_1)$, then
$p_1\in
\partial D_\infty$. The same argument can be used to prove that if
$p_2\in E_1^+(p_1)$, then $p_1\in \partial D_\infty$.

Finally, if $p_2\in E_1^-(p_1)$ or $p_2\in E_{-1}^-(p_1)$, then we
can switch the roles of the poles $p_1$ and $p_2$ in our previous
proof and conclude that $p_2\in \partial D_\infty$ in these cases.
This proves the first part of statement \textbf{6.3(b2)}.

\bigskip

Now, possible positions of zeros $p_1$ and $p_2$ on boundaries of
the corresponding circle and strip domains are determined for all
cases. Next, we will discuss limiting behavior of critical
trajectories. We will give a proof for the most general case when
the domain configuration consists of a circle  domain $D_\infty$
and strip domains $G_1$ and $G_2$. In all other cases proofs are
similar.

Let $\Delta$ denote the set of pairs $(p_1,p_2)$, for which the
limiting behavior of critical trajectories is shown in Fig.~4a or
 in more general case in Fig.~4b. That is when $\gamma_1$ joins $p_1\in
\partial D_\infty\cap \partial G_1$ and $z=1$, $\gamma_{-1}$ joins
$p_2\in
\partial G_1\cap \partial G_2$ and $z=-1$, and $\gamma_0^+$ and
$\gamma_0^-$ each joins $p_2$ and $z=1$. First, we note that
$\Delta$ is not empty since $(p_1,p_2)\in \Delta$ when  $p_1>1$
and $-p_1<p_2<-1$.  In this case the intervals $(p_2,-1)$ and
$(1,p_1)$ represent critical trajectories $\gamma_1$ and
$\gamma_{-1}$ and critical trajectories $\gamma_0^+$ and
$\gamma_0^-$ connect a zero at $p_2$ with a pole at $z=1$; see
Fig.~4a.

We claim that $\Delta$  is open. To prove this claim, suppose that
$(p_1^0,p_2^0)\in \Delta$ and that $(p_1^k,p_2^k)\to
(p_1^0,p_2^0)$ as $k\to \infty$, $k=1,2,\ldots$ Fix $\varepsilon>$
small enough and consider the arc
$\gamma_1^0(\varepsilon)=\gamma_1^0\setminus
\{z:\,|z-1|<\varepsilon\}$ of the critical trajectory
$\gamma_1^0$, which goes from $p_1^0$ to the pole $z=1$. Since
$(p_1^k,p_2^k)\to (p_1^0,p_2^0)$ it follows that for all $k$
sufficiently big there is a critical trajectory $\gamma_1^k$
having one point at $p_1^k$ which has a subarc
$\gamma_1^k(\varepsilon)$ which lies in the
$\varepsilon/10$-neighborhood of the arc
$\gamma_1^0(\varepsilon)$. In particular, eventually,
$\gamma_1^k(\varepsilon)$ enters the disk
$\{z:\,|z-1|<\varepsilon\}$. Therefore, it follows from the
standard continuity argument and Lemma~4 that $\gamma_1^k$
approaches the pole $z=1$. The same argument works for all other
critical trajectories of the quadratic differential (\ref{6.1})
with $p_1=p_1^k$, $p_2=p_2^k$. Thus, we have proved that $\Delta$
is open.

Same argument can be applied to show that all other sets of points
$(p_1,p_2)$ responsible for different types of limiting behavior
of critical trajectories mentioned in part \textbf{6.3(b2)} of
Theorem~\ref{Theorem 5.1} are also nonempty and open. The latter
implies that each of these sets must coincide with some connected
component of the set $\mathbb{C}\setminus (L(p_1)\cup H(p_1))$.
This proves the desired statement in the case under consideration.


\textbf{7.B.} The local behavior of trajectories near second order
poles at $z=1$ and $z=-1$ is controlled by Laurent coefficients
$C_1$ and $C_{-1}$, respectively, which are given by formula
(\ref{5.1}). The radial structure near $z=1$ or near $z=-1$ occurs
if and only if $C_1<0$ or $C_{-1}<0$, respectively. The latter
inequalities are equivalent to the following relations: %
\begin{equation} \label{5.23} %
\arg(p_1-1)=-\arg(p_2-1)+\pi %
\end{equation} %
or %
\begin{equation} \label{5.24} %
\arg(p_1+1)=-\arg(p_2+1)+\pi. %
\end{equation} %
Now, statement  \textbf{(1)} about radial behavior follows from
(\ref{5.23}) and (\ref{5.24}).

Next, trajectories of $Q(z)\,dz^2$ approaching the pole $z=1$
spiral clockwise if and only if $0<\arg C_1<\pi$. The latter is
equivalent to the inequalities: %
$$ %
-\arg{p_1-1}<\arg(p_2-1)<-\arg(p_1-1)+\pi, %
$$ %
which imply the desired statement for   the case when trajectories
of $Q(z)\,dz^2$ approaching $z=1$ spiral clockwise.  In the
remaining cases the proof is similar.

The proof of Theorem~\ref{Theorem 5.1} is now complete. %
\hfill $\Box$

\begin{remark} \label{Remark-1} %
The case when $\Im p_1=0$ but $\Im p_2\not=0$ can  be reduced to
the case covered by Theorem~\ref{Theorem 5.1} by changing
numeration of zeros. In the remaining case when $\Im p_1=0$ and
$\Im p_2=0$, the domain configurations are rather simple; they are
symmetric with respect to the real axis as it is shown in
Figures~1a, 1b, 2a, 3a, and some other figures.
\end{remark} %

\section{Identifying simple critical    geodesics and critical loops}  \label{Section-7}
\setcounter{equation}{0}

Topological information obtained in Section~6 is sufficient to
identify all critical geodesics and all critical geodesic loops of
the quadratic differential (\ref{6.1}) in all cases. In
particular, we can identify all simple geodesics.

Cases \textbf{6.1(a)} and \textbf{6.1(b)}; see Fig.~1a and
Fig.~1b. Let $\gamma$ be a geodesic joining $p_1$ and $p_2$. Since
$D_\infty$, $D_1$, and $D_{-1}$ are simply connected and $p_1\in
\partial D_\infty\cap
\partial D_1$ and $p_2\in \partial D_\infty\cap \partial D_{-1}$
it follows from Lemma~4 that $\gamma$ does not intersect
$D_\infty$, $D_1$, and $D_{-1}$. In this case, $\gamma$ must be
composed of a finite numbers of copies of $\gamma_0$, a finite
number of copies of $\gamma_1$, and a finite number of copies of
$\gamma_{-1}$. Therefore the only simple geodesic joining $p_1$
and $p_2$ in this case is the segment $\gamma_0=[p_2,p_1]$.

In addition, by Lemma~\ref{Lemma-4.2}, $\gamma_1$ is the only
simple non-degenerate geodesic from the point $p_1$ to itself and
$\gamma_{-1}$ is the only short geodesic from $p_2$ to $p_2$.


Case \textbf{6.1(c)};  see Fig.~1c. %
As in the previous case, any geodesic $\gamma$ joining $p_1$ and
$p_2$ must be composed of a finite number of copies of $\gamma_0$,
a finite number of copies of $\gamma_1$, and a finite number of
copies of $\gamma_{-1}$. Thus, in this case there exist exactly
three simple geodesics joining $p_1$ and $p_2$, which are
$\gamma_0$, $\gamma_1$, and $\gamma_{-1}$. By
Lemma~\ref{Lemma-4.2}, there are no geodesic loops in this case.

\smallskip

Case \textbf{6.2}; see Fig.~2a, 2b. %
Suppose that $\mathcal{D}_Q$ consists of circle domains $D_\infty$
and $D_{-1}$ and
 a strip domain $G_1$. Let $\gamma$ be a geodesic joining $p_1$
 and $p_2$. If $\gamma$ contains a point $\zeta\in \gamma_{-1}$ or
 a point $\zeta\in \gamma_\infty$, then it follows from
 Lemma~\ref{Lemma-4.1}  that $\gamma_{-1}$ or, respectively, $\gamma_\infty$ is a subarc
 of $\gamma$. Thus, $\gamma$ is not simple in these cases.

 Suppose now that $\gamma\subset G_1\cup \gamma_1^+\cup
 \gamma_1^-$.
 Since $G_1$ is a strip domain the function $w=F(z)$ defined by %
 \begin{equation} \label{4.1} %
 F(z)=\frac{1}{2\pi}\,\int_{p_1}^z \sqrt{Q(z)}\,dz, %
 \end{equation} %
 with an appropriate  choice of the radical, maps $G_1$ conformally and
 one-to-one onto the horizontal strip $S_{h_1}$, where $S_h=\{w:\,0<\Im
 w<h_1\}$,
  in such a way
 that the trajectory $\gamma_\infty$ is mapped onto an interval
 $(x_1,x_1')\subset \mathbb{R}$ with $x_1=0$ and  $x'_1=1$. Here
 $h_1$ is the normalized height of the strip domain $G_1$ defined
 by (\ref{7.9.1}).
 Fig.~8a and Fig.~9a illustrate some notions relevant to Case
 \textbf{6.2}.
To simplify notations in our figures, we will use the same
notations for $Q$-geodesics (such as $\gamma_\infty$,
$\gamma_{11}$, $\gamma'_{12}$, etc.) in the $z$-plane and for
their images under the mapping $w=F(z)$ in the $w$-plane.

 \smallskip

 The indefinite integral $\Phi(z)=\frac{1}{2\pi}\int\sqrt{Q(z)}\,dz$ can be
 expressed explicitly in terms of elementary functions as follows:%
 \begin{equation} \label{7.9.2} %
\begin{array}{ll} %
\hspace{-0.2cm}\Phi(z)\hspace{-0.25cm}&=\frac{1}{4\pi i}\left(
\sqrt{(p_1-1)(p_2-1)}\log(z-1)-\sqrt{(p_1+1)(p_2+1)}\log(z+1)
\right. \\ %
{}&+4\log(\sqrt{z-p_1}+\sqrt{z-p_2})\\
{}&+2\sqrt{(p_1+1)(p_2+1)}  \log(\sqrt{(p_1+1)(z-p_2)}-\sqrt{(p_2+1)(z-p_1)}) \\ %
{}&-\left. 2\sqrt{(p_1-1)(p_2-1)}
\log(\sqrt{(p_1-1)(z-p_2)}-\sqrt{(p_2-1)(z-p_1)})\right). %
\end{array} %
 \end{equation} %
Equation (\ref{7.9.2}) can be  verified by straightforward
differentiation. Alternatively, it can be verified with
\emph{Mathematica} or \emph{Maple}.  With
(\ref{7.9.2}) at hands, the function $F(z)$ can be written as %
\begin{equation} \label{7.9.3} %
F(z)=\Phi(z)-\Phi(p_1), %
\end{equation} %
where %
\begin{equation} \label{7.9.4} %
\Phi(p_1)=\frac{1}{4\pi
i}\left(2+\sqrt{(p_1-1)(p_2-1)}-\sqrt{(p_1+1)(p_2+1)}\right)\log(p_1-p_2).
\end{equation} %
Calculating $\Phi(p_2)$,   
after some algebra, we find that: %
\begin{equation} \label{7.9.5} %
F(p_2)=\frac{1}{2}+\frac{1}{4}\left(\sqrt{(p_1-1)(p_2-1)}-\sqrt{(p_1+1)(p_2+1)}\right).
\end{equation} %
Of course, all branches of the radicals and logarithms in
(\ref{7.9.2})--(\ref{7.9.5}) have to be appropriately chosen.

To explain more precisely our choice of branches of multi-valued
functions in (\ref{7.9.2})--(\ref{7.9.5}), we note that the points
$p_1$, $p_2$ and points of the arcs $\gamma_1^+$ and $\gamma_1^-$
each represents two distinct boundary points of $G_1$ and
therefore every such point has two images under the mapping
$F(z)$. These images will be denoted by $x_1(\zeta)$ and
$x'_1(\zeta)$ if $\zeta\in \gamma_1^+\cup\{p_1\}$ and by
$x_2(\zeta)+ih_1$ and $x'_2(\zeta)+ih_1$ if $\zeta\in
\gamma_1^-\cup \{p_2\}$. We assume here that
$x_1(\zeta)<x'_1(\zeta)$ for all $\zeta\in \gamma_1^+\cup \{p_1\}$
and $x_2(\zeta)<x'_2(\zeta)$ for all $\zeta\in \gamma_1^-\cup
\{p_2\}$. In accordance with our notation above, $x_1(p_1)=x_1=0$
and $x'_1(p_1)=x'_1=1$. We also will abbreviate $x_2(p_2)$ and
$x'_2(p_2)$ as $x_2$ and $x'_2$, respectively.

For every $\zeta\in \gamma_1^+$, the segments $[x_1(\zeta),x_1]$
and $[x'_1,x'_1(\zeta)]$ are the images of the same arc on
$\gamma_1^+$. Therefore they have equal lengthes. Similarly, the
segments $[x_2(\zeta)+ih_1,x_2+ih_1]$ and
$[x'_2+ih_1,x'_2(\zeta)+ih_1]$ have equal
lengthes. Thus, for every $\zeta\in \gamma_1^+$ and every $\zeta\in \gamma_1^-$, we have, respectively: %
\begin{equation} \label{4.2} %
x_1-x_1(\zeta)=x'_1(\zeta)-x'_1 \quad {\mbox{and}} \quad
x_2-x_2(\zeta)=x'_2(\zeta)-x'_2. %
\end{equation} %

We know that the preimage under the mapping $F(z)$ of every
straight line segment is a geodesic. This immediately implies that
in the case under consideration there exist four simple critical
geodesics, which are the following preimages: %
\begin{equation} \label{4.3} %
\begin{array}{ll} \gamma_{12}=F^{-1}((x_1,x_2+ih_1)),& \quad
\gamma'_{12}=F^{-1}((x_1,x'_2+ih_1)),\\ %
\gamma_{21}=F^{-1}((x'_1,x_2+ih_1)),& \quad
\gamma'_{21}=F^{-1}((x'_1,x'_2+ih_1)). %
\end{array} %
\end{equation}%
The geodesic loops $\gamma_\infty$ and $\gamma_{-1}$ are the
following preimeges: %
\begin{equation} \label{4.4} %
\gamma_\infty=F^{-1}((x_1,x_1)), \quad
\gamma_{-1}=F^{-1}((x_2+ih_1,x'_2+ih_1)). %
\end{equation}%

We claim that there is no other simple geodesic joining the points
$p_1$ and $p_2$. Fig.~9a illustrates some notation used in the
proof of this claim. Suppose that $\tau$ is a geodesic ray issuing
from $p_1$ into the region $G_1$. Let $\tau_k$, $k=1,\ldots, N$,
be connected components of the intersection $\tau\cap G_1$
enumerated in their natural order on $\tau$. In particular,
$\tau_1$ starts at $p_1$. We may have finite or infinite number of
such components. Thus, $N$ is  a finite number or $N=\infty$. Let
$l_k=F(\tau_k)$. Since all $\tau_k$ lie  on the same geodesic it
follows that $l_k$ are parallel line intervals in $S$ joining the
real axis and the horizontal line $L_{h_1}$, where $L_h =\{w:\,
\Im w=h\}$. Let $v'_k$ and $v''_k$ be the initial point and
terminal point of $l_k$, respectively. Then $v'_k=e'_k$ and
$v''_k=e''_k+ih_1$ with real $e'_k$ and $e''_k$ if $k$ is odd and
$v'_k=e'_k+ih_1$, $v''_k=e''_k$ with real $e'_k$ and $e''_k$ if
$k$ is even.

The interval $l_1$ may start at $x_1$ or at $x'_1$. To be
definite, suppose that $e'_1=x_1$. For the position of $e''_1$ we
have the
following possibilities: %

\begin{enumerate} %
\item[(a)] %
$e''_1=x_2$ or $e''_1=x'_2$. In this case, $\tau_1=\gamma_{12}$ or
$\tau_1=\gamma'_{12}$. Thus we obtain two out of four
geodesics  in (\ref{4.3}). %
\item[(b)] %
$x_1<e''_1<x'_1$. In this case, $\tau_1$ has its end point on
$\gamma_{-1}$. By Lemma~\ref{Lemma-4.1}, the continuation of
$\tau_1$ as a geodesic will stay in $D_{-1}$ and will approach to
the pole $z=-1$. Thus, $\tau$ is not  a geodesic from $p_1$ to
$p_2$ or a geodesic loop from $p_1$ to itself in
this case. %
\item[(c)] %
$e''_1>x'_2$. Let $d=e''_1-x'_2$. It follows from (\ref{4.2}) that
$e'_2=x_2-d$. Then $e''_2=x_1-d$. In general,
$e'_{2k-1}=x'_1+(k-1)d$, $e''_{2k-1}=x'_2+kd$ for $k=1,2,\ldots$,
and $e'_{2k}=x_2-kd$, $e''_{2k}=x_1-kd$ for $k=1,2,\ldots$. Thus,
$\tau$ cannot terminate at $p_1$ or $p_2$. Instead, $\tau$
approaches to the pole at $z=1$ as a logarithmic spiral.  %
\item[(d)] %
$e''_1<x_2$. Let $d_0=x_2-e''_1$. Then $e'_2=x'_2+d_0$ by
(\ref{4.2}). For the position of $e''_2$ we have three
possibilities. %
\begin{itemize} %
\item[$(\alpha$)] %
$x_1<e''_2<x'_1$. In this case by Lemma~\ref{Lemma-4.1}, the
continuation of $\tau_2$ as a geodesic ray will stay in $D_\infty$
and will approach to the pole $z=\infty$. Thus, $\tau$ is not  a
geodesic from $p_1$ to $p_2$ or a geodesic loop in this
case. %
\item[$(\beta$)] %
$e''_2=x'_1$. In this case, $\tau$ is a critical geodesic loop
$\gamma_{11}=F^{-1}((x_1,v''_1]\cup [v'_2,x'_1))$ from $p_1$ to
itself. We emphasize here, that since the segments $l_1$ and $l_2$
are parallel a critical geodesic loop from $p_1$ to itself occurs
if and only if
 $|\gamma_\infty|_Q=x'_1-x_1>x'_2-x_2=|\gamma_{-1}|_Q$. If $|\gamma_\infty|_Q<|
 \gamma_{-1}|_Q$, then there is a critical geodesic loop $\gamma_{22}$ with end points at $p_2$.%
\item[$(\gamma$)] %
$e''_2>x'_1$. Let $d=e''_2-x'_1$. Then, as in the case c), we
obtain that $e'_{2k+1}=x_1-kd$, $e''_{2k+1}=x_2-d_0-kd$ for
$k=1,2,\ldots$, and $e'_{2k}=x'_2+d_0+kd$, $e''_{2k}=x'_1+kd$ for
$k=1,2,\ldots$. Therefore, $\tau$ does not terminate at $p_1$ or
$p_2$. Instead, $\tau$ approaches to the pole at $z=1$ as a
logarithmic spiral. %
\end{itemize} %
\end{enumerate} %

If $l_1$ has its initial point at $x'_1$, the same argument shows
that there are exactly two geodesics joining $p_1$ and $p_2$,
which are the geodesics $\gamma_{21}$ and $\gamma'_{21}$ defined
by (\ref{4.3}).

\smallskip

Combining our findings for  Case \textbf{6.2}, we conclude that in
this case there exist exactly four distinct geodesics joining
$p_1$ and $p_2$, which are given by (\ref{4.3}). The geodesic
loops $\gamma_\infty$ and $\gamma_{-1}$ are given by (\ref{4.4}).
In addition, if $|\gamma_\infty|_Q\not=|\gamma_{-1}|_Q$, then
there is exactly one geodesic loop containing the pole $z=1$ in
its interior domain, which has its end points at a zero of
$Q(z)\,dz^2$. This loop has the pole $z=1$ in its interior domain,
which does not contain other critical points of $Q(z)\,dz^2$, and
has both its end points at $p_1$ or at $p_2$, if
$|\gamma_\infty|_Q>|\gamma_{-1}|_Q$ or
$|\gamma_\infty|_Q<|\gamma_{-1}|_Q$, respectively.

Finally, if $|\gamma_\infty|_Q=|\gamma_{-1}|_Q$, then the
geodesics $\gamma_{12}$ and $\gamma'_{21}$ together with points
$z=p_1$ and $p_2$ form a boundary of a simply connected bounded
domain, which contains the pole $z=1$ and does not contain other
critical points of $Q(z)\,dz^2$. There are no geodesic loops
containing $z=1$ in its interior domain in this case.

\smallskip

The argument based on the construction of parallel segments
divergent to $\infty$, which was used above to prove non-existence
of some geodesics, will be used for the same purpose in several
other cases considered below. Since the detailed construction is
rather lengthy, the detailed exposition will be given for one more
case when we have two strip domains. In other cases, we  will just
refer to this argument (which  actually is rather standard, see
\cite[Ch. IV]{Str}) and call it the ``proof by
 construction of divergent geodesic segments''.

\smallskip

Case \textbf{6.3(a)}; see Fig.~8b. In this case, the domain
configuration $\mathcal{D}_Q$ consists of a circle domain
$D_\infty$ and a strip domain $G_2$ having its vertices at the
poles $z=1$ and $z=-1$. The function $F(z)$ defined by (\ref{4.1})
maps $G_2$ conformally and one-to-one onto the strip $S_{h_1}$
such that  the trajectory $\gamma_\infty^+$ is mapped onto the
interval $(x_1,x_2)\subset \mathbb{R}$ with $x_1=0$ and some
$x_2$, $0<x_2<1$. The points $z=p_1$ and $z=p_2$ each has two
images under the mapping $F(z)$. Let $x_1=0$ and $x'_1+ih_1$ with
some real $x'_1$ be the images of $p_1$ and let $x_2$ and
$x'_2+ih_1$ with  $x'_2=x'_1+(1-x_2)$ be the images of $p_2$.
Arguing as in Case \textbf{6.2}, one can easily find four distinct
simple geodesics joining the points $p_1$ and $p_2$. These
geodesics are: %
$$ 
\begin{array}{ll} %
\gamma_{12}=F^{-1}((x_1,x_2))=\gamma_\infty^+, &\quad
\gamma'_{12}=F^{-1}((x'_1+ih_1,x'_2+ih_1))=\gamma_\infty^-, \\
\gamma_{21}=F^{-1}((x_1,x'_2+ih_1)), &\quad
\gamma'_{21}=F^{-1}((x_2,x'_1+ih_1)).
\end{array} %
$$ 
In addition, there are two critical geodesic loops:
$$ 
\gamma_{11}=F^{-1}((x_1,x'_1+ih_1)) \quad {\mbox{and}} \quad
\gamma_{22}=F^{-1}((x_2,x'_2+ih_1)). %
$$ 
It follows from Lemma~\ref{Lemma-4.2} that there are no other such
loops.

Using the proof by construction of divergent geodesic segments as
in Case \textbf{6.2}, we can show that there are no other simple
geodesics joining $p_1$ and $p_2$.

Case \textbf{6.3(b1)}; see Fig.~8c.  We still have a circle domain
$D_\infty$ and a strip domain $G_2$. In this case, the function
$F(z)$ defined by (\ref{4.1}) as in Case \textbf{6.2} maps $G_2$
conformally and one-to-one onto $S_{h_1}$ such that
$\gamma_\infty$ is mapped onto the interval $(x_1,x'_1)\subset
\mathbb{R}$, where $x_1=0$ and $x'_1=1$. The difference is that
that now the point $p_2$ represents three boundary points of
$G_2$. Two of them belong to the side $l_2$ and the third point
belongs to the side $l_1$. Accordingly, there are three images of
$p_2$ under the mapping $F(z)$, which we will denote by
$x_2+ih_1$, $x'_2$, and $x''_2$. Here $x_2$ may be any real number
while $x'_2$ and $x''_2$ satisfy the following conditions: %
$$ 
x'_2>x'_1, \quad x''_2<x_1, \quad {\mbox{and}} \quad x'_2-x'_1=x_1-x''_2. %
$$ 

In this case, there are three short geodesics, which are the
following
preimages: %
$$ %
\gamma_0=F^{-1}((x''_1,x_1))=F^{-1}((x'_1,x'_2)) %
$$ %
and %
$$ %
\gamma_{12}=F^{-1}((x_1,x_2+ih_1)), \quad
\gamma'_{12}=F^{-1}((x'_1,x_2+ih_1)). %
$$ %
In addition, there are three geodesic loops: %
$$ %
\gamma_\infty=F^{-1}((x_1,x'_1)),
 \quad \gamma'_{22}=F^{-1}((x_2+ih_1,x'_2)), \quad
\gamma''_{22}=F^{-1}((x_2+ih_1,x''_2)). %
$$ %

Using the proof by construction of divergent segments as above, it
is not difficult to show that there are no other simple geodesics
joining the points $p_1$ and $p_2$.

Case \textbf{6.3(b2)}. This is the most general case with many
subcases illustrated in Fig.~10a-10i. In this case we have a
circle domain $D_\infty$ and two strip domains $G_1$ and $G_2$. We
assume that $\mathcal{D}_Q$ has topological type  shown in
Fig.~4b. In other cases the proof follows same lines. The function
$F(z)$ defined by (\ref{4.1}) maps $G_1$ conformally and
one-to-one onto the strip $S_{h_1}$ such that $\gamma_\infty$ is
mapped onto the interval $(x_1,x'_1)\subset \mathbb{R}$, where
$x_1=0$ and $x'_1=1$. The point $p_2$ represents one boundary
point of $G_1$ and two boundary points of $G_2$. Let $x_2+ih_1$ be
the image of $p_2$ considered as a boundary point of $G_1$. Then
the trajectory $\gamma_0^+$ considered as boundary arc of $G_1$ is
mapped onto the ray $r_1=\{w=t+ih_1:\,t<x_2\}$, while the
trajectory $\gamma_0^-$ is mapped onto the ray
$r_2=\{w=t+ih_1:\,t>x_2\}$. The function $F(z)$ can be continued
analytically through the trajectory $\gamma_0^+$. The continued
function (for which we keep our previous notation $F(z)$) maps
$G_2$ conformally and one-to-one onto the strip $S(h_1,h)= \{w:\,
h_1<\Im w<h\}$ with $h=h_1+h_2$, where $h_1$ and $h_2$ are defined
by (\ref{7.9.1}).
%
Two boundary points of $G_2$ situated at $p_2$ are mapped onto the
points $x_2+ih_1$ and $x'_2+ih$ with some $x'_2\in \mathbb{R}$.
Thus, the domain $\widetilde{D}=G_1\cup G_2\cup \gamma_0^+$ is
mapped by $F(z)$ conformally and one-to-one onto the slit strip
$\widehat{S}(h_1,h)=\{w: \, 0<\Im w <h\}\setminus
\{w=t+ih_1:\,t\ge x_2\}$.

We note that every boundary point $\zeta\in
\gamma_1\cup\gamma_{-1}\cup \gamma_0^-$ under the mapping $F(z)$
has two images $w_1(\zeta)$ and $w_2(\zeta)$, which satisfy the
following conditions similar to conditions (\ref{4.2}):  %
\begin{equation} \label{4.8} %
x_1-w_1(\zeta)=w_2(\zeta)-x'_1>0 \quad {\mbox{if $\zeta\in
\gamma_1$}}, %
\end{equation} %
\begin{equation} \label{4.9} %
w_1(\zeta)=u_1(\zeta)+ih, \ \ w_2(\zeta)=u_2(\zeta)+ih_1, %
\end{equation} %
where $x'_2-u_1(\zeta)=u_2(\zeta)-x_2>0$ if $\zeta\in \gamma_0^-$,
and %
$$ 
w_1(\zeta)=u_1(\zeta)+ih, \ \ w_2(\zeta)=u_2(\zeta)+ih_1,%
$$ 
where $u_1(\zeta)-x'_2=u_2(\zeta)-x_2>0$ if $\zeta\in
\gamma_{-1}$.

Consider four straight lines $P_k$, $k=1,2,3,4$, where $P_2$
passes through $x'_1$ and $x_2+ih_1$, $P_3$ passes through $x_1$
and $x_2+ih_1$, $P_1$ passes through $x_1$ and is parallel to
$P_2$, and $P_4$ passes through $x'_1$ and is parallel to $P_3$.
Let $u_k+ih$ denote the point of intersection of $P_k$ and the
horizontal line $L(h)$, where $L(m)$ stands for the line
$\{w:\,\Im w=m\}$. Then the points $u_k+ih$, $k=1,2,3,4$, are
ordered in the positive direction on $L(h)$; see Fig.~10a.

Next, we consider five possible  positions for $x'_2$, which
correspond to ``non-degenerate'' cases and four positions
corresponding to ``degenerate'' cases. Fig.~10a--10i illustrate
our constructions of critical geodesics and critical geodesic
loops in all these cases. First, we will work with non-degenerate
cases, which are cases (a), (c), (e), (g), and (i)
and after that we will briefly mention degenerate cases (b), (d), (f), and (h). %
\begin{itemize} %
\item[(a)] %
$x'_2<u_1$. Then the slit strip $S_1$ contains four intervals:
$(x_1,x_2+ih_1)$, $(x'_1,x_2+ih_1)$, $(x_1,x'_2+ih)$, and
$(x'_1,x'_2+ih)$. Therefore the preimages of these intervals under
the mapping
$F(z)$ provide four distinct geodesics joining the points $p_1$ and $p_2$: %
\begin{equation} \label{4.11}%
\begin{array}{ll} %
\gamma_{12}=F^{-1}((x_1,x_2+ih_1)), &\quad
\gamma'_{12}=F^{-1}((x'_1,x_2+ih_1)), \\
\gamma_{21}=F^{-1}((x_1,x'_2+ih)), &\quad
\gamma'_{21}=F^{-1}((x'_1,x'_2+ih)). \\
\end{array} %
\end{equation} %
In addition, there are two critical geodesic  loops: %
\begin{equation} \label{4.12} %
\gamma_\infty=F^{-1}((x_1,x'_1)) \quad {\mbox{and }} \quad
\gamma_{22}=F^{-1}((x_2+ih_1,x'_2+ih)). %
\end{equation} %
The curve $\gamma_{22}\cup\{p_2\}$ bounds a simply connected
domain, call it $D_{-1}$, which contains the trajectory $\gamma_2$
and the pole $z=-1$.

One more critical geodesic loop can be found as follows. Let $P_5$
be the line through $x'_2+ih$ that is parallel to $P_1$ and let
$u'_5$ be the point of intersection of $P_5$ with the real axis.
It follows from elementary geometry that there exists a point
$u_5$, $u'_5<u_5<x_1$ such that the  line segments $[x'_2+ih,u_5]$
and $[u_6, x_2+ih_1]$ with $u_6=x'_1+x_1-u_5$ are parallel to each
other. Therefore, it follows from equation (\ref{4.8})  that the
preimage $\gamma'_{22}=F^{-1}((x'_2+ih,u_5]\cup [u_6,x_2+ih_1))$
is a geodesic loop from $p_2$ to $p_2$ containing the pole $z=1$
in its interior domain.

We claim that there no other simple critical geodesics in this
case. The proof is by the method of construction of divergent
geodesic segments. An example of such construction for the case
under consideration is shown in Fig.~9b.

Suppose that $\tau$ is a geodesic ray issuing from $p_1$ into the
region $\widetilde{G}$. Let $\tau_k$, $k=1,\ldots,N$, where $N$ is
a finite integer or $N=\infty$, be connected component of
$\tau\cap \widetilde{G}$ enumerated in the natural order on
$\tau$. Let $l_k=F(\tau_k)$ and let $e'_k$ and $e''_k$ be the
initial and terminal points of $l_k$, respectively.

The interval $l_1$ may start at $x_1$ or at $x'_1$. To be
definite, assume that $e'_1=x_1$. Then for $e''_1$ we have the
following cases: %
\begin{itemize} %
\item[$(\alpha$)] %
$e''_1=x'_2-d_1+ih$ with some $d_1>0$,

\item[$(\beta$)] %
$e''_1=x'_2+d_1+ih$ with some $d_1>0$,
\item[$(\gamma$)] %
$e''_1=x_2+d_1+ih_1$ with some $d_1>0$. %
\end{itemize} %

We give a proof for the case $\alpha$). In two other case the
proof is similar. By (\ref{4.9}), $e'_2=x_2+d_1+ih_1$ and
$e''_2>x'_1$. Let $d=e''_2-x'_1$. Continuing, we find the
following expressions for
the end points of the segments $l_k$: %
$$ %
\begin{array}{ll} %
e'_{2k-1}=x_1+(k-1)d, &\quad e''_{2k-1}=x'_2+d_1+(k-1)d+ih,\\
e'_{2k}=x_2+d_1+(k-1)d+ih_1, &\quad e''_{2k}=x'_1+kd. %
\end{array} %
$$ %
Thus, in this case $\tau$ cannot terminate at $p_2$. Instead, it
approaches to the pole $z=1$ as a logarithmic spiral. %

\item[(c)] %
$u_1<x'_2<u_2$. In this case we still have geodesics (\ref{4.11})
and loops (\ref{4.12}). The only difference is that we cannot
construct the loop $\gamma'_{22}$ as in part (a). Instead, we can
construct a loop $\gamma'_{11}$ from $p_1$ to $p_1$. Indeed, using
elementary geometry, we easily find that there is a point $u_7+ih$
with $u_7<x'_2$ such that the segments $[x_1,u_7+ih]$ and
$[u_8+ih_1,x'_1]$ with $u_8=x_2+x'_2-u_7$ are parallel. Therefore
using (\ref{4.9}), we conclude that
$\gamma'_{11}=F^{-1}((x_1,u_7+ih]\cup [u_8+ih_1,x_1))$ is a
critical geodesic loop.

\item[(e)] %
$u_2<x'_2<u_3$. We still have geodesics $\gamma_{12}$,
$\gamma'_{12}$, and $\gamma_{21}$ given by (\ref{4.11}) and the
loops $\gamma_\infty$, $\gamma_{22}$, and $\gamma'_{11}$ as in the
case c). But the geodesic $\gamma'_{21}$ in (\ref{4.11}) should be
replaced with a geodesic constructed as follows. From elementary
geometry we find that there is $u_9>x_2$ such that the segments
$[x'_1,u_9+ih_1]$ and $[u_{10}+ih,x_2+ih_1]$ with
$u_{10}=x'_2-u_9+x_2$ are parallel. Using (\ref{4.9}), we conclude
that the arc $\gamma'_{21}=F^{-1}((x'_1,u_9+ih_1]\cup
[u_{10}+ih,x_2+ih_1))$ is a geodesic from $p_1$ to $p_2$.

\item[(g)] %
$u_3<x'_2<u_4$. The geodesics $\gamma_{12}$, $\gamma'_{12}$, and
$\gamma'_{21}$ and all three critical geodesic loops can be
constructed as in part (e). The geodesic $\gamma_{21}$ in this
case can be constructed as follows. Using elementary geometry one
can find that there is $u_{11}>x_2$ such that the segments
$[x_1,u_{11}+ih_1]$ and $[u_{12}+ih,x_2+ih_1]$ with
$u_{12}=x'_2+x_2-u_{11}$ are parallel. Using (\ref{4.9}) we
conclude that the arc $\gamma_{21}=F^{-1}((x_1,u_{11}+ih_1]\cup
[u_{12}+ih,x_2+ih_1))$ is a geodesic from $p_1$ to $p_2$.

\item[(i)] %
$x'_2>u_4$. The geodesics from $p_1$ to $p_2$ can be constructed
as in case (g). Of course, we still have loops (\ref{4.12}). The
third geodesic critical loop can be obtained as follows. For
$u_{13}<x_1=0$, let $l^1$ be the line segment joining the real
axis and the line $L(h)$, which has its initial point at
$z=u_{13}$ and passes through $z=x_2+i$. Let $z=u_{14}+ih$ be the
terminal point of $l^1$ on $L(h)$. We consider only those values
of $u_{13}$, for which $u_{14}<x'_2$. Let $d=x'_2-u_{14}$ and let
$l^2$ be a line segment joining the real axis and $L(h_1)$, which
is parallel to $l^1$ and has its initial point at $u_{15}=x'_1+d$.
Let $z=u_{16}+ih_1$ be the terminal point of $l^2$ on $L(h_1)$. It
follows from elementary geometry that
we can find a unique value  of $u_{13}$ such that for this value %
$u_{16}-x_2=x'_2-u_{14}$.

It follows from our construction and from the identification
properties (\ref{4.8}) and (\ref{4.9}) that the preimage %
$$ %
\gamma'_{22}=F^{-1}([u_{13},x_2+ih_1)\cup(x_2+ih_1,u_{14}+ih]\cup[u_{15},u_{16}+ih_1]) %
$$ %
is a geodesic loop from the point $p_2$ to itself. In addition,
this loop contains the pole $z=1$ in its interior, which does not
contain other critical points.

\end{itemize} %

Now we consider four ``degenerate'' cases. %
\begin{enumerate} %
\item[(b)] %
If $x'_2=u_1$, then we still have critical geodesics (\ref{4.11})
and critical geodesic
 loops (\ref{4.12}). But there is no critical geodesic loop
 separating the pole $z=1$ from other critical points. Instead,
 the boundary of a simply connected domain having $z=1$ inside
 and bounded by critical geodesics will consist of geodesics
 $\gamma'_{12}$ and $\gamma_{22}$.
\item[(d)] %
 If $x'_2=u_2$, then we have all critical geodesic loops and
 geodesics $\gamma_{12}$, $\gamma'_{12}$, and $\gamma_{21}$ as in
 the case $u_1<x'_2<u_2$ but instead of geodesic $\gamma'_{21}$
 we have a non-simple geodesic, which is the union
 $\gamma'_{12}\cup \gamma_{22}$.

\item[(f)] %
 If $x'_2=u_3$, then we have all critical geodesic loops
and geodesics $\gamma_{12}$, $\gamma'_{12}$, and $\gamma'_{21}$ as
in the case $u_2<x'_2<u_3$ but instead of geodesic $\gamma_{21}$
we have a non-simple geodesic, which is the union
$\gamma_{12}\cup\gamma_{22}$.

\item[(h)] %
 If $x'_2=u_4$, then we have all geodesics and loops
$\gamma_\infty$, $\gamma_{22}$ constructed as in the case
$u_3<x'_2<u_4$ but instead of the loop $\gamma'_{11}$ we will have
non-simple critical geodesic separating the pole $z=1$ from all
other critical points. This non-simple critical geodesic is the
union $\gamma_{12}\cup \gamma'_{21}$.

\end{enumerate} %

Using the proof by construction of divergent geodesic segments one
can show that in all cases considered above there are no any other
critical
geodesics or critical geodesic loops.%

\medskip

Quadratic differentials defined by formula (\ref{6.1}) depend on
four real parameters which are real parts and imaginary parts of
zeroes $p_1$ and $p_2$. As the reader may noticed in the generic
case configurations shown in Figures~10 also depend on four real
parameters which are  $x_2$, $x'_2$, $h_1$, and $h$. This is not a
coincidence; in fact, the set of pairs $(p_1,p_2)$ is in a
one-to-one correspondence with the set of these diagrams. To
explain how this one-to-one correspondence works, we will show
three basic steps. To be definite, we assume that the domain
configuration consists of a circle domain $D_\infty$ and strip
domains $G_1$ and $G_2$. Thus, we will consider diagrams shown in
Figures~10.

\begin{enumerate} %
\item[$\bullet$] %
As we described above, for any given $p_1$ and $p_2$, the function
$F(z)$ defined by (\ref{4.1}) maps $G_1$ and $G_2$ onto horizontal
strips shown in Figures~10. Furthermore, for fixed $p_1$ and
$p_2$, the values of the parameters  $x_2$, $x'_2$, $h_1$, and $h$
are uniquely defined via function  $F(z)$.
\item[$\bullet$] %
To prove that different pairs $(p_1,p_2)$ define different
diagrams, we argue by contradiction. Suppose that mappings
$F_1(z)$ and $F_2(z)$ constructed by formula (\ref{4.1}) for
distinct pairs $(p_1^1,p_2^1)$ and $(p_1^2,p_2^2)$ produce
identical diagrams  of the form shown in Figures~10. Then the
composition $\varphi=F_1^{-1}\circ F_2$ is well-defined and
defines a one-to-one meromorphic mapping from
$\overline{\mathbb{C}}$ onto itself. Since $\varphi(1)=1$,
$\varphi(-1)=-1$, and $\varphi(\infty)=\infty$ we conclude that
$\varphi$ is the identity mapping. Thus, $\varphi(z)
\equiv z$ and therefore $p_1^1=p_1^2$ and $p_2^1=p_2^2$. %
\item[$\bullet$] %
Now, we want to show that every diagram of the form shown in
Fig.~10a--10i corresponds via  a mapping defined by formula
(\ref{4.1}) to a quadratic differential of the form (\ref{6.1})
with some $p_1$ and $p_2$.

To show this, we will construct a compact Riemann surface
$\mathcal{R}$ using identification of appropriate edges of the
diagram. For more general quadratic differentials, similar
construction was used in \cite{S2}.

\smallskip

To be definite, we will give detailed construction for the diagram
shown in Fig.~10a. In all other cases constructions of an
appropriate Riemann surface follow same lines.
 Consider a domain $\Omega$ defined by %
$$ %
\begin{array}{ll} %
\Omega =&\{w:\,x_1<\Re w<x'_1,\, \Im w\le 0\} \cup  \\ %
{}&\{w:\,0<\Im
w<h\} \setminus \{w=t+ih_1:\,t\ge x_2\}. %
\end{array} %
$$ %
Thus, $\Omega$ is a slit horizontal strip shown in Fig.~10a with a
vertical half strip $\{w:\,x_1<\Re w<x'_1,\, \Im w\le 0\}$
attached to this horizontal strip along the interval $(x_1,x'_1)$;
see Fig.~11. To construct  a Riemann surface $\mathcal{R}$
mentioned above, we
identify boundary points of $\Omega$ as follows: %
\begin{equation} \label{8.01}%
\begin{array}{rll} %
iy &\simeq 1+iy & {\mbox{for $y\le 0$,}} \\
-x&\simeq 1+x & {\mbox{for $x\ge 0$,}} \\ %
 x+x_2+i(h_1-0)&\simeq -x+x'_2+ih&
{\mbox{for $x\ge 0$,}} \\ %
 x+x_2+i(h_1+0)&\simeq x+x'_2+ih&
{\mbox{for $x\ge 0$.}}  %
\end{array} %
\end{equation} %

After identifying points by rules (\ref{8.01}), we obtain a
surface, which is homeomorphic to a complex sphere
$\overline{\mathbb{C}}$ punctured at three points. These punctures
correspond boundary points of $\Omega$ situated at $\infty$. One
puncture corresponds to the point of $\partial \Omega$, we call it
$b_1$, which is  accessible along the path $\{z=\frac{1}{2}+it\}$
as $t\to -\infty$. Second puncture corresponds to a point $b_2$ in
$\partial \Omega$, which is  accessible along the path
$\{z=t+i\frac{h_1+h}{2}\}$ as $t \to \infty$. The third puncture
corresponds to two  boundary points of $\Omega$; one of them, we
call it $b_3^1$,  is accessible along the path $\{z=t+ih_1\}$ as
$t\to -\infty$ and the other one, we call it $b_3^2$,  is
accessible along the path $\{z=t+\frac{h_1}{2}\}$ as $t \to
\infty$. Adding these three punctures, we obtain a compact surface
$\mathcal{R}$ which is homeomorphic to a sphere
$\overline{\mathbb{C}}$.

Next, we introduce a complex structure on $\mathcal{R}$ as
follows. Every point of $\mathcal{R}$ corresponding to a point of
$\Omega$ inherits its complex structure from $\Omega$ as a subset
of $\mathbb{C}$. A point of $\mathcal{R}$ corresponding to $iy$
inherits its complex structure from two half-disks
$\{z:\,|z-iy|<\varepsilon, -\pi/2\le \arg(z-iy)\le \pi/2\}$ and
$\{z:\,|z-(1+iy)|<\varepsilon, \pi/2\le \arg(z-iy)\le 3\pi/2\}$.
Similarly, every point of $\mathcal{R}$ corresponding to a finite
boundary point of $\Omega$, except  those which corresponds to the
points $x_1$, and $x_2+ih_1$, inherits its complex structure from
the corresponding boundary half-disks.

Now we assign complex charts for five remaining special points.
For a point $x_1\simeq x'_1$ a complex chart can be assigned as
follows: %
\begin{equation} \label{8.02} %
\zeta=\left\{ \begin{array}{ll} %
(w-1)^{\frac{2}{3}} & {\mbox{if $|w-1|<\varepsilon$, $0\le \arg w
\le \frac{3\pi}{2}$,}} \\ %
(-w)^{\frac{2}{3}} & {\mbox{if $|w|<\varepsilon$,
$-\frac{\pi}{2}\le \arg w \le \pi$,}} %
\end{array} %
\right.
\end{equation} %
where the branches of the radicals are taken such that
$\zeta(w)>0$ when $w$ is real such that $w>1$ or  $w<0$.

Similarly, to assign a complex chart to a point $x_2+ih_1\simeq
x'_2+i h$, we use the following mapping: %

\begin{equation} \label{8.03} %
\zeta=\left\{ \begin{array}{ll} %
(w-(x_2+ih_1))^{\frac{2}{3}} & {\mbox{if
$|w-(x_2+ih_1)|<\varepsilon$,}} \\ %
{} &{\ \ \mbox{ $0\le \arg (w-(x_2+ih_1))
\le 2\pi$,}} \\ %
(w-(x'_2+i h))^{\frac{2}{3}} & {\mbox{if
$|w-(x'_2+ih)|<\varepsilon$,}} \\ %
{} & {\ \ \ \mbox{$\pi\le \arg (w-(x'_2+ih)) \le 2\pi$,}} %
\end{array} %
\right.
\end{equation} %
with appropriate branches of the radicals.

To a point of $\mathcal{R}$ corresponding to an infinite  boundary
point
$b_1$, a complex chart can be assigned via the function %
\begin{equation} \label{8.04} %
\zeta=\exp(-2\pi i w) \quad {\mbox{for $w$ such that $0\le \Re
w\le 1$, $\Im w<0$,}} %
\end{equation} %
which maps the half-strip $\{w:\,0\le \Re w\le 1,\, \Im w<0\}$
onto the unit disc punctured at $\zeta=0$. This mapping respects
the first identification rule in (\ref{8.01}) and  the origin
$\zeta=0$  represents the point $b_1$.

To assign a complex chart to a puncture corresponding to a pair of
boundary points $b_3^1$ and $b_3^2$, we will work with horizontal
half-strips $H_3^1$ and $H_3^2$ defined as follows. The boundary
of  $H_3^1$ consists of two horizontal rays $\{w:\,w=t:\,t\ge
u_6\}$ and $\{w=t+ih_1:\,t\ge x_2\}$ and a line segment
$[u_6,x_2+ih_1]$; the boundary of $H_3^2$ consists of two
horizontal rays $\{w:\,w=t:\,t\le u_5\}$ and $\{w=t+ih:\,t\le
x'_2\}$ and a line segment $[u_5,x'_2+ih]$. To construct a
required chart, we rotate the half-strip $H_3^1$ by angle $\pi$
with respect to the point $w=1/2$ and then we glue the result to
the half-strip $H_3^2$ along the interval $(-\infty,u_5)$.  As a
result, we obtain a wider half-strip $\widetilde{H}_3$  the
boundary of which consists of horizontal rays
$\{w=t+ih:\,t<x'_2\}$ and $\{w=t-ih_1:\, t<1-x_2\}$ and a line
segment $[1-x_2-ih_1,x'_2+ih]$. After that we map an obtained
wider half-strip $\widetilde{H}_3$ conformally onto the unit disk
in such a way that horizontal rays are mapped onto appropriate
logarithmic spirals. The conformal mapping just described can be
expressed explicitly in the following form: %
\begin{equation} \label{8.05} %
\zeta=\left\{ \begin{array}{ll} %
\exp(2\pi i C_3(1-u_5-w))  & {\mbox{if $w\in H_3^1$,}} \\ %
\exp(2\pi i C_3w) & {\mbox{if $w\in H_3^2$,}} %
\end{array} %
\right.
\end{equation} %
where 
$$ 
C_3=\frac{(x_2+x'_2-1)-i(h+h_1)}{|(x_2+x'_2-1)-i(h+h_1)|^2}. %
$$ 

In a similar way we can assign a complex chart to the puncture
corresponding to the boundary point $b_2$. In this case, we use
the following mapping from the horizontal half-strip $H_2$, the
boundary of which consists of the rays $\{w=t+ih_1:\,t\ge x_2\}$
and $\{w=t+ih:\,t\ge x'_2\}$ and a line segment
$[x_2+ih_1,x'_2+ih]$,
onto the unit disk: %
\begin{equation} \label{8.06} %
\zeta=\exp(-2\pi i C_2(w-(x_2+ih_1))) \quad {\mbox{for  $w\in
H_2$,}}
\end{equation} %
where %
$$ 
C_2=\frac{(x'_2-x_2)-i(h-h_1)}{|(x'_2-x_2)-i(h-h_1)|^2}. %
$$ 
\end{enumerate} %

Now, our compact surface $\mathcal{R}$ with conformal structure
introduced above is conformally equivalent to the Riemann sphere
$\overline{\mathbb{C}}$.  Let $\Phi(w)$ be a conformal mapping
from $\mathcal{R}$ onto $\overline{\mathbb{C}}$ uniquely determined by conditions %
$$ 
\Phi(b_1)=\infty, \quad \Phi(b_2)=1, \quad
\Phi(b_3^1)=\Phi(b_3^2)=-1. %
$$ 

Next, we consider a quadratic differential $\mathcal{Q}(w)\,dw^2$
on $\mathcal{R}$ defined by %
\begin{equation} \label{8.13} %
\mathcal{Q}(w)\, dw^2 =1\cdot dw^2%
\end{equation} %
if $w$ is finite and $w\not= x_1$ and  $w\not=x_2+ih_1$.  This
quadratic differential can be extended to the  points $w=x_1$ and
$w=x_2+ih_1$ as a quadratic differential having simple zeroes at
these points in terms of the local parameters defined by formulas
(\ref{8.02}) and (\ref{8.03}), respectively.

Similarly, using local parameters defined by formulas
(\ref{8.04}), (\ref{8.05}), and (\ref{8.06}), we can extend
quadratic differential (\ref{8.13}) to the points of $\mathcal{R}$
corresponding to the infinite boundary points of $\Omega$ situated
at $b_1$ $b_2$, and $b_3^1\simeq b_3^2$, respectively.

We note that the horizontal strips $\{w:\,0<\Im w<h_1\}$ and
$\{w:\,h_1<\Im w<h\}$ are strip domains of the quadratic
differential (\ref{8.13}), while the half-strip $\{w:\,0\le \Re
w\le 1,\, \Im w<0\}$, which boundary points are identified by the
first rule in (\ref{8.01}), defines a circle domain of this
quadratic differential.

Now, when the quadratic differential (\ref{8.13}) have been
extended to a quadratic differential defined on the whole Riemann
surface $\mathcal{R}$, we may use conformal mapping $z=\Phi(w)$ to
transplant this quadratic differential to get a quadratic
differential $\widehat{Q}(z)\,dz^2$  defined on
$\overline{\mathbb{C}}$. Since critical points of a quadratic
differential are invariant under conformal mapping, it follows
that $\widehat{Q}(z)\,dz^2$ has second order poles at  the points
$z=\infty$, $z=1$ and $z=-1$ and it has simple zeroes at the
images $\Phi(x_1)$ and $\Phi(x_2+ih_1)$ of the points $w=x_1$ and
$w=x_2+ih_1$.

Furthermore, the pole $z=\infty$ belongs to a circle domain of
$\widehat{Q}(z)\,dz^2$  and every trajectory in this circle domain
has length $1$.  Using the above information, we conclude that
$\widehat{Q}(z)\,dz^2=\frac{1}{4\pi^2}Q(z)\,dz^2$, where
$Q(z)\,dz^2$ is given by formula (\ref{6.1}) with $p_1=\Phi(x_1)$
and $p_2=\Phi(x_2+ih_1)$.

Combining our observations made in this section, we conclude the
following:

\emph{Every quadratic differential of the form (\ref{6.1}) having
two strip domains  generates a diagram of the type shown in
Fig.~10a--10i and every diagram of this type corresponds to one
and only one quadratic differential with two strip domains in its
domain configuration of the form (\ref{6.1}).}

\section{How parameters count critical geodesics and critical loops}  \label{Section-9}
\setcounter{equation}{0}

In Section~8, we described $Q$-geodesics corresponding to the
quadratic differential (\ref{6.1}) in terms of Euclidean geodesics
in the $w$-plane. In this section, we explain how this information
can be used to find the number of short geodesics and geodesic
loops for each pair of zeros $p_1$  and $p_2$.

To be definite, we will work with the case \textbf{6.3(b2)} of
Theorem~4 assuming that %
\begin{equation} \label{9.1} %
\Im p_1>0,\quad  {\mbox{and $p_2\in
E_{-1}^+(p_1)$.}} %
\end{equation} %
In all other cases, the number of short geodesics and geodesic
loops can be found similarly.

Under conditions (\ref{9.1}), the domain configuration of the
quadratic differential (\ref{6.1}) consists of domains $D_\infty$,
$G_1$, and $G_2$ as it is shown in Fig.~4a and Fig.~4b and
possible configurations of images of $G_1$ and $G_2$ under the
mapping (\ref{4.1}) are shown in Fig.~10a-10i.

Let $\varepsilon>0$ be sufficiently small and let
$dz_\varepsilon^+$ denote a tangent vector to the trajectory of
the quadratic differential (\ref{6.1}) at $z=1+\varepsilon$, which
can be found from the equation $Q(z)\,dz^2>0$. Using (\ref{7.1.1})
and (\ref{5.1}), we find that %
\begin{equation}  \label{9.2} %
\arg(dz_\varepsilon^+)=\frac{\pi}{2}-\frac{1}{2}\arg
C_1+o(1)=\frac{\pi}{2}-\frac{1}{2} \arg((p_1-1)(p_2-1))+o(1), %
\end{equation} %
where $o(1)\to 0$ as $\varepsilon\to 0$. We assume here that
$-\frac{\pi}{2}\le \arg(dz_\varepsilon^+)\le \frac{\pi}{2}$.

If $1+\varepsilon\in \gamma_1$ then the tangent vector
$dz_\varepsilon^+$ corresponds to the direction on $\gamma_1$ from
$z=1$ to $z=p_1$. Let $\alpha_\varepsilon^+=\alpha^++o(1)$, where
$\alpha^+$  is a constant such that $0\le \alpha^+\le \pi$, denote
the angle formed at the point $1+\varepsilon \in \gamma_1$ by
$dz_\varepsilon ^+$ and the vector $\overrightarrow{v}=-i$, which
is tangent to the circle $\{z:\,|z-1|=\varepsilon\}$ at
$z=1+\varepsilon$. It follows from (\ref{9.2}) that %
\begin{equation} \label{9.3} %
\alpha^+=\pi-\frac{1}{2}\arg C_1=\pi-\frac{1}{2}\arg((p_1-1)(p_2-1)). %
\end{equation} %

Similarly, if $dz_\varepsilon^-$ denote the tangent vector to the
trajectory of the quadratic differential (\ref{6.1}) at
$z=-1+\varepsilon$, then %
\begin{equation}  \label{9.4} %
\arg(dz_\varepsilon^-)=\frac{\pi}{2}-\frac{1}{2}\arg
C_{-1}+o(1)=\frac{\pi}{2}-\frac{1}{2} \arg((p_1+1)(p_2+1))+o(1). %
\end{equation} %

Suppose that $1+\varepsilon\in \gamma_{-1}$ and that
$d_\varepsilon^-$ shows direction on $\gamma_{-1}$ from $z=-1$ to
$z=p_2$. As before we can find constant $\alpha^-$, $0\le \alpha^-
\le \pi$, such that the angle formed at $z=-1+\varepsilon\in
\gamma_{-1}$ by the vectors $dz_\varepsilon^+$ and
$\overrightarrow{v}=-i$ is equal to $\alpha^-+o(1)$, where
$o(1)\to 0$ as $\varepsilon\to 0$ and %
\begin{equation} \label{9.5} %
\alpha^-=\pi-\frac{1}{2}\arg C_{-1}=\pi-\frac{1}{2}\arg((p_1+1)(p_2+1)). %
\end{equation} %

To relate angles $\alpha^+$ and $\alpha^-$ to geometric
characteristics of diagrams in Fig.~10a-10i, we recall that
geodesics are conformally invariant and that for small
$\varepsilon>0$ a geodesic loop $\gamma_\varepsilon^+$ which
passes through the point $z=1+\varepsilon$ and surrounds the pole
$z=1$ is an infinitesimal circle. Therefore the angle formed by
the vector $dz_\varepsilon^+$ and the tangent vector to
$\gamma_\varepsilon^+$ at $z=1+\varepsilon$ equals
$\alpha^++o(1)$.

Similarly, the angle formed by the vector $dz_\varepsilon^-$ and
the tangent vector to the corresponding geodesic loop
$\gamma_\varepsilon^-\ni -1+\varepsilon$ surrounding the pole at
$z=-1$ is equal to $\alpha^-+o(1)$.

Since geodesics are conformally invariant and since conformal
mappings preserve angles, we conclude that trajectories of the
quadratic differential $\mathcal{Q}(w)\,dw^2$  defined in
Section~8 (see formula (\ref{8.13}) ) form angles of opening
$\alpha^+$ or $\alpha^-$ with the images of the corresponding
geodesic loops $\gamma_\varepsilon^+$ or $\gamma_\varepsilon^-$,
respectively. Since the metric defined by the quadratic
differential (\ref{8.13}) is Euclidean, it follows that the
corresponding images of geodesic loops are line segments joining
pairs of points identified by relations (\ref{8.01}).

Using this observation and identification rule $-x+x'_2+ih\simeq
x+x_2+ih_1$, we conclude that the segment $[x_2+ih_1,x'_2+ih]$
forms an angle $\pi-\alpha^-$ with the positive real axis; i.e., %
\begin{equation} \label{9.6} %
\pi-\alpha^-=\arg((x'_2-x_2)+i(h-h_1)). %
\end{equation} %

To find an equation for the angle $\alpha^+$, we will use the
half-strip $\widetilde{H}_3$  constructed at the end of Section~8,
which is related to a conformal mapping defined by
formula~(\ref{8.05}). In this case, $\pi-\alpha^+$ is equal to the
angle formed by the segment $[1-x_2-ih_1,x'_2+ih]$ with the
positive real axis; i.e., %
\begin{equation} \label{9.7} %
\pi-\alpha^+=\arg((x_2+x'_2-1)+i(h+h_1)). %
\end{equation} %

Equating the right-hand sides of equations (\ref{9.3}) and
(\ref{9.4}) to the right-hand sides of equations (\ref{9.7}) and
(\ref{9.6}), respectively, we obtain two equations, which relate
parameters $x_2$, $x'_2$, $h_1$, and $h$. Combining this with
equations (\ref{5.10})--(\ref{7.9.1}), we obtain the following
system of four equations: %
$$ 
\begin{array}{l} %
\arg((x_2+x'_2-1)+i(h+h_1))=\frac{1}{2}\arg ((p_1-1)(p_2-1)) \\ %
\arg((x'_2-x_2)+i(h-h_1))=\frac{1}{2}\arg ((p_1+1)(p_2+1)) \\ %
h_1=\frac{1}{4} \Im \left(\sqrt{(p_1-1)(p_2-1)}-\sqrt{(p_1+1)(p_2+1)}\right) \\ %
h=\frac{1}{4} \Im \left(\sqrt{(p_1-1)(p_2-1)}+\sqrt{(p_1+1)(p_2+1)}\right). %
\end{array}
$$ 
This system of equations can be solved to obtain the following: %
\begin{equation} \label{9.9} %
\begin{array}{l} %
x_2+ih_1=\frac{1}{2}+\frac{1}{4}
\left(\sqrt{(p_1-1)(p_2-1)}-\sqrt{(p_1+1)(p_2+1)}\right), \\ %
x'_2+ih=\frac{1}{2}+\frac{1}{4}
\left(\sqrt{(p_1-1)(p_2-1)}+\sqrt{(p_1+1)(p_2+1)}\right).
\end{array} %
\end{equation} %


Now, when the points $x_2+ih_1$ and $x'_2+ih$ are determined, we
can give explicit conditions on  the zeros $p_1$ and $p_2$ which
correspond to all subcases (a)--(i) of the case \textbf{6.3(b2)}
discussed in Section~8.

\begin{theorem} \label{Theorem-5} %
Suppose that zeros $p_1$ and $p_2$ satisfy conditions (\ref{9.1}).
Then the number of short geodesics and geodesic loops and their
topology are determined by the following inequalities, which
corresponds to the subcases (a)--(i) of Case \textbf{6.3(b2)} described in Section~8 and shown in Fig.~10a--10i: %

Case {\rm{(a)}} with four short geodesics and three critical geodesic loops occurs if the following conditions are satisfied: %
$$ %
\begin{array}{ll} %
0&<\arg(-\frac{1}{2}+\frac{1}{4}(\sqrt{(p_1-1)(p_2-1)}-\sqrt{(p_1+1)(p_2+1)})
\\   %
{}&<\arg(\frac{1}{2}+\frac{1}{4}(\sqrt{(p_1-1)(p_2-1)}+\sqrt{(p_1+1)(p_2+1)})<\pi.  %
\end{array} %
$$ %

Case {\rm{(b)}} with four short geodesics and two critical geodesic loops occurs if the following conditions are satisfied: %
$$ %
\begin{array}{ll} %
0&<\arg(-\frac{1}{2}+\frac{1}{4}(\sqrt{(p_1-1)(p_2-1)}-\sqrt{(p_1+1)(p_2+1)})
\\   %
{}&=\arg(\frac{1}{2}+\frac{1}{4}(\sqrt{(p_1-1)(p_2-1)}+\sqrt{(p_1+1)(p_2+1)})<\pi.  %
\end{array} %
$$ %

Case {\rm{(c)}} with four short geodesics and three critical geodesic loops occurs if the following conditions are satisfied: %
$$ %
\begin{array}{ll} %
0&<\arg(\frac{1}{2}+\frac{1}{4}(\sqrt{(p_1-1)(p_2-1)}+\sqrt{(p_1+1)(p_2+1)})
\\   %
{}&<\arg(-\frac{1}{2}+\frac{1}{4}(\sqrt{(p_1-1)(p_2-1)}-\sqrt{(p_1+1)(p_2+1)})<\pi,  %
\end{array} %
$$ %
$$ %
\begin{array}{ll} %
0&<\arg(-\frac{1}{2}+\frac{1}{4}(\sqrt{(p_1-1)(p_2-1)}-\sqrt{(p_1+1)(p_2+1)})
\\   %
{}&<\arg(-\frac{1}{2}+\frac{1}{4}(\sqrt{(p_1-1)(p_2-1)}+\sqrt{(p_1+1)(p_2+1)})<\pi.  %
\end{array} %
$$ %

Case {\rm{(d)}} with three short geodesics and three critical geodesic loops occurs if the following conditions are satisfied: %
$$ %
\begin{array}{ll} %
0&<\arg(-\frac{1}{2}+\frac{1}{4}(\sqrt{(p_1-1)(p_2-1)}-\sqrt{(p_1+1)(p_2+1)})
\\   %
{}&=\arg(-\frac{1}{2}+\frac{1}{4}(\sqrt{(p_1-1)(p_2-1)}+\sqrt{(p_1+1)(p_2+1)})<\pi.  %
\end{array} %
$$ %

Case {\rm{(e)}} with four short geodesics and three critical geodesic loops occurs if the following conditions are satisfied: %
$$ %
\begin{array}{ll} %
0&<\arg(-\frac{1}{2}+\frac{1}{4}(\sqrt{(p_1-1)(p_2-1)}+\sqrt{(p_1+1)(p_2+1)})
\\   %
{}&<\arg(-\frac{1}{2}+\frac{1}{4}(\sqrt{(p_1-1)(p_2-1)}-\sqrt{(p_1+1)(p_2+1)})<\pi,  %
\end{array} %
$$ %
$$ %
\begin{array}{ll} %
0&<\arg(\frac{1}{2}+\frac{1}{4}(\sqrt{(p_1-1)(p_2-1)}-\sqrt{(p_1+1)(p_2+1)})
\\   %
{}&<\arg(\frac{1}{2}+\frac{1}{4}(\sqrt{(p_1-1)(p_2-1)}+\sqrt{(p_1+1)(p_2+1)})<\pi.  %
\end{array} %
$$ %

Case {\rm{(f)}} with three short geodesics and three critical geodesic loops occurs if the following conditions are satisfied: %
$$ %
\begin{array}{ll} %
0&<\arg(\frac{1}{2}+\frac{1}{4}(\sqrt{(p_1-1)(p_2-1)}-\sqrt{(p_1+1)(p_2+1)})
\\   %
{}&=\arg(\frac{1}{2}+\frac{1}{4}(\sqrt{(p_1-1)(p_2-1)}+\sqrt{(p_1+1)(p_2+1)})<\pi.  %
\end{array} %
$$ %

Case {\rm{(g)}} with four short geodesics and three critical geodesic loops occurs if the following conditions are satisfied: %
$$ %
\begin{array}{ll} %
0&<\arg(\frac{1}{2}+\frac{1}{4}(\sqrt{(p_1-1)(p_2-1)}+\sqrt{(p_1+1)(p_2+1)})
\\   %
{}&<\arg(\frac{1}{2}+\frac{1}{4}(\sqrt{(p_1-1)(p_2-1)}-\sqrt{(p_1+1)(p_2+1)})<\pi,  %
\end{array} %
$$ %
$$ %
\begin{array}{ll} %
0&<\arg(\frac{1}{2}+\frac{1}{4}(\sqrt{(p_1-1)(p_2-1)}-\sqrt{(p_1+1)(p_2+1)})
\\   %
{}&<\arg(-\frac{1}{2}+\frac{1}{4}(\sqrt{(p_1-1)(p_2-1)}+\sqrt{(p_1+1)(p_2+1)})<\pi.  %
\end{array} %
$$ %

Case {\rm{(h)}} with four short geodesics and two critical geodesic loops occurs if the following conditions are satisfied: %
$$ %
\begin{array}{ll} %
0&<\arg(\frac{1}{2}+\frac{1}{4}(\sqrt{(p_1-1)(p_2-1)}-\sqrt{(p_1+1)(p_2+1)})
\\   %
{}&=\arg(-\frac{1}{2}+\frac{1}{4}(\sqrt{(p_1-1)(p_2-1)}+\sqrt{(p_1+1)(p_2+1)})<\pi.  %
\end{array} %
$$ %

Case {\rm{(i)}} with four short geodesics and three critical geodesic loops occurs if the following conditions are satisfied: %
$$ %
\begin{array}{ll} %
0&<\arg(-\frac{1}{2}+\frac{1}{4}(\sqrt{(p_1-1)(p_2-1)}+\sqrt{(p_1+1)(p_2+1)})
\\   %
{}&<\arg(\frac{1}{2}+\frac{1}{4}(\sqrt{(p_1-1)(p_2-1)}-\sqrt{(p_1+1)(p_2+1)})<\pi.  %
\end{array} %
$$ %
\end{theorem}  %

\section{Some related questions}  \label{Section-9}
\setcounter{equation}{0}

Our results presented in Sections~6-9 provide complete information
concerning critical trajectories and $Q$-geodesic of the quadratic
differential (\ref{6.1}). This allows us to answer many related
questions. As an example,  we will discuss three questions
originated in the study of limiting  distributions of zeros of
Jacobi polynomials.

Below,  we suppose that  $p_1,p_2\in \mathbb{C}$ are fixed. Then
we  consider the family of quadratic differentials $Q_s(z)\,dz^2$
depending on the real parameter $s$, $0\le s<2\pi$, such
that %

\begin{equation} \label{10.1} %
Q_s(z)\,dz^2:=e^{-is}Q(z)\,dz^2=-e^{-is}\frac{(z-p_1)(z-p_2)}{(z-1)^2(z+1)^2}\,dz^2. %
\end{equation} %


\begin{enumerate} %

\item[1)] %
 For how many values of
$s$, $0\le s<2\pi$,  the quadratic differential $Q_s(z)\,dz^2$ has
a trajectory  loop with end points at $p_1$ and for how many
values of $s$
$Q_s(z)\,dz^2$ has a trajectory  loop with end points at $p_2$? %

\item[2)] %
 For how many values of
$s$, $0\le s<2\pi$, the corresponding quadratic differential
$Q_s(z)\,dz^2$ has a short critical trajectory? %

\item[3)] %
 How we can find the values of
$s$, $0\le s<2\pi$, mentioned in questions  stated above?
\end{enumerate} %

To answer these questions we need two simple facts: %
\begin{enumerate} %
\item[(a)] %
First, we note  that $\gamma$ is a short trajectory loop or,
respectively,  a short critical trajectory for the quadratic
differential (\ref{10.1}) with some $s$ if and only if $\gamma$ is
a short geodesic loop or, respectively, a short geodesic joining
points $p_1$ and $p_2$ for the quadratic differential (\ref{6.1}).
Thus, the numbers of values $s$ in question (1) and question (2),
respectively,  are bounded by the number of short geodesic loops
and the number of short geodesics, respectively. In the most
general case with one circle domain and two strip domains, these
short geodesic loops and short geodesics were described in
Theorem~\ref{Theorem-5} and their images under the canonical
mapping were shown in Fig.~10a-10i. Of course, one value of $s$
can correspond to more than one short geodesic loop and more than
one short geodesic.

\item[(b)] %
To find the values of $s$ in question~3), we use the following
observation. If $l$ is a straight line segment in the image domain
$\Omega$ forming an angle $\alpha$, $0\le \alpha<\pi$, with the
direction of the positive real axis, then $l$ is an image under
the canonical mapping (\ref{4.1})  of an arc of a trajectory of
the quadratic differential (\ref{10.1})
with %
\begin{equation} \label{10.2} %
s=2\alpha.
\end{equation} %
\end{enumerate} %

We will use (\ref{10.2}) to find values of $s$ which turn short
geodesic loops and short geodesics into short trajectory loops and
short trajectories, respectively. It is convenient to introduce
notations $\alpha_\infty$, $\alpha_{12}$, $\alpha'_{12}$,
$\alpha_{22}$, $\alpha'_{22}$, $\alpha''_{22}$, and so on, to
denote the angles formed by corresponding geodesics
$\gamma_\infty$, $\gamma_{12}$, $\gamma'_{12}$, $\gamma_{22}$,
$\gamma'_{22}$, $\gamma''_{22}$, and so on (considered in the
$w$-plane) with the positive direction of the real axis.
Furthermore, we will use notations  ${\mathcal{A}}(6.1)$,
${\mathcal{A}}(6.1(a))$, ${\mathcal{A}}(6.2)$,
${\mathcal{A}}(6.3(a))$, ${\mathcal{A}}(6.3(b1))$,
${\mathcal{A}}(6.3(b2)(a))$, and so on, to denote the sets of all
angles introduced above in the cases under consideration; i.e. in
the cases $\mathbf{6.1}$, $\mathbf{6.2}$, $\mathbf{6.3(a)}$,
$\mathbf{6.3(b_1)}$, $\mathbf{6.3(b2)}(a)$, and so on.

Now, we are ready to answer questions stated above. We proceed
with two steps. First, we identify the type of domain
configuration $\mathcal{D}_Q$. This will provide us with the first
portion of necessary information. We recall that in general there
are at most three geodesic loops centered at $z=\infty$, $z=1$,
and $z=-1$. Thus, the maximal number of values $s$ in question~1)
is at most three. Then we identify which of the schemes
corresponds to the parameters $p_1$, $p_2$ (in the most general
case these schemes are shown in Fig.~10a-10i). This will provide
us with the remaining portion of necessary information.

$\bullet$ \ Suppose that $\mathcal{D}_Q$ has type~\textbf{6.1}.
Then  we already have three circle domains and therefore $s=0$ is
the only value for which $Q_sz)\,dz^2$ may have short trajectory
loops. In case \textbf{6.1(a)}, we have short trajectory loops
centered at $z=1$ and $z=-1$ and no other such loops. In case
\textbf{6.1(b)} with $1<p_2<p_1$ (respectively with $p_1<p_2<-1$),
we have short trajectory loops centered  at $z=\infty$ and $z=1$
(respectively, at $z=\infty$ and $z=-1$). In case \textbf{6.1(c)},
there are no short geodesic loops.

As concerns short critical trajectories for domain configuration
of type \textbf{6.1}, again $s=0$ is the only value for which
there are such trajectories. This follows from the fact discussed
in Section~8 that in case~\textbf{6.1} there are no other simple
geodesics joining $p_1$ and $p_2$. In cases \textbf{6.1(a)} and
\textbf{6.1(b)}, there is a single short critical trajectory which
is the interval $\gamma_0=(p_2,p_1)$. In case \textbf{6.1(c)},
there are three short critical trajectories which are arcs
$\gamma_0$, $\gamma_1$, and $\gamma_{-1}$ shown in Fig.~1c.

$\bullet$ \  Next, we consider the case when $\mathcal{D}_Q$ has
type \textbf{6.2}. For $s=0$, we have two short trajectory loops.
As before, we assume that these loops surround points $z=-1$ and
$z=\infty$. In other cases discussion is similar, we just have to
switch roles of the poles of the quadratic differential
(\ref{10.1}).

In this case,
${\mathcal{A}}(6.2)=\{0,\alpha_{11},\alpha_{12},\alpha'_{12},\alpha_{21},\alpha'_{21}\}$.
One more value of $s$, for which we may have a short trajectory
loop (centered at $z=1$) may occur for
$s=2\alpha_{11}=-\arg((1-p_1)(1-p_2))$. If
$|\gamma_\infty|_Q>|\gamma_{-1}|_Q$ then we will have a short
geodesic loop from $p_1$ to $p_1$. This loop corresponds to a
geodesic $\gamma_{11}$ in Fig.~8a.  If
$|\gamma_\infty|_Q<|\gamma_{-1}|_Q$, then we will have a similar
short geodesic loop from $p_2$ to $p_2$. In the case
$|\gamma_\infty|_Q=|\gamma_{-1}|_Q$, we have
$\alpha_{11}=\alpha_{12}=\alpha'_{21}$. In this case, we do not
have the third short geodesic loop. Instead, we have two short
critical trajectories joining $p_1$ and $p_2$.

By (\ref{10.2}), the value of $s$, which corresponds to the third
loop (if it exists) is equal to $2\alpha_{11}$.  As concerns
values of $s$ corresponding to short critical trajectories, in
case \textbf{6.2} with $|\gamma_\infty|_Q\not=|\gamma_{-1}|_Q$ we
have four such values. These values are $2\alpha_{12}$,
$2\alpha'_{12}$, $2\alpha_{21}$, and $2\alpha'_{21}$ (see
Fig.~8a).

If $|\gamma_\infty|_Q=|\gamma_{-1}|_Q$, then there are three
values of $s$, which produce short geodesics from $p_1$ to $p_2$.
Two of these values, $s=2\alpha'_{12}$ and $s=2\alpha_{21}$,
generate one short critical trajectory each. The third value
$s=2\alpha_{12}$ generates two short critical trajectories.

\smallskip

$\bullet$ \  Turning to the most general case \textbf{6.3}, we
will give detailed account for subcases \textbf{6.3(b1)} and
\textbf{6.3(b2)}(i), in all other subcases consideration is
similar.

First, we consider the subcase \textbf{6.3(b1)} when the domain
configuration ${\mathcal{D}}_Q$ consists of one circle domain and
one strip domain; see Fig.~3a--3e.  In this case,
${\mathcal{A}}(6.3(b1))=\{0,\alpha'_{22},\alpha''_{22},\alpha_{12},\alpha'_{12}\}$.
The value $s=0$ generates one short trajectory loop and one short
trajectory. The values  $s=2\alpha'_{22}$ and $s=2\alpha''_{22}$
generate one short trajectory loop each and the values
$s=2\alpha_{12}$ and $s=2\alpha'_{12}$ generate one short
trajectory each.

Let us consider case \textbf{6.3(b2)}(i) shown in Fig.~10i. We
have
${\mathcal{A}}(6.3(b2)(i))=\{0,\alpha_{22},\alpha'_{22},\alpha_{12},\alpha'_{12},\alpha_{21},\alpha'_{21}\}$
where all angles are distinct. The values  $s=0$,
$s=2\alpha_{22}$, and $s=2\alpha'_{22}$ generate  short trajectory
loops $\gamma_{\infty}$, $\gamma_{22}$, and $\gamma''_{22}$,
respectively.  Remaining values  $s=2\alpha_{12}$,
$s=2\alpha'_{12}$, $s=2\alpha_{21}$, $s=2\alpha'_{21}$ generate
short trajectories $\gamma_{12}$, $\gamma'_{12}$, $\gamma_{21}$,
and $\gamma'_{21}$, respectively.

Finally, we note that position of points $x_1$, $x'_1$,
$x_2+ih_1$, and $x'_2+ih$ are given explicitly; see formulas
(\ref{9.9}). Using these formulas one can find explicit
expressions for all angles $\alpha_{12}$, $\alpha'_{12}$,
$\alpha_{21}$, $\alpha'_{21}$, and so on, in all possible cases.

\section{Figures Zoo}
\setcounter{equation}{0} %
\FloatBarrier
This section contains all our figures. For convenience, we divide
the set of all figures in eleven groups.

\medskip

\textbf{I.} Configurations with three circle domains.


\vspace{-5cm} %

\begin{figure}
$$\includegraphics[scale=.65,angle=0]{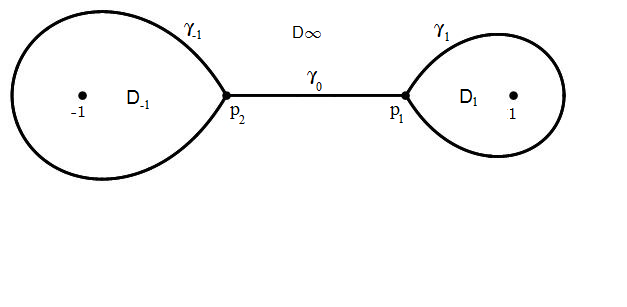} $$
\vspace{-3cm} %
\caption{1a. Three circle domains. Case \textbf{6.1(a)}.} 
\end{figure}


\vspace{3cm}

\medskip %

\begin{figure}
$$\includegraphics[scale=.75,angle=0]{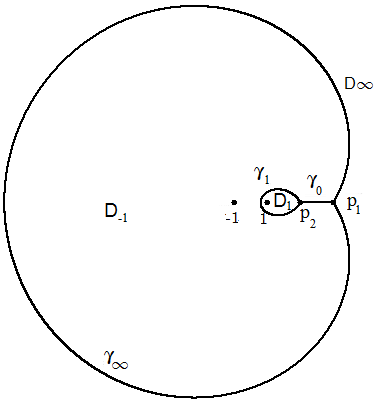} $$
\vspace{-1cm} %
\caption{1b. Three circle domains. Case \textbf{6.1(b)}.} 
\end{figure}

\medskip


\bigskip

\vspace{2cm}

\begin{figure}[b]
$$\includegraphics[scale=.75,angle=0]{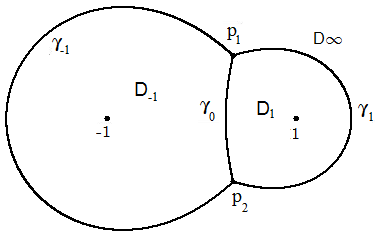} $$
\vspace{-1cm} %
\caption{1c. Three circle domains. Case \textbf{6.1(c)}.} 
\end{figure}

\medskip

\newpage

\FloatBarrier

 \textbf{II.} Configurations with two circle domains.

\begin{figure}[h]
$$\includegraphics[scale=.7,angle=0]{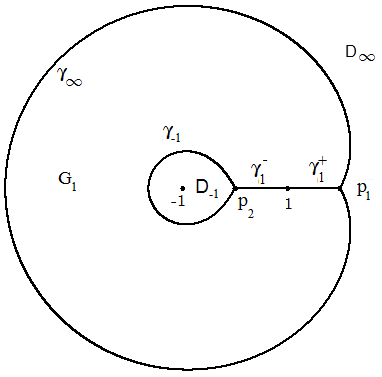} $$
\vspace{-1cm} %
\caption{2a. Two circle domains. Case \textbf{6.2} with symmetric domains.} 
\end{figure}



\begin{figure}[h]
$$\includegraphics[scale=.75,angle=0]{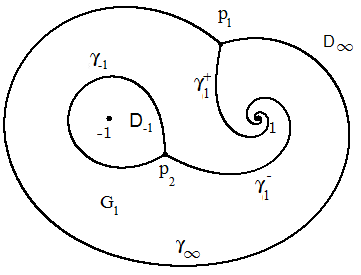} $$
\vspace{-1cm} %
\caption{2b. Two circle domains. Case \textbf{6.2} with non-symmetric domains.} 
\end{figure}

\newpage

\FloatBarrier

 \textbf{III.} Configurations with one circle
domain and one
strip domain. %


\medskip

\begin{figure}
$$\includegraphics[scale=.6,angle=0]{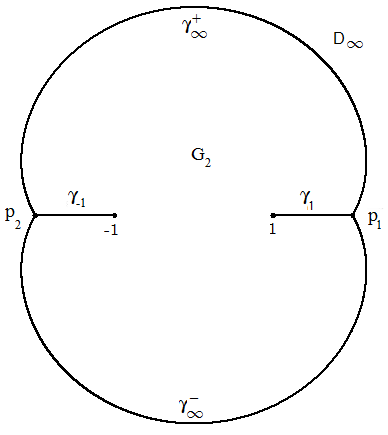} $$
\vspace{-1cm} %
\caption{3a. One circle domain. Case \textbf{6.3(a)} with axial symmetry.} 
\end{figure}


\medskip

\begin{figure}
$$\includegraphics[scale=.7,angle=0]{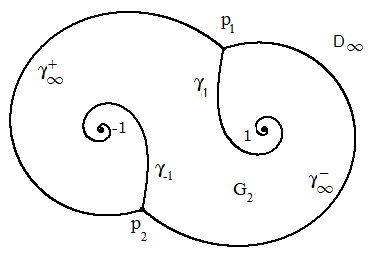} $$
\vspace{-1cm} %
\caption{3b. One circle domain. Case \textbf{6.3(a)} with central symmetry.} 
\end{figure}


\medskip

\begin{figure}
$$\includegraphics[scale=.7,angle=0]{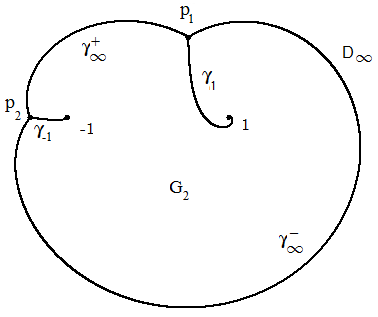} $$
\vspace{-1cm} %
\caption{3c. One circle domain. Case \textbf{6.3(a)} with non-symmetric domains.} 
\end{figure}

\newpage

\medskip

\begin{figure}
$$\includegraphics[scale=.65,angle=0]{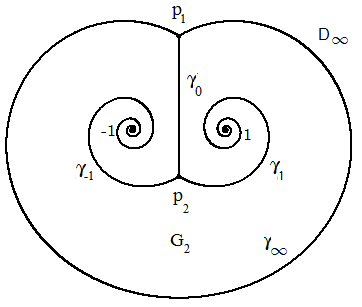} $$
\vspace{-1cm} %
\caption{3d. One circle domain. Case \textbf{6.3(b1)} with symmetric domains.} 
\end{figure}


\medskip

\begin{figure}
$$\includegraphics[scale=.65,angle=0]{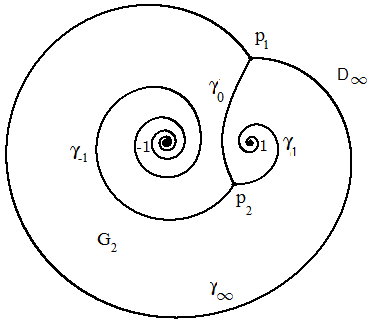} $$
\vspace{-1cm} %
\caption{3e. One circle domain. Case \textbf{6.3(b1)} with non-symmetric domains.} 
\end{figure}



\textbf{IV.} Configurations with one circle domain and two
strip domains. %

 \FloatBarrier

\medskip

\begin{figure}
$$\includegraphics[scale=.65,angle=0]{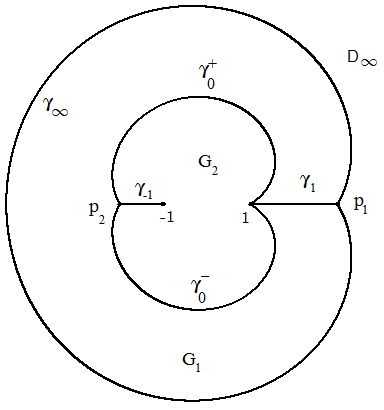} $$
\vspace{-1cm} %
\caption{4a. One circle domain. Case \textbf{6.3(b2)} with symmetric domains.} 
\end{figure}

\newpage

\medskip

\begin{figure}
$$\includegraphics[scale=.65,angle=0]{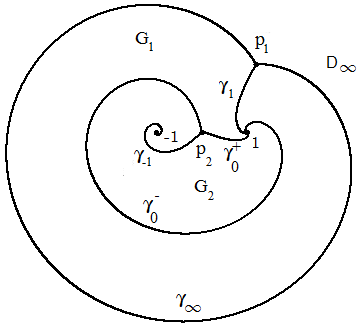} $$
\vspace{-1cm} %
\caption{4b. One circle domain. Case \textbf{6.3(b2)} with non-symmetric domains.} 
\end{figure}



\medskip

\textbf{V.} Degenerate configurations. %

 \FloatBarrier

\medskip

\begin{figure}
$$\includegraphics[scale=.55,angle=0]{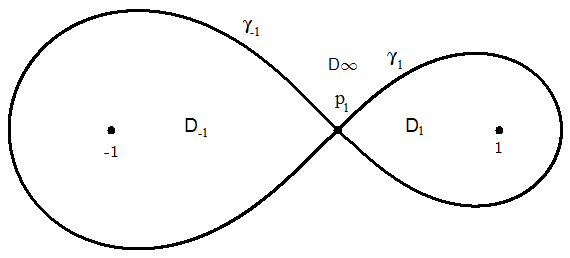} $$
\vspace{-1cm} %
\caption{5a. Degenerate case with $-1<p_1=p_2<1$.} 
\end{figure}


\medskip

\vspace{2cm}

\begin{figure}
$$\includegraphics[scale=.6,angle=0]{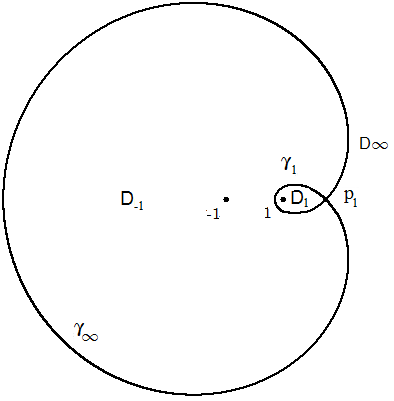} $$
\vspace{-1cm} %
\caption{5b. Degenerate case with $p_1=p_2>1$.} 
\end{figure}

\newpage

\medskip

\begin{figure}
$$\includegraphics[scale=.75,angle=0]{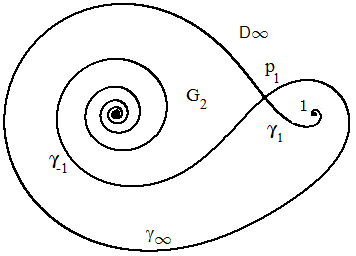} $$
\vspace{-1cm} %
\caption{5c. Degenerate case with $p_1=p_2$, $\Im p_1>0$.} 
\end{figure}

\newpage

\begin{figure}
$$\includegraphics[scale=.75,angle=0]{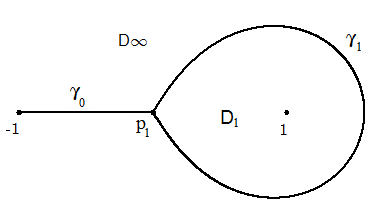} $$
\vspace{-1cm} %
\caption{5d. Degenerate case with $p_2=-1$, $-1<p_1<1$.} 
\end{figure}

\newpage

\begin{figure}
$$\includegraphics[scale=.6,angle=0]{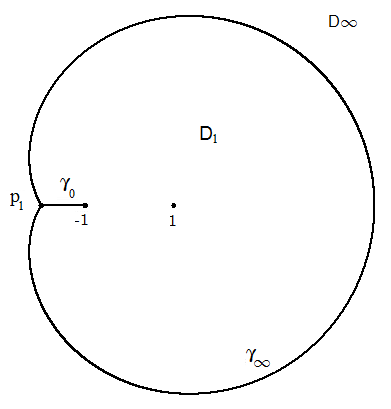} $$
\vspace{-1cm} %
\caption{5e. Degenerate case with $p_2=-1$, $p_1<-1$.} 
\end{figure}

\medskip

\begin{figure}[p]
$$\includegraphics[scale=.6,angle=0]{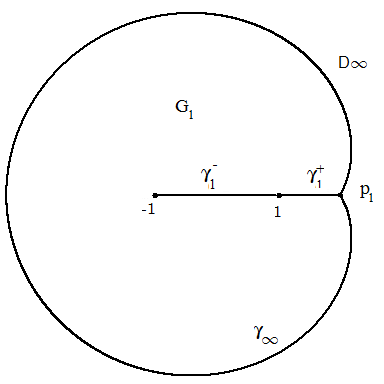} $$
\vspace{-1cm} %
\caption{5f. Degenerate case with $p_2=-1$, $p_1>1$.} 
\end{figure}

\begin{figure}[p]
$$\includegraphics[scale=.65,angle=0]{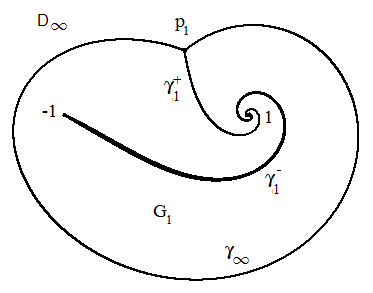} $$
\vspace{-1cm} %
\caption{5g. Degenerate case with $p_2=-1$, $\Im p_1>0$.} 
\end{figure}


\FloatBarrier

\textbf{VI.} Type regions. %

\vspace{2cm}

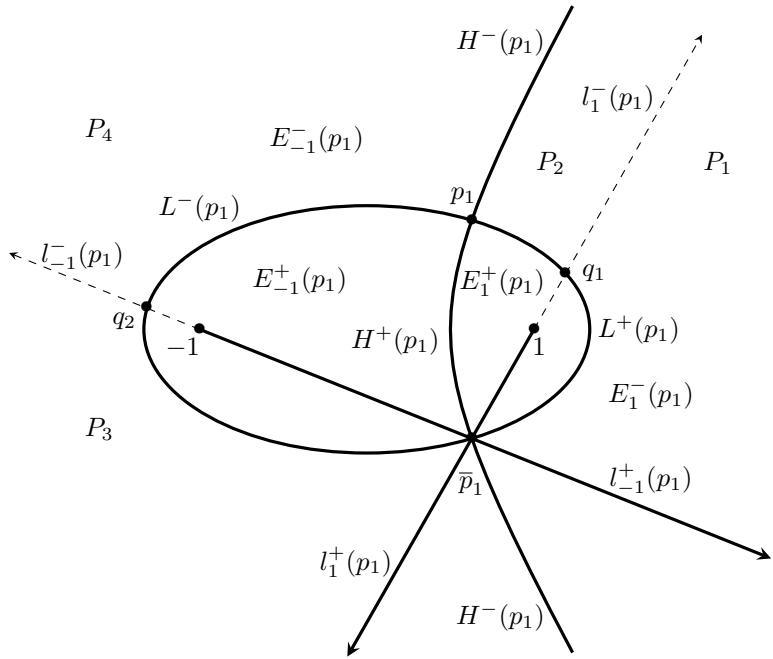
\begin{figure}[h]
\centering %
\hspace{-5.3cm} %
\begin{minipage}{0.6\linewidth}
\begin{tikzpicture}
    [inner sep=1mm,
     place/.style={circle,draw=blue!50,
     fill=blue!20,thick},
     scale=2.2]%


\draw [black, very thick](0,0) ellipse (4/3 and 0.7453559923); %
\node at (1,0) [] {$\bullet$}; %
\node at (1.03,0) [below] {$1$}; %
\node at (-1,0) [] {$\bullet$}; %
\node at (-1.1,0) [below] {$-1$}; %
\node at (1.187,0.341) [] {$\bullet$}; %
\node at (1.24,0.341) [right] {$q_1$}; %
\node at (-1.317,0.132) [] {$\bullet$}; %
\node at (-1.33,0.05) [left] {$q_2$}; %

     \draw [black, very thick] plot[domain=-1.55:1.55] ({(1/2)*(exp(\x)+exp(-\x))/2},{sqrt(3)/2*(exp(\x)-exp(-\x))/2});

   \draw [->,>=stealth,black, very thick]
   (1,0) to (1-3*0.3722022787,-3*0.6575636373);
   \draw [->,>=stealth,dashed]
   (1,0)to (1+2.7*0.3722022787,2.7*0.6575636373);

\draw [->,>=stealth,black, very thick]
   (-1,0)to (-1+2.1*1.627797721,-2.1*0.6575636373);
\draw [->,>=stealth,dashed] 
   (-1,0)to (-1-0.7*1.627797721,0.7*0.6575636373);
   \node at (0.6277977213,-0.6575636373) {$\bullet$};
    \node at (0.6277977213,-0.6575636373-0.15) [below] {$\overline{p}_1$};
    \node at (0.6277977213,0.6575636373) {$\bullet$};
 \node at (0.6277977213-0.05,0.6575636373+0.05) [above] {$p_1$}; %
 \node at (0.17,-0.2) [above] {$H^+(p_1)$}; %
  \node at (0.8,1.6) [above] {$H^-(p_1)$}; %
  \node at (0.8,-1.6) [below] {$H^-(p_1)$}; %
  \node at (4/3,0) [right] {$L^+(p_1)$}; %
  \node at (-1,0.6) [above] {$L^-(p_1)$}; %

  \node at (0.81,0.15) [above] {$E_1^+(p_1)$}; %
   \node at (-0.4,0.15) [above] {$E_{-1}^+(p_1)$}; %
\node at (1.7,-0.25) [below] {$E_1^-(p_1)$}; %
\node at (-0.3,1) [above] {$E_{-1}^-(p_1)$}; %

\node at (0.2,-1.4) [left] {$l_1^+(p_1)$}; %
\node at (1.77,1.4) [left] {$l_1^-(p_1)$}; %

\node at (1.7,-1.05) [above] {$l_{-1}^+(p_1)$}; %
\node at (-1.7,0.3) [above] {$l_{-1}^-(p_1)$}; %
\node at (2.1,1) [] {$P_1$}; %
\node at (1.1,1) [] {$P_2$}; %
\node at (-1.6,1.2) [] {$P_4$}; %
\node at (-1.6,-0.6) [] {$P_3$};

  \end{tikzpicture}
\end{minipage}
 \caption{6. Type regions.}
\end{figure}


\FloatBarrier

\vspace{2cm}

\textbf{VII.} Figures for the proof of Theorem~4. %

\vspace{1cm}
\begin{figure}[h]
$$\includegraphics[scale=.23,angle=0]{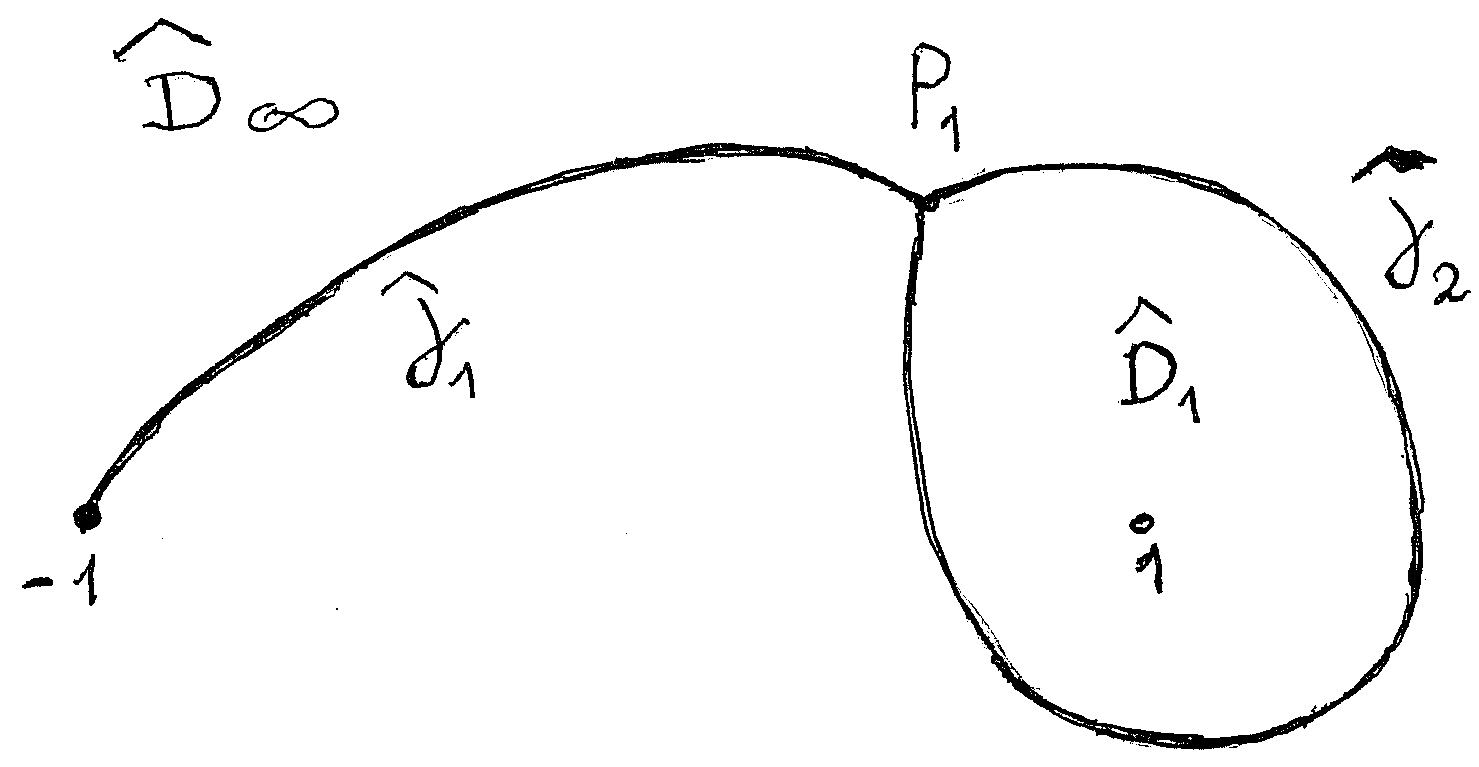} $$
\vspace{-1cm} %
\caption{7a. Proof of Theorem~4: Impossible limit configuration.} 
\end{figure}

\medskip


\begin{figure}
$$\includegraphics[scale=.25,angle=0]{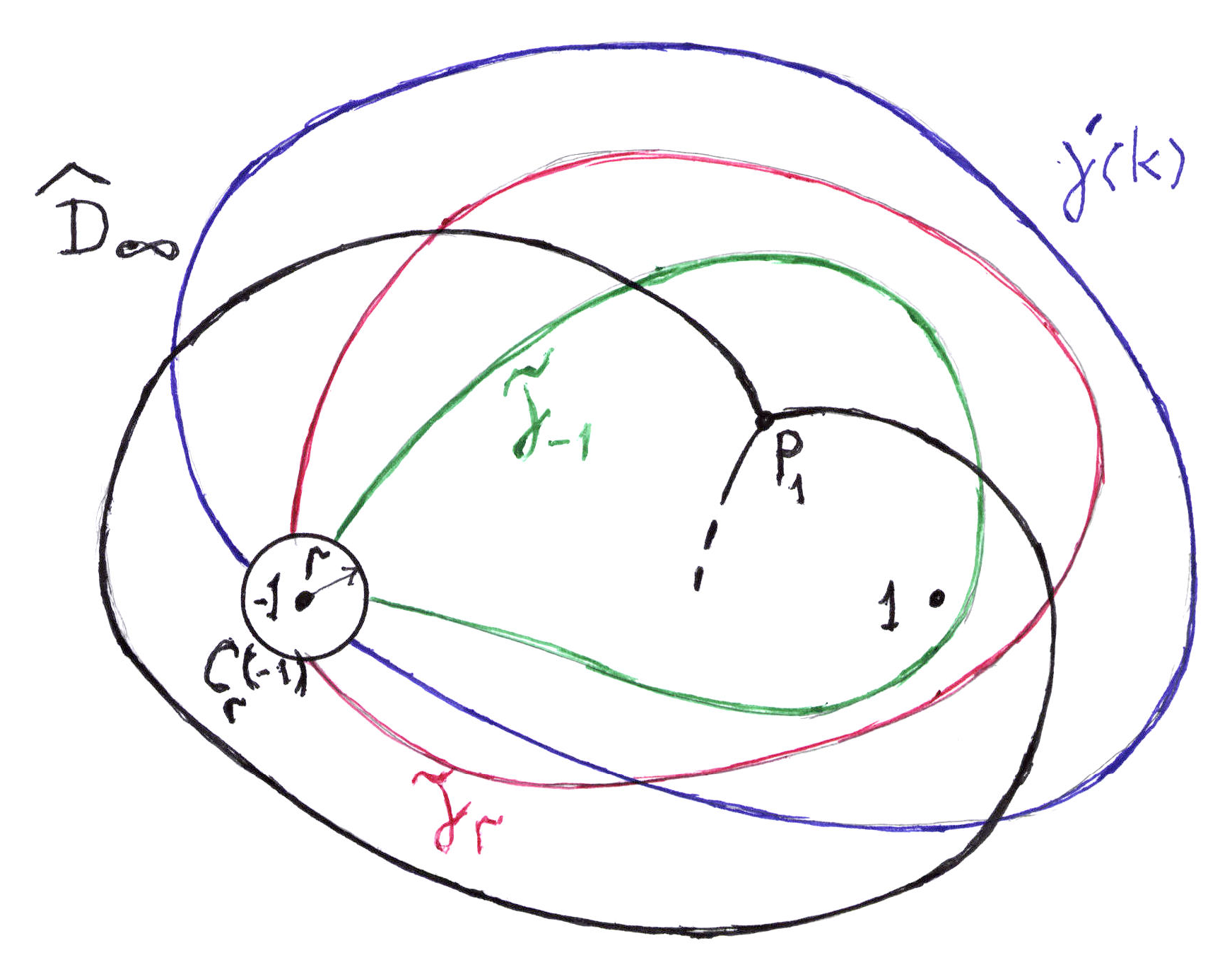}$$
\vspace{-1cm} %
\caption{7b. Proof of Theorem~4: Limit configuration.} 
\end{figure}

\medskip

\vspace{6cm}


\begin{figure}[b]
$$\includegraphics[scale=.25,angle=0]{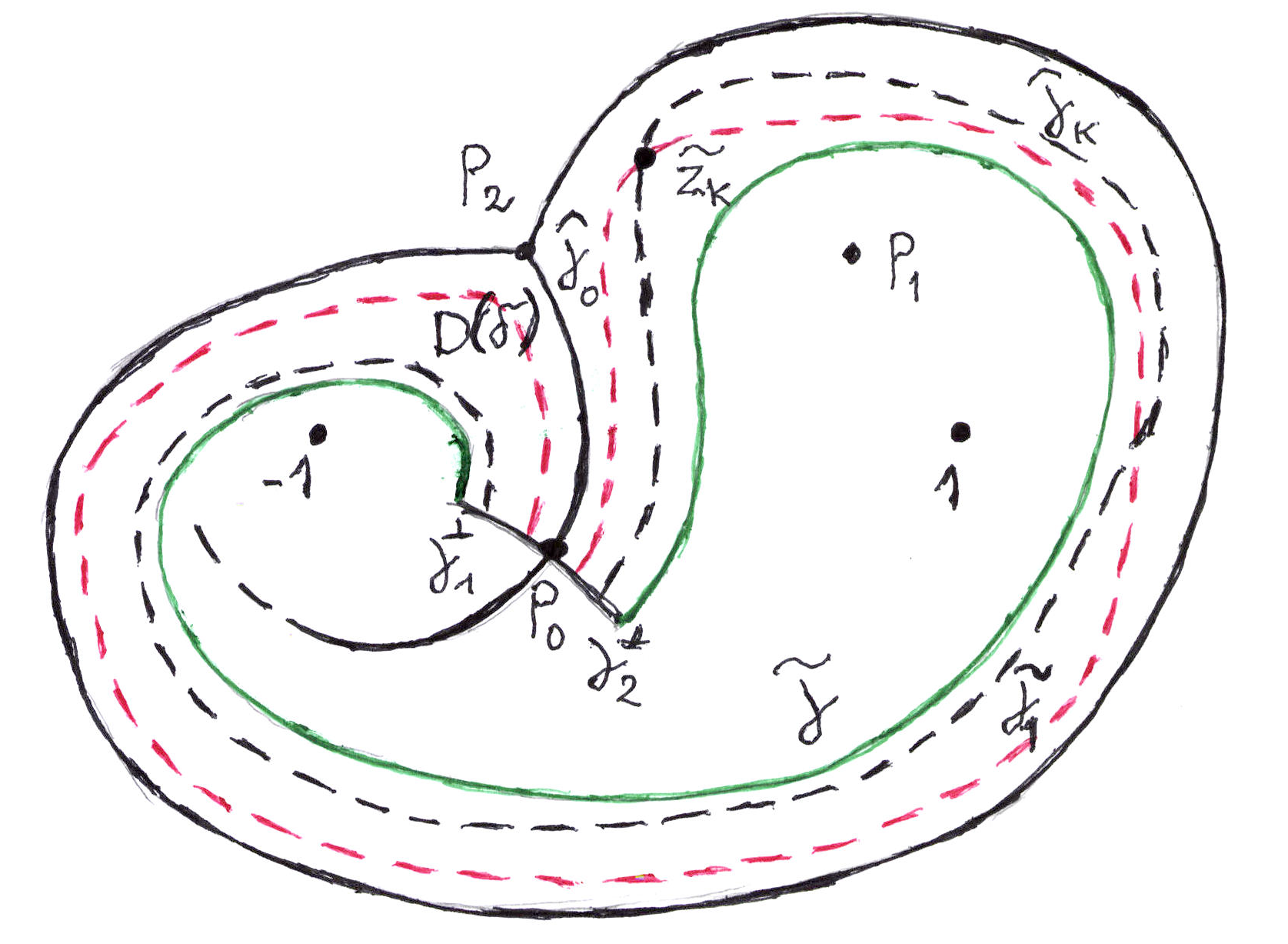}$$
\vspace{-1cm} %
\caption{7c. Proof of Theorem~4: $Q^0$-rectangle $D(\delta)$ with trajectories.} 
\end{figure}

\FloatBarrier

\newpage

 \textbf{VIII.} Geodesics and loops in simple cases.


\begin{figure*}[h]
\centering
\hskip -4cm
\begin{minipage}{0.60\linewidth}  %

  \end{minipage}
\caption{8a. Geodesics and loops. Case \textbf{6.2}.}
\end{figure*}

\medskip

\begin{figure*}[h]
\centering
\hskip -4cm
\begin{minipage}{0.60\linewidth}  %
%
  \end{minipage}
\caption{8b. Geodesics and loops. Case \textbf{6.3(a)}.}
\end{figure*}

\medskip

\begin{figure*}[h]
\centering
\hskip -4cm
\begin{minipage}{0.60\linewidth}  %
%
  \end{minipage}
\caption{8c. Geodesics and loops. Case \textbf{6.3(b1)}.}
\end{figure*}

\medskip

\FloatBarrier

\newpage

\bigskip

\textbf{IX.} Divergent segments.
%


\begin{figure*}[h]
\centering %
%
\caption{9a. Divergent segments.  Case \textbf{6.2}.}
\end{figure*}

\medskip

\begin{figure}[h]
\centering %
\hspace{-4.5cm} %
\begin{minipage}{0.60\linewidth}
%
\end{minipage}
\caption{9b. Divergent segments. Case \textbf{6.3(b2)}.}
\end{figure}

\FloatBarrier

\newpage

\FloatBarrier

\textbf{X.} Geodesics and loops in the most general case.

\FloatBarrier

\begin{figure}[h]
\centering %
\hspace{-4.5cm} %
\begin{minipage}{0.60\linewidth}
%
\end{minipage}

\caption{10a. Critical geodesics and loops. Case
\textbf{6.3(b2)}(a).}
\end{figure}

\medskip

\begin{figure}[h]
\centering %
\hspace{-4.5cm} %
\begin{minipage}{0.60\linewidth}
%
\end{minipage}
\caption{10b. Critical geodesics and loops. Case
\textbf{6.3(b2)}(b).}
\end{figure}

\medskip

\begin{figure}[h]
\centering %
\hspace{-4.5cm} %
\begin{minipage}{0.60\linewidth}
%
\end{minipage}
\caption{10c. Critical geodesics and loops. Case
\textbf{6.3(b2)}(c).}
\end{figure}

\medskip

\begin{figure}[h]
\centering %
\hspace{-4.5cm} %
\begin{minipage}{0.60\linewidth}
%
\end{minipage}
\caption{10d. Critical geodesics and loops. Case
\textbf{6.3(b2)}(d).}
\end{figure}

\medskip

\begin{figure}[h]
\centering %
\hspace{-4.5cm} %
\begin{minipage}{0.60\linewidth}
%
\end{minipage}
\caption{10e. Critical geodesics and loops. Case
\textbf{6.3(b2)}(e).}
\end{figure}

\medskip

\begin{figure}[h]
\centering %
\hspace{-4.5cm} %
\begin{minipage}{0.60\linewidth}
%
\end{minipage}
\caption{10f. Critical geodesics and loops. Case
\textbf{6.3(b2)}(f).}
\end{figure}

\medskip

\FloatBarrier

\newpage

\begin{figure}[h]
\centering %
\hspace{-4.5cm} %
\begin{minipage}{0.60\linewidth}
%
\end{minipage}
\caption{10g. Critical geodesics and loops. Case
\textbf{6.3(b2)}(g).}
\end{figure}

\medskip

\newpage

\begin{figure}
\centering %
\hspace{-4.5cm} %
\begin{minipage}{0.60\linewidth}
%
\end{minipage}
\caption{10h. Critical geodesics and loops. Case
\textbf{6.3(b2)}(h).}
\end{figure}

\medskip

\newpage

\begin{figure}
\centering %
\hspace{-4.5cm} %
\begin{minipage}{0.60\linewidth}
%
\end{minipage}
\caption{10i. Critical geodesics and loops. Case
\textbf{6.3(b2)}(i).}
\end{figure}

\medskip
\FloatBarrier

 \textbf{XI.} Identification rules.
\FloatBarrier

\begin{figure}[h]
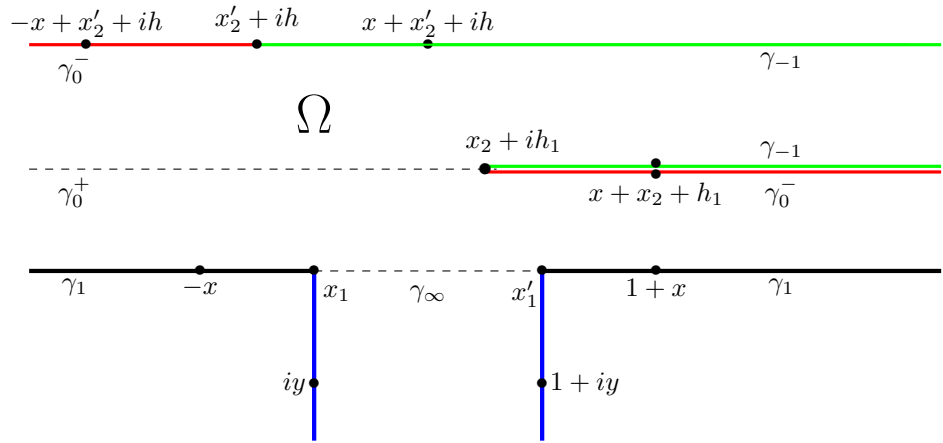

\centering %
\hspace{-4.5cm} %
\begin{minipage}{0.60\linewidth}
%
\end{minipage}

\caption{11. Domain $\Omega$ and identification rules. }
\end{figure}


\FloatBarrier


\end{document}